\documentclass[a4paper, 11pt]{article}
\usepackage [a4paper,left=2.5cm,bottom=2.5cm,right=2.5cm,top=2.5cm]{geometry}

\usepackage[english]{babel}

\usepackage[utf8]{inputenc} 	
\usepackage[T1]{fontenc}
\usepackage{lmodern}
\usepackage{yhmath}
\usepackage{comment}
\usepackage{bm}
\usepackage{enumerate}   

\usepackage{amsthm,amsmath,amssymb,amsbsy,bbm,mathrsfs,supertabular,
eurosym,graphicx,enumitem,xcolor}
\usepackage{mathtools}
\usepackage{ulem}

\newtheoremstyle{BBstyle0}  {}{}{\itshape}{}{\bfseries}{}{6pt}{}
\newtheoremstyle{BBstyle1}  {3pt}{3pt}{\rmfamily}{}{\itshape}{: }{3pt}{}
\newtheoremstyle{BBstyle2}  {3pt}{3pt}{\itshape}{}{\bfseries\large}{}{0pt}{}
\newtheoremstyle{BBstyle3}  {}{}{\itshape}{}{\bfseries}{: }{3pt}{}
\newtheoremstyle{BBstyle4}  {}{}{\rmfamily}{}{\bfseries}{}{6pt}{}
  
\newtheorem{thm}{Theorem}

\newtheorem{lemma}{Lemma}

\newtheorem{prop}{Proposition}
\newtheorem{df}{Definition}
\newtheorem{cor}{Corollary}
\newtheorem{ass}{Assumption}

\newtheorem{rem}{Remark\!\!}

\theoremstyle{definition}

\usepackage[english]{babel}

\usepackage{enumitem}
\usepackage{hyperref}  

\newcommand{\norm}[1]{\left\|{#1}\right\|}

\newcommand{\ab}[1]{\left|{#1}\right|}

\newcommand{\id}[2]{\{#1,\dots,#2\}} 

\newcommand{\argmax}{\mathop{\rm argmax}}

\newcommand{\E}{{\mathbb{E}}}

\newcommand{\N}{{\mathbb{N}}}
\renewcommand{\P}{{\mathbb{P}}}
 
\newcommand{\R}{{\mathbb{R}}}

\newcommand{\cB}{{\mathcal{B}}}
\newcommand{\cC}{{\mathcal{C}}}

\newcommand{\cF}{{\mathcal{F}}}

\newcommand{\cL}{{\mathcal{L}}}

\newcommand{\bs}[1]{\boldsymbol{#1}}

\newcommand{\modar}{\color{black}}
\newcommand{\modarn}{\color{black}}
\newcommand{\modch}{\color{black}}
\newcommand{\modchi}{\color{black}}

\newcommand{\modhel}{\color{black}}
\newcommand{\revar}{\color{black}}
\newcommand{\rev}{\color{black}}

\DeclarePairedDelimiter{\abs}{\lvert}{\rvert}

\title{Evolving privacy: drift parameter estimation for discretely observed i.i.d. diffusion processes under LDP}
\author{Chiara Amorino\thanks{ Universit\'e du Luxembourg, L-4364 Esch-Sur-Alzette, Luxembourg. The author gratefully acknowledges financial support of ERC Consolidator Grant 815703 “STAMFORD: Statistical Methods for High Dimensional Diffusions”.} \qquad Arnaud Gloter \thanks{Laboratoire de Math\'ematiques et Mod\'elisation d'Evry, CNRS, Univ Evry, Universit\'e Paris-Saclay, 91037, Evry, France.} \qquad Hélène Halconruy \thanks{Télécom SudParis, Institut polytechnique de Paris, 91120, Palaiseau, France.}} 

\begin{document}
\maketitle




\begin{abstract}
\noindent
The problem of estimating a parameter in the drift coefficient is addressed for $N$ discretely observed independent and identically distributed stochastic differential equations (SDEs). This is done considering additional constraints, wherein only public data can be published and used for inference. The concept of local differential privacy (LDP) is formally introduced for a system of stochastic differential equations. The objective is to estimate the drift parameter by proposing a contrast function based on a pseudo-likelihood approach. A suitably scaled Laplace noise is incorporated to meet the privacy requirements. Our key findings encompass the derivation of explicit conditions tied to the privacy level. Under these conditions, we establish the consistency and asymptotic normality of the associated estimator. Notably, the convergence rate is intricately linked to the privacy level, and in some situations may be completely different from the case where privacy constraints are ignored.
Our results hold true as the discretization step approaches zero and the number of processes $N$ tends to infinity.
 
\end{abstract}

\tableofcontents

\section{Introduction}

In recent years, the rapid accumulation of large-scale data, including medical records, cellphone location information, and internet browsing history, underscores the critical need for a nuanced understanding of the tradeoffs between privacy and the utility derived from collected data. Traditional privacy-preserving mechanisms like permutation or basic anonymization have proven inadequate, leading to notable privacy breaches with genomic and movie rating data (see \cite{
46Fisher}). {\rev For instance, in 2010, Netflix canceled its second data competition, the Netflix Prize, which was initially launched to improve its movie recommendations, due to privacy breaches. The Federal Trade Commission (FTC) investigated, prompting Netflix to prioritize user privacy over the competition. Other similar incidents highlight the urgency of balancing utility with the protection of sensitive information.}\\

\noindent
To address these challenges, researchers in statistics, databases, and computer science started studying differential privacy as a formalization of disclosure risk limitation (for example in \cite{Martin, 30Fisher, Dwo06}).  Differential privacy has evolved from a theoretical paradigm to a widely deployed technology in various applications over the last decade \cite{24Avella, 29Avella, 35Avella}. It addresses the need for protecting individual data while allowing statistical analysis of aggregate databases. This is achieved through a trusted curator holding individual data, and the protection is ensured by injecting noise into released information. The challenge is to optimize statistical performance while preserving privacy in a remote access query system. \\ 
{The first attempt to design a private methodology traces back to Dwork et al.'s work \cite{Dwo06}, that formalized global differential privacy. Research in this field now encompasses central or global privacy models, where a curator collects and privatizes data before releasing only the output, and local privacy models, involving randomized data collection.} Major technology companies like Apple \cite{5Fisher, 71Avella} and Google \cite{1Fisher, 29Avella} have adopted local differential privacy protections, reflecting the broader impact of privacy measures in billions of devices. Understanding the fundamental limitations and possibilities of learning with local privacy notions is crucial in this context.\\

\noindent
Historically, methods for locally private analysis were primarily focused on estimating parameters within a binomial distribution \cite{37CLDP}. However, recent advancements in research have introduced mechanisms that extend to a broader array of statistical challenges. These encompass hypothesis testing \cite{4CLDP, 29CLDP}, M-estimation \cite{Avella}, robustness \cite{31CLDP}, change point analysis \cite{6CLDP}, mean and median estimation \cite{Martin}, and nonparametric estimation \cite{11CLDP, 10CLDP, Kroll}, among others. With the growing importance of data protection, striking the right balance between statistical utility and privacy becomes crucial. It is imperative to safeguard data against privacy breaches while still facilitating the extraction of valuable insights. As a result, finding the right balance between these two factors has taken on greater importance.

\noindent
To the best of our understanding, the exploration of statistical inference under privacy constraints has traditionally centered on $N$ random variables. The consideration of variables having a historical context, involving $N$ stochastic processes across a time horizon $[0,T]$, has been noticeably absent from prior investigations. This paper aims to address this gap by examining the drift parameter estimation from i.i.d. paths of diffusion processes while adhering to local differential privacy constraints.

\noindent
Note that the literature on statistical estimation for stochastic differential equations is extensive, owing to the model's versatility and its applicability across various domains. Examples abound, spanning biology \cite{46Est}, neurology \cite{29Est}, 
finance \cite{30Est}, and economics \cite{11Est}. Classical applications extend to physics \cite{44Est} and mechanics \cite{36Est}. {\rev The framework mentioned can also be employed to model the inherent variability of biomedical experiments, with a specific emphasis on pharmacology (refer to \cite{Wang,22Wang, 23Wang, 24Wang} and \cite{21Wang} for an overview of stochastic differential equation estimation for pharmacodynamic models)}. This is why focusing on diffusion processes becomes interesting when dealing with sensitive data that requires privacy guarantees. However, a significant challenge arises when applying privacy in this context due to the dependency structure of the process. Indeed, it becomes possible to recover information about the private process at a particular time instant through observations about its past or future.\\

\noindent
In this work we consider $N$ i.i.d. diffusion processes $(X_t^{\theta,i})_{t\in[0,T]}$ for $i \in \{1, \dots , N \}$, with a fixed time horizon $T > 0$. They follow the dynamics
\begin{equation*}
X_t^{\theta, i}=X_0^{\theta, i}+\int_0^tb(\theta,X_s^{\theta, i})ds+\int_0^t {\modch \sigma}(X_s^{\theta,i})dW_s^{i},
\end{equation*}
where the processes $(W^i_t)_{t \in [0, T]}$ are independent standard Brownian motions, and are also independent of the i.i.d. random variables $(X_0^{\theta, i})_i$.
We aim to estimate the true parameter value $\theta^{\star} \in \Theta$ based on the discrete observations $(X_{t_{j}}^{\theta^\star,i})_{i=1, \ldots, N, j=1, \ldots, n}$ under $\alpha$ local differential privacy constraints, for $N, n \rightarrow \infty$. For simplicity, we will denote $X_{t_{j}}^{i}$ when the process is computed in the true value of the parameter $\theta^{\star}$, that is $X_{t_{j}}^{\theta^{\star},i}$.
In our context, the discrete observations of the private processes $(X_t^1)_t, \ldots, (X_t^N)_t$ are not directly available. Instead, our estimator relies on a public sample derived from the original discrete observations, subject to $\alpha$ local differential privacy. To achieve this, we begin by adapting the concept of local differential privacy to a context that accommodates the presence of processes, leveraging the definition of \textit{componentwise local differential privacy} (CLDP) proposed in \cite{CLDP} (refer to Section \ref{s: formulation privacy} for comprehensive details). In our interpretation, the $N$ processes represent $N$ independent individuals evolving over time. Each individual contributes $n$ observations, subject to componentwise local differential privacy, ensuring that the information at different time points is privatized in a different way. This approach enables us to make public each observation corresponding to the same individual separately, using a distinct privacy channel with a privacy level $\alpha_j$, for $j \in \{1, ... , n \}$. This strategy is advantageous as it allows for tailoring the level of privacy protection based on the specific context, recognizing that different observations corresponding to various moments in an individual's life may warrant varying degrees of privacy safeguards.
\\
The rationale behind this choice lies in the recognition that disclosing information about the distant past may have different implications than disclosing information about the present. By treating each observation independently, we can tailor the level of privacy protection based on the specific circumstances.
\\
In our study, we examine a privacy mechanism where the data holder can observe two consecutive realizations of $X^i$ each time. This scenario models situations such as when a patient's vitals are taken, and the doctor (data holder) has access not only to the current data but also to the data from the previous check-up (control data).\\
Given the correlated nature of the observations, it becomes crucial to operate within a framework that acknowledges this dependency. A pertinent question arises: Is it possible to extract sensitive information about the present by exploiting the fact that some extra information is carried by another observation from the past? \\
To answer this question, it is important to understand the constraints on the privacy levels $\alpha_1, ... , \alpha_n$ necessary to obtain well-performing estimators. Specifically, we will explore how the behavior of $\alpha_1, ... , \alpha_n$ as functions of $N$ and $n$ leads to two distinct asymptotic regimes, which we will refer to as the "significant contribution of privacy" and "negligible contribution of privacy".
\\
Without the presence of privacy constraints, a natural approach to estimating unknown parameters from the continuous observation of a SDE would be to use maximum likelihood estimation. However, the likelihood function based on the discrete sample is not tractable, as it depends on the transition densities of the process, which are not explicitly known. To overcome this difficulty, several methods have been developed for high-frequency estimation of discretely observed classical SDEs. A widely-used method involves considering a pseudo-likelihood function, often based on the high-frequency approximation of the process dynamics using the Euler scheme, as seen in \cite{25McK, 42McK, 60McK}.\\
This statistical analysis relies on the minimization of a contrast function, akin to methods proposed for classical SDEs, extended to L\'evy-driven SDEs \cite{Shimizu, Contrast} and interacting particle systems \cite{Imperial, McKean}.
\\
Even in our context, with the presence of local differential privacy constraints, it seems natural to leverage the minimization of a contrast function technique, incorporating this quantity into the definition of the privacy mechanism employed in our estimation procedure. Furthermore, introducing centered Laplace-distributed noise to bounded random variables is known to ensure ${\alpha}$-differential privacy (see \cite{CLDP}, \cite{Martin}, \cite{Kroll}). This motivates our choice of the anonymization procedure. Specifically, we adopt a Laplace mechanism to construct the public counterpart of the raw samples, as illustrated in \eqref{eq: def Z} below. Some technical challenges arise from such a definition, mainly due to the fact that Laplace random variables are defined only for some values of $\theta$ on a grid, whose size $L_n$ goes to $\infty$ for $n \rightarrow \infty$. Consequently, our definition of the public sample $(Z^i_j(\theta))_{i =1, \dots, N,\, j=1, \dots, n }$ holds true only on the grid as well. However, to define our estimator, we need to extend the definition of the public sample to $\theta$ belonging to the whole parameter space $\Theta$. To address this, we rely on a spline approximation method, detailed in Section \ref{s: splines}. Then, with access to the public data $Z_j^i(\theta)$ for $i=1, \dots , N$ and $j=1, \dots , n$, the statistician can propose an estimator obtained by minimizing the contrast function, which is the spline approximation of the private version of the contrast function in the case of classical SDEs (see Section \ref{Sec_results} for details).\\

\noindent
The main result of the paper is the consistency and asymptotic normality of the resulting estimator, demonstrated using a central limit theorem for martingale difference triangular arrays.

\noindent{Let us introduce 
\begin{equation}\label{Eq_def_rnN}
r_{n,N} := \frac{L_n^2 \log(n)}{\sqrt{\bar{\alpha}_2}},
\end{equation}
where $\bar{\alpha}_2$  is the harmonic mean over squared different privacy levels, defined as $1/\bar{\alpha}_2 := n^{-1} \sum_{j = 1}^n 1/\alpha_j^2$ {\modhel and that may be dependent on $N$.}
We establish the consistency of the proposed estimator under the assumption that 
$$\frac{1}{L_n} \sqrt{\frac{\log(L_n)}{N}} r_{n,N} \rightarrow 0\quad \text{as}\; n, N, L_n \rightarrow \infty.$$}
 This requirement implies that the case of perfect privacy ($\bar{\alpha} = 0$) is not allowed and informs us about the price to pay for the privacy guarantee in obtaining reasonable statistical results.

\noindent
Moreover, as anticipated earlier, we prove the asymptotic normality of our estimator under the two regimes delineated by different values of $\bar{\alpha}$. In particular, we find that if the privacy parameters $\alpha_1, ..., \alpha_n$ are large enough to guarantee 
{\modch $r_{n,N} \sqrt{\log(L_n)} \rightarrow 0$}, 
 with some technical conditions, 
 we essentially obtain the same result as in the case without privacy:
$$\sqrt{N} (\widehat\theta_n^N - \theta^{\star}) \xrightarrow{\mathcal{L}} {\mathcal N}\bigg(0,2 \Big(\int_0^T \E\bigg[\Big(\frac{\partial_\theta b(\theta^{\star}, X_s)}{{\modch \sigma}(X_s)}\Big)^2\bigg] ds \Big)^{-1}\bigg) = : {\mathcal N}\big(0,2( \Sigma_0)^{-1}\big)  \quad \mbox{as } n, N\rightarrow\infty.$$

\noindent
Indeed, the convergence above asserts the asymptotic Gaussianity of our estimator with a convergence rate and a variance that resemble the classic scenario without any privacy constraints.\\
When the contribution of privacy is instead the dominant one (i.e. 
{\modch $r_{n,N} \rightarrow \infty$,} {with $r_{n,N}$ defined by \eqref{Eq_def_rnN}}), we still recover the asymptotic normality of our estimator, but with a different convergence rate, depending on the average amount of privacy 
$\bar{\alpha}_2$. In this context, some extra challenges appear, leading us to replace the previous grid with a random one, depending on a uniform random variable $S$ (see Equation \eqref{eq: random grid}). Then, subject to some technical conditions we are able to prove the following:
$$\frac{\sqrt{N \bar{\alpha}_2}}{ 4(a + 1) L_n^2 \log(n) \sqrt{T }} (\widehat\theta_n^N - \theta^{\star})\xrightarrow{\mathcal{L}} \sqrt{\overline{v}(s)} (\Sigma_0)^{-1} \, {\mathcal N},$$
 where ${\mathcal N}$ is a gaussian $\mathcal N(0,1)$ random variable independent of $S$, $a$ and $\overline{v}$ are respectively a tuning parameter and an auxiliary function properly defined in \eqref{E: def hat v}; both depending on the spline functions under consideration.\\
{\modch It is noteworthy that our results distinctly reveal $L_n^2 \log(n)$ as the threshold for $\sqrt{\bar{\alpha}_2}$, thereby delineating whether the newly introduced term arising from the privacy constraints is a primary contributing factor. The significance of this threshold becomes apparent through our findings. Theorem \ref{th: as norm privacy negl} indicates that the case of 'negligible' privacy corresponds to the constraint $r_{n,N} \sqrt{\log(L_n)} \rightarrow 0$, while Theorem \ref{th: as norm} establishes the condition $r_{n,N} \rightarrow \infty$ for achieving 'significant' privacy. In Corollary \ref{cor: threshold}, we investigate the scenario where $r_{n,N}$ converges to a constant. In this case, we are able to demonstrate the convergence in law of our estimator to the sum of the two Gaussian random variables obtained in Theorems \ref{th: as norm} and \ref{th: as norm privacy negl}, respectively. \\
This emphasizes that the threshold demarcating the two asymptotic regimes of 'significant' and 'negligible' privacy is dictated by the asymptotic behavior of $r_{n,N}$. The additional $\log(L_n)$ in the condition outlined in Theorem \ref{th: as norm privacy negl} is introduced for technical reasons associated with the grid, representing a non-optimal condition necessary to mitigate the impact of privacy constraints.} {\rev Additionally, we present an example where the drift is polynomial in \(\theta\), motivated by the fact that, in this case, the spline approximation is easier to handle. This allows us to concentrate on the impact of privacy. We also include a discussion on effective privacy and its influence on the convergence rates.} \\
 \\
The paper is organized as follows. In Section \ref{Sec_setting}, we introduce the model and its underlying assumptions. Notably, we formulate the definition of local differential privacy tailored to stochastic processes in Section \ref{s: formulation privacy}. Section \ref{Sec_results} is dedicated to presenting the contrast function pivotal for our estimator's definition, alongside articulating our main results, whose strengths and weaknesses we discuss in Section \ref{s: discussion}. Turning to Section \ref{Sec_tools}, we furnish essential tools essential for proving our main results. This includes an exploration of spline functions in Section \ref{s: splines}, followed by technical results outlined in Section \ref{s: technical}. In summary, the proofs of our main and technical results find their place in Sections \ref{s: proof main} and \ref{s: proof preliminary}, respectively.

\color{black}\section{Mathematical framework}\label{Sec_setting}

\subsection{Setting and assumptions}

\color{black}

Let $T>0$. Let $W^1,\dots,W^N$ be $N\in\N^* = \{1, 2, 3, ... \}$ independent standard Brownian motions. For any $i\in\{1,\dots,N\}$, we consider the diffusion process $(X_t^{\theta,i})_{t\in[0,T]}$ defined by
\begin{equation}\label{Eq_EDS}
X_t^{\theta, i}=X_0^{\theta, i}+\int_0^tb(\theta,X_s^{\theta, i})ds+\int_0^t{\modch \sigma}(X_s^{\theta,i})dW_s^{i},
\end{equation}
where $b:\Theta\times\R\rightarrow\R$, ${\modch \sigma}:\R\rightarrow\R$ and $\theta\in\Theta$. {\modch We fix $\Theta:=[0,1]$ to simplify the notation}. Assume that for any $i\in\{1,\dots,N\}$, the processes $(W^i_t)_{t \in [0, T]}$ are independent of the initial value $(X_0^{\theta,1}, \dots , X_0^{\theta, N})$. We also assume that $(X_0^{\theta, i})_i$ are i.i.d with $X_0^{\theta, i} \in \cap_{p \ge 1} L^p$.\\
As anticipated in the introduction, we aim at estimating the parameter {\modch $\theta^{\star} \in \overset{\circ}{\Theta}=(0,1)$} given the observations $(X_{t_{j}}^{\theta,i})_{i=1, \dots , N,\, j=1, \dots , n}$ subject to local differential privacy constraints, see next section for a formal definition and a rigorous introduction of the problem.  \\
We introduce the discretization step as $\Delta_n:= T/n$, so that $t_{j,n} = jT/n = j \Delta_n$. The asymptotic framework here considered is such that both $N, n \rightarrow \infty$ while the time horizon $T$ is fixed. {\modchi Moreover, $N$ goes to $\infty$ as a polynomial of $n$, i.e. there exists $r > 0$ such that $N = O(n^r)$.} \\
\\
Let us consider the assumptions : 
\begin{ass}\label{Ass_1} For all $\theta\in\Theta$, the function $b(\theta,\cdot)$ is bounded. Moreover, $b(\theta,\cdot)$ and ${\modch \sigma}$ are globally Lipschitz, i.e  there exists $c > 0$ such that, for all $x, y \in \R$, 
$$|b(\theta, x) - b(\theta, y)| + |{\modch \sigma}(x) - {\modch \sigma}(y)| \le c |x-y|.$$
\end{ass}

\begin{ass}\label{Ass_2}
We assume that the diffusion coefficient is bounded away from $0$: for some ${\modch \sigma_{\min}}>0$,
\begin{equation*}
{\modch \sigma_{\min}}^2\le {\modch \sigma^2}(x).\\
\end{equation*}
\end{ass}

\noindent
Under assumptions \ref{Ass_1} and \ref{Ass_2}, Equation \eqref{Eq_EDS} has a unique strong solution $(X_t^{\theta, i})_{t\in[0,T]}$ taking its values in $(\R,\cB(\R))$.

\begin{ass}\label{Ass_3}
For all $k \in \mathbb{N}^*= \{1, 2, ... \},$ the function $\frac{\partial^k}{\partial \theta^{k}} b$ is bounded uniformly in $\theta$: $\sup_{x, \theta}| \frac{\partial^k}{\partial \theta^{k}} b(\theta, x)| < \infty$.
\end{ass}

%
%

\begin{ass}\label{Ass_4}[Identifiability] \\
For all $\theta\in \Theta$ such that $\theta \neq \theta^\star$, it is 
\[\int_0^T\E\bigg[\frac{\big(b(\theta,X_s)-b(\theta^\star,X_s)\big)^2}{{\modch \sigma^2}(X_s)}\bigg]ds > 0.\] 
\end{ass}
\noindent One can easily check that Assumption \ref{Ass_4} is equivalent to ask that, for any $\theta \in \Theta$ such that $\theta \neq \theta^\star$, it is $b(\theta, \cdot) \neq b(\theta^\star, \cdot)$ almost surely. This will be crucial in order to prove the consistency of the estimator we will propose. \\
Observe that, in the sequel, we will often shorten the notation $\partial_\theta^k g$ for $\frac{\partial^k}{\partial \theta^{k}} g$, for any derivable function $g$.

\begin{ass}\label{Ass_5}[Invertibility] \\
Define $I(\theta) := \int_0^T\E\big[( \frac{\partial_\theta b(\theta, X_s)}{{\modch \sigma}(X_s)} )^2 \big]ds$. We assume that, for any $\theta \in \Theta$, $I(\theta) > 0$.
\end{ass}
\noindent We will see that Assumption \ref{Ass_5} will be essential to prove the asymptotic normality of our estimator.

{\revar
\subsection{Problem formulation}{\label{s: formulation privacy}}
%
%
\subsubsection{The local differential privacy formalism}
Let $({X}_t^1), \ldots, ({X}_t^N)$ be solutions of the stochastic differential equation \eqref{Eq_EDS}. Since they are driven by independent Brownian motions, we can regard them as $N$ independent realizations of the same diffusion process. It is important to note that, in general, estimating the drift in SDEs driven by a diffusion is not feasible over a finite time horizon when relying solely on discrete observations from a single path of the solution. However, in this context, the number of copies $N$ will serve a similar purpose as the time horizon $T$ does in parameter estimation for SDEs driven by a diffusion. As $N$ tends to infinity, the sample size will effectively increase, allowing for our estimation procedure. 
The $N$ copies follow the law of the same stochastic process $(X_t)_{t\in [0,T]}$, and we observe each of them in $n + 1$ different instants of time $0=t_0 \le t_1 \le \dots \le t_n = T$. 
They can represent the information coming from $N$ different individuals, that evolve in time. \\

\noindent
In our approach, we adopt componentwise local differential privacy as introduced in \cite{CLDP}. {It means that we do not release the public data pertaining to each individual based on the whole vector of its private data, 
but instead release data relying on componentwise observation of this private vector. As shown in \cite{CLDP}, this constraint usually makes harder to infer the joint law of the private data, but it might be more suitable in practice when dealing with temporal series.} 


\noindent
Let us now formalize the framework discussed earlier. We introduce $\bm{X}^i := (X^i_{t_0}, \dots , X^i_{t_n})$ for any $i \in \{1, \dots , N \}$. The process of privatizing the raw samples $(\bm{X}^i)_{i=1,\ldots,N}$ and transforming them into the public set of samples $(\bm{Z}^i)_{i=1,\ldots,N}$ is captured by a conditional distribution, known as privacy mechanism or channel distribution.
We make the assumption that each component of a disclosed observation, denoted by $Z^i_{j}$, is privatized independently and belongs to a specific space $\mathcal{Z}^i$, which may vary for each $i$. This implies that the observation $\bm{Z}^i$ belongs to the product space $\bm{\mathcal{Z}}:=\prod_{j=1}^n \mathcal{Z}_j$.
Additionally, we assume that the spaces $\mathcal{Z}_j$ are separable complete metric spaces. Their corresponding Borel sigma-fields define measurable spaces $(\mathcal{Z}_j,\Xi_{\mathcal{Z}_j})$ for all $j\in \{
0,\dots, n\}$.\\
\noindent
We now explore the properties and structure of the privacy mechanism within our framework. For simplicity, the privacy mechanism is designed to be non interactive between the $N$ individuals. This gives rise to the following independence structure : for $ j\in\{1,\dots,n\}$, 
	\begin{equation}\label{Eq: PM structure}
		\{ {X}^i_{t_j}, {X}^i_{t_{j-1}}\} \rightarrow Z^i_{j}, \qquad Z^i_{j} \perp \!\!\! \perp X^k_{t_l}  \, \text{ for } k \neq i, \, \forall l \in \{0,\dots,n\}.
	\end{equation}}
{\rev	In the non-interactive case we are considering, \eqref{Eq: PM structure} means that for $j=1,\dots, n$ and $i= 1, \dots, N$, given $X^i_{t_j} = x^i_j\in\R$ and $X^i_{t_{j-1}} = x^i_{j-1}\in\R$,
	the public output $Z^i_{j}\in\mathcal{Z}_j$ is drawn as
\begin{equation}
	{\label{eq: def Z}}
	Z^i_{j} \sim Q_j (\cdot | X^i_{t_j} = x^i_j, X^i_{t_{j-1}} = x^i_{j-1})
\end{equation}
for Markov kernels $Q_j: \Xi_{\mathcal{Z}_j} \times (\R \times \R ) \rightarrow [0, 1]$.
The notation $(\bm{\mathcal{Z}}, \Xi_{\bm{\mathcal{Z}}})=(\prod_{j=1}^n \mathcal{Z}_j,\otimes_{j=1}^n \Xi_{\mathcal{Z}_j})$ refers to the measurable space of non-private (or public) data. 
The space of public data, denoted by $\bm{\mathcal{Z}}$, can be quite general, as it is selected by the statistician based on a specific privatization mechanism. Nonetheless, in the parameter estimation discussed below, it will be valued in $\bm{\mathcal{Z}}=\mathbb{R}^{d_Z \times n}$ for some dimension $d_Z$ given in {Section \ref{Sec_results}}. \\

\noindent
The notion of privacy can be quantified using the concept of local differential privacy. Let $\bm{\alpha} = (\alpha_1, \ldots , \alpha_n)$ be a given parameter, where $\alpha_j \ge 0$ for each $j \in \{1, \ldots , n \}$. We say that the random variable $Z^i_{j}$ is an $\alpha_j$-differentially locally privatized view of $X^i_{t_j}$ and $X^i_{t_{j-1}}$ if, for all $x_j, x_j' \in \R$, and $x_{j-1}, x_{j-1}' \in \R$, the following condition holds:
\begin{equation}{\label{eq: def local privacy}}
	\sup_{A \in  \Xi_{\mathcal{Z}_j}}
	\frac{Q_j(A |X^i_{t_j} = x_j, X^i_{t_{j-1}} = x_{j-1} )}{Q_j(A |X^i_{t_j} = x'_j, X^i_{t_{j-1}} = x'_{j-1})} \le \exp(\alpha_j).
\end{equation}
\noindent
We define the privacy mechanism $\bm{Q}=(Q_1,\ldots,Q_n)$
	to be $\bm{\alpha}$-differentially locally private if each variable $Z^i_{j}$ satisfies the condition of being an $\alpha_j$-differentially locally privatized view of $X^i_{t_j}$ and $X^i_{t_{j-1}}$. The parameter $\alpha_j$ serves as a measure of the level of privacy guaranteed to the variables $X^i_{t_j}$ and $X^i_{t_{j-1}}$. By setting $\alpha_j = 0$, we ensure perfect privacy, meaning that it is impossible to recover these variables from the perspective of $Z^i_{j}$.
On the other hand, as $\alpha_j$ tends to infinity, the privacy restrictions become less stringent. \\

\subsubsection{Modelling issues}

The privatization structure described in the previous section 
highlights 
a major difference
compared to the definition of componentwise local differential privacy (CLDP) introduced in \cite{CLDP}. 
In our case, the variable $Z^i_{j}$ is derived not only from the private value of $X^i_{t_j}$, but also from the previous observation at time $t_{j-1}$. This allows for incorporating additional temporal information in the privatization process. \\
\noindent
The description of the privacy mechanism is tied to modelling issues. Indeed, if in practice the data holder of the private data could access the entire time series $\bm{X}^i=(X^i_{t_j})_{j=0,\dots,n}$, it would be statistically better to output one public value $Z^{i}$ based on the whole time series for each individual. The formal definition of the $\alpha$-LDP constraint \eqref{eq: def local privacy} would be modified in this situation, by conditioning on the whole time series instead of two consecutive values. Our privacy mechanism, based on two consecutive data, is in practice more flexible than a privatization based on the whole vector $\bm{X}^i$. It enables for instance a statistician accessing to the record of just two consecutive private values of an individual to output the public data.\\
 \noindent
 However, this description excludes the situation where the private data can be accessed only for a single date before releasing the public value. In this case the value
 $Z^{i}_j$ would be a randomization of $X^i_{t_j}$ only.  Such privacy mechanism would be necessary if the private data has to be destroyed promptly after being collected, in particular making impossible a public randomization based on two temporally distinct data.
By \cite{CLDP}, we know that a one-component based randomization of a vector $\bm{X}^i$ increases the error of estimation for the joint law of $\bm{X}^i$. In turn, it certainly would deteriorate the quality of estimation of the drift parameter which is a main feature of the joint law of two consecutive data. For this reason, we exclude it from our analysis.
\vspace{5pt}\\
Let us mention that a more general concept of interactive privacy mechanism could be considered. In that case, the output $Z_j^i$ would be constructed on the basis of some private data and all the public data already available.  
Although it is typically easier to work with non-interactive algorithms, 
as they yield independent and identically distributed privatized samples, 
there are situations where it is advantageous for the channel's output to depend on previous computations. Stochastic approximation schemes, for instance, require this kind of dependency (see \cite{49Martin}).  In our framework the temporal aspect of the private data $X_j^i$ is along the index $j$, whereas the usual situation of interactive mechanism is to construct $\bm{Z}^i$ inductively on the index $i$ corresponding to the individuals. Thus, to cope with real situation the definition of interactive mechanism should be modified accordingly in our case. For this reason, we do not pursue the discussion here about interactive mechanism.
\\

\subsubsection{Effective privacy} \label{Ss: effective privacy v2}

When considering local differential privacy, a natural question arises about the likelihood of recovering private data from observing public data. Specifically, the definition of $\alpha_j$-LDP given in \eqref{eq: def local privacy} tells us how well the values of $X^i_{t_j}$ and $X^i_{t_{j-1}}$ are protected when $Z^i_j$ is observed. As shown in \cite{WassermanZhou10}, even if someone gains access to private values $(x_{j-1}, x_j)$ and $(x'_{j-1}, x'_j)$, they cannot reliably determine which pair corresponds to a given public observation $Z^i_j$. Any attempt to make such a distinction would result in an error with a probability of at least $(1 + e^{\alpha_j})^{-1}$.
\\
\noindent
However, the public information available related to the individual $i$ is $(Z^i_j)_{1\le j \le n}$, and 	
some information about $X^i_{t_j}$ and $X^i_{t_{j-1}}$ could be also conveyed by $Z^i_k$, for $k \neq j$. Therefore, it is worth understanding how precisely the values of $X^i_{t_j}$ and $X^i_{t_{j-1}}$ could be revealed by the observations of $Z^i_1, \dots, Z^i_n$, which are publicly available. \\
Similar questions have been explored, for example, in \cite{WassermanZhou10} and in Section 4.1 in \cite{CLDP}. In particular, it is well understood that if a vector is privatized with an independent channel for each component and the components are independent, {then no information on a private component is carried by the public views of the others.}
	\\
The situation becomes more intricate if the components of the vector $(X^i_{t_0}, \dots, X^i_{t_n})$ are dependent, as the observation of $Z^i_1, \dots, Z^i_n$ imparts extra information on $X^i_{t_j}$ and $X^i_{t_{j-1}}$. This is evident in the present case, where $X^i_{t_0}, \dots, X^i_{t_n}$ represent the evolution of a single individual, making the components dependent. In Section 4.1 of \cite{CLDP}, we quantify this effect by evaluating how the low dependence between the components of a private vector reduces the privacy loss of $X^i_j$ revealed by public data $Z^i_l$, for $l \neq j$.
From the infill asymptotic $n\Delta_n =T $ with fixed $T$, we expect that the dependence between the components of the time series $\bm{X}^i$ are strong.
Thus, the knowledge of the whole public data  $(Z^i_l)_{l=0,\dots,n}$ reveals much more information on $X^i_{t_j}$ and $X^i_{t_{j-1}}$ than the unique value $Z^i_j$. 
Let us define the kernel $\bm{\overline{Q}}$ by $\bm{\overline{Q}}(A_1\times\dots\times A_n \mid X^i_{t_j}=x_j, X^i_{t_{j-1}}=x_{j-1}):=\mathbb{P}(Z^i_1\in A_1,\dots,Z^i_n\in A_n  \mid X^i_{t_j}=x^i_j, X^i_{t_{j-1}}=x^i_{j-1})$ for $(A_1,\dots,A_n) \in \prod_{l=1}^n \Xi_{\mathcal{Z}_l}$. The kernel $\bm{\overline{Q}}$ is the law of the whole vector of public data containing information about  $(X^i_{t_j}, X^i_{t_{j-1}})$. This kernel satisfies the LDP constraint
\begin{equation}\label{Eq : LDP overline bm Q}
	\sup_{\bm{A}\in \otimes_{j=1}^n \Xi_{\mathcal{Z}_j}}
	\frac{	\bm{\overline{Q}}(\bm{A} \mid X^i_{t_j}=x'_j, X^i_{t_{j-1}}=x'_{j-1})}{	\bm{\overline{Q}}(\bm{A} \mid X^i_{t_j}=x_j, X^i_{t_{j-1}}=x_{j-1})}
	\le \exp\left(\sum_{l=1}^n \alpha_l\right).
\end{equation}
A proof of  \eqref{Eq : LDP overline bm Q} is given in Section \ref{Ss: proof LDP overline bm Q}. 
This equation provides an upper bound on how the privacy of ($X^i_{t_j}, X^i_{t_{j-1}}$) is affected by observing the entire public dataset $Z^i_1, \dots, Z^i_n$. Specifically, the upper bound on the effective level of privacy is given by $\alpha_{\text{eff}} := \sum_{j=1}^n \alpha_j$. Consequently, the lower bound for the minimal error in estimating the true values of $X^i_{t_j}$ and $X^i_{t_{j-1}}$ is reduced to $\left(1 + \exp(\alpha_{\text{eff}})\right)^{-1}$. We conclude that effective privacy is ensured as long as $\alpha_{\text{eff}} = \sum_{j=1}^n \alpha_j = O(1)$.
\\

\noindent
In the next section, we introduce a privacy mechanism that forms the basis of our estimation procedure. First, we will show that it meets the $\bm{\alpha}$-local differential privacy condition, as defined in Equation \eqref{eq: def local privacy}. Then, we will demonstrate that the estimator we propose, based solely on observations of the privatized views $Z^i_{j}$ for $i = 1, \dots, N$ and $j = 1,\dots, n$, is consistent and asymptotically normal, with the convergence rate depending on the privacy level $\bm{\alpha}$, in the case of significant privacy.}

\color{black}
\section{Statistical procedure and main results}\label{Sec_results}
\subsection{Privatization mechanism and estimation procedure} \label{S: Privatization mechanism}

\color{black}
We assume that the functions $b$ and $\sigma$ are known and we aim at estimating the unknown parameter $\theta^\star$ under $\bm{\alpha}$-local differential privacy, as properly formalized in the previous section. Hence, we need to define a public sample $(Z_{j}^{i})_{1\leq i\leq N,1\leq j\leq n}$ obtained from observations $(X_{t_j}^{\theta,i})_{1\leq i\leq N,0\leq j\leq n}$ of the initial data via a privatization mechanism that satisfies the condition in \eqref{eq: def local privacy}. \\
\\
Consider the classical scenario in statistics, where the goal is to estimate the parameter $\textcolor{black}{\theta^\star}$ based on continuous observations of the stochastic differential equation given in \eqref{Eq_EDS}. It is well-known that the maximum likelihood estimator (MLE) performs optimally in this case, being consistent and asymptotically Gaussian with an optimal variance. However, when privacy is not a concern and only discrete observations of the equation are available, the transition density (and hence the likelihood) of the process is generally no longer accessible. To address this challenge, a commonly used approach in the literature is to employ a contrast function that serves as a substitute for the likelihood. This contrast function approximates the likelihood based on the Euler approximation scheme. Specifically, \textcolor{black}{for any $\theta\in\Theta$}, it takes the form :
$$\sum_{i=1}^N\sum_{j=1}^n\frac{(X_{t_j}^{i}-X_{t_{j-1}}^{i}-\Delta_n b(\theta,X_{t_{j-1}}^{i}))^2}{{\modch \sigma}^{2}(X_{t_{j-1}}^{i})}.$$
The proposed estimator, denoted as $\hat{\theta}_n^N$, minimizes the above quantity over the parameter set $\Theta$. It can be verified that minimizing the aforementioned quantity is equivalent to maximizing the following expression:
\begin{equation}\label{Eq_Contrast_function}
\sum_{i=1}^N\sum_{j=1}^n\frac{2b(\theta,X_{t_{j-1}}^{i})(X_{t_j}^{i}-X_{t_{j-1}}^{i})-\Delta_n b^2(\theta,X_{t_{j-1}}^{i})}{{\modch \sigma}^{2}(X_{t_{j-1}}^{i})}=:\sum_{i=1}^N\sum_{j=1}^nf(\theta;X_{t_{j-1}}^{i},X_{t_j}^{i}).
\end{equation}
Hence, it is natural to incorporate the above quantity into the definition of the privacy mechanism employed in our estimation procedure. \\
Furthermore, as said in the introduction, it is well-known that introducing centered Laplace-distributed noise to bounded random variables ensures $\bm{\alpha}$-differential privacy. Laplace random variables will therefore play a role in defining the privacy mechanism. \vspace{3pt}\\
%
\noindent
We commence by establishing a grid {\modar of the parameter space $\Theta=[0,1]$ 
upon which we will construct the Laplace random variables. Let us denote by $\Xi=\{0\le \theta_0<\cdots<\theta_{L_n-1}\le 1\}$ this grid. It has cardinality $L_n\in\N^{*}$. For $n \rightarrow \infty$, we assume $L_n$ to go to $\infty$ {\modchi with the restriction that there exists $r > 0$ such that $L_n=O(n^r)$.} 
} 


\noindent In order to define the privatized views of our data we need to introduce a smooth version of the indicator function, that we denote as $\varphi$. It is such that $\varphi(\xi) = 0$ for $|\xi| \ge 2$, $\varphi(\xi) = 1$ for $|\xi| \le 1$ and, for $1 < |\xi| < 2$, $\varphi \in C^\infty(\R)$. \\
{\modar We fix $a \in \mathbb{N}^*$ and, for $(i,j,k,\ell)\in\{1,\dots,N\}\times\{1,\dots, n\}\times\{0,\dots,a\}\times\{0,\dots,L_n-1\}$, let us denote by $\mathcal{E}^{i, \ell,(k)}_j$ a random variable, such that $(\mathcal{E}^{i, \ell,(k)}_j)_{i,j,k,\ell}$ are independent variables with law 
\begin{equation} \label{E: law Epsilon}
 \mathcal{E}^{i, \ell,(k)}_j \sim	\cL(2 \tau_n \,L_n \, (a+1)/ \alpha_j),
\end{equation} 
where $\mathcal L(\lambda)$ stands for a Laplace distribution with mean $0$ and location parameter $\lambda>$0, $\tau_n:= \sqrt{\Delta_n} \log(n)$ and $L_n$ is the cardinality of $\Xi$ as above. We assume that the variables $\mathcal{E}^{i, \ell, (k)}_j$ are independent from the data $X_{t_{j'}}^{i'}$ for $i'=1, \dots , N$ and $j'=0, \dots , n$. Then, we set for any $i\in \{1, \dots , N \}$, $j \in \{1, \dots , n \}$, $k \in \{0,\dots,a\}$, $\ell \in \{0,\dots,L_n-1\} $	
	%
\begin{align}\label{Eq_private_Z}
		Z_j^{i,(k)}(\theta_\ell)&:=f^{(k)}(\theta_\ell;X_{t_{j-1}}^{i},X_{t_j}^{i})
	\varphi \left( f^{(k)}(\theta_\ell;X_{t_{j-1}}^{i},X_{t_j}^{i}) / \tau_n \right) +\mathcal{E}_j^{i,\ell,(k)} 	
\\ \nonumber
&=f^{i,\ell,(k)}_j \times \varphi_{\tau_n} \left( f^{i,\ell,(k)}_j\right) + \mathcal{E}_j^{i,\ell,(k)},
\end{align}
where 
\begin{equation} \label{E: def f_ijlk}
f^{i,\ell,(k)}_j:=	f^{(k)}(\theta_\ell;X_{t_{j-1}}^{i},X_{t_j}^{i}) = \frac{\partial^k}{\partial \theta^{k}} f (\theta_\ell, X_{t_{j-1}}^{i},X_{t_j}^{i})
\end{equation}
and $\varphi_{\tau_n}(\cdot):=\varphi(\cdot/\tau_n)$. Remark that the dependence on $k$ in $\mathcal{E}_j^{i,\ell,(k)}$ does not stand for the derivatives of the Laplace itself but it is a reminder of the fact we are adding some noise to the derivatives of $f$. An analogous comment applies to $Z_j^{i,(k)}$.  \\

\noindent The public data is then defined by 
\begin{equation} \label{eq : publiv Zij}
	Z_j^i:= ( Z_j^{i,(k)}(\theta))_{ \theta\in\Xi,0\leq k\leq a}\in\mathbb{R}^{L_n\times(a+1)}.
\end{equation}}

\noindent It is easy to check that the local differential privacy control holds true, as proven in the following lemma.
\begin{lemma}{\label{l: ldp}}
The public variables described in \eqref{eq : publiv Zij} are $\bm{\alpha}$-local differential private views of the original $(X_{t_{j-1}}^{i},X_{t_j}^{i})$. 
\end{lemma}
\begin{proof}
Let us denote by $(z_{j}^{\theta,k})_{\theta \in \Xi, 0\le k \le a}\mapsto q_j\big((z_{j}^{\theta,k})_{\theta \in \Xi, 0\le k \le a} \mid X_{t_j}^{i}=x_j, X_{t_{j-1}}^{i} = x_{j-1}\big)$ the density of the public data $Z_j^i = (Z_j^{i,(k)}(\theta))_{\theta \in \Xi, 0\le k \le a} \in \mathbb{R}^{L_n\times (a+1)}$ conditional on $X_{t_j}^{i}=x_j$, $X_{t_{j-1}}^{i} = x_{j-1}$ for $j\in\{1,\dots , n\}$ and $i \in \{1, \dots , N \}$. \\
As $\mathcal{E}^{i, \ell,(k)}_j$ is distributed as a centered Laplace random variable with scale parameter $2\tau_n L_n(a+1)/\alpha_j$, its density at the point $x\in \R$ is given by $\frac{1}{ 2\tau_n L_n(a+1)} \alpha_j \exp(- \frac{1}{ 2\tau_n L_n(a+1)} \alpha_j |x|)$. Then, using the independence of the variables $(Z_j^{i,(k)}(\theta))_{\theta \in \Xi, 0 \le k \le a}$
{\modch and denoting  $\varphi_{f^{(k)},\tau_n}^{\theta}(x,y): = \varphi\left( f^{(k)}(\theta;x ,y)/\tau_n\right)$ for $\theta \in \Xi$, we have}
\begin{align*}
&\frac{q_j((z_{j}^{\theta,k})_{\theta \in \Xi, 0 \le k \le a} \mid X_{t_j}^{i}=x_j, X_{t_{j-1}}^{i} = x_{j-1})}{q_j((z_{j}^{\theta,k})_{\theta \in \Xi, 0 \le k \le a} \mid X_{t_j}^{i}=x_j', X_{t_{j-1}}^{i} = x_{j-1}')}\\
&= \prod_{\theta\in \Xi } \prod_{k=0}^a \exp\bigg[\frac{\alpha_j}{2\tau_n L_n(a+1) }\big|z-f^{(k)}(\theta;x_j,x_{j-1})\varphi_{f^{(k)},\tau_n}^{\theta}(x_{j-1},x_j)\big|\\
&\hspace{100pt}-\frac{\alpha_j}{2 \tau_n L_n(a+1) }\big|z-f^{(k)}(\theta;x_j',x_{j-1}') \varphi_{f^{(k)},\tau_n}^{\theta}(x_{j-1}',x_j')\big|\bigg]\\
& \leq \prod_{\theta\in \Xi } \prod_{k=0}^a \exp\bigg[\frac{\alpha_j}{2\tau_n L_n(a+1) }\big|f^{(k)}(\theta;x_j,x_{j-1})\varphi_{f^{(k)},\tau_n}^{\theta}(x_{j-1},x_j) \\
&\hspace{100pt}- f^{(k)}(\theta;x_j',x_{j-1}') \varphi_{f^{(k)},\tau_n}^{\theta}(x_{j-1}',x_j')\big|\bigg] \\
&\le \prod_{\theta \in \Xi} \prod_{k=0}^a \exp\Big(\frac{\alpha_j}{L_n(a+1)}\Big)=\exp\big(\alpha_j\big),
\end{align*}
where we have 
used the fact that $f^{(k)}(\theta;X_{t_{j-1}}^{i},X_{t_j}^{i})\varphi_{f^{(k)},\tau_n}^{\theta}(X_{t_{j-1}}^{i},X_{t_j}^{i})$ is bounded by ${\tau_n}$ from the construction of $\varphi_{f^{(k)},\tau_n}^{\theta}$, together with $\mathrm{card}(\Xi)=L_n$.
By integrating the numerator and the denominator over any measurable set $A\in\Xi_{\mathcal{Z}_j}$, and then taking the supremum, we get
\begin{equation*}
\sup_{A \in  \Xi_{\mathcal{Z}_j}}
\frac{Q_j(A |X^i_{t_j} = x_j, X^i_{t_{j-1}} = x_{j-1} )}{Q_j(A |X^i_{t_j} = x'_j, X^i_{t_{j-1}} = x'_{j-1})} \le \exp(\alpha_j).
\end{equation*}
Hence the result.
\end{proof}

%
%
%
In order to define our estimator and derive the convergence results, we must {\modar smoothly} extend the definition of
 {\modar $Z_j^{i,(0)}(\theta_\ell)=f(\theta_\ell;X_{t_{j-1}}^{i},X_{t_j}^{i})
 	\varphi \left( f(\theta_\ell;X_{t_{j-1}}^{i},X_{t_j}^{i}) / \tau_n \right) +\mathcal{E}_j^{i,\ell,{\modch(0)}} $ from} the grid $\Xi$ to the entire space $\Theta$.
{\modch Indeed, recalling the contrast function without privacy is as in  \eqref{Eq_Contrast_function}, one can easily see} it will enable us to define a smooth contrast function by summing these extensions on $i,j$.
 
  The extension is achieved through the application of a spline approximation method, which is comprehensively discussed in Section \ref{s: splines}. More specifically, starting with the initial definition of $Z_j^{i}(\theta)$, valid for all $\theta \in \Xi$, we will construct its spline approximation of order $a$, denoted as ${\modar H_\Xi} Z_j^{i}(\theta)$. The extended function is well-defined for any $\theta$ and coincides with $Z_j^{i, (0)}(\theta)$ on $\Xi$. {\modar More formally, the spline extension $H_\Xi$ is given by a linear operator
\begin{equation} \label{E: interpol spline avant}
	H_\Xi: 
	\left\{ 
	\begin{aligned}
	   \mathbb{R}^{L_n\times(a+1)}  &\to \mathcal{C}^a([\theta_0,\theta_{L_n-1}])
	\\
			 (g_\ell^{\modch k})_{0\le \ell \le L_n-1,0 \le k \le a} &\mapsto (g(\theta))_{\theta \in [\theta_0,\theta_{L_n-1}]}
	\end{aligned}
\right.
\end{equation}
where the function $g$ is such that $g^{(k)}(\theta_\ell) = g_\ell^k$ for all $\theta_\ell \in \Xi$.}
A thorough elucidation of the process involved in creating such a function can be found in Section \ref{s: splines}. 
\noindent
The natural contrast function in this context, defined for $\theta\in\Xi$, is
$S_n^N(\theta):=\sum_{i=1}^N\sum_{j=1}^nZ_j^{i{\modar ,(0)}}(\theta)$.
Its extension to {\modar $[\theta_0,\theta_{L_n-1}]$} 
	is given by
{\modar 
\begin{equation}\label{Eq_definition_wtSN} 
S^{N,\text{pub}}_n (\theta) :=\sum_{i=1}^N\sum_{j=1}^n H_{\modar \Xi} \left( Z_j^{i} \right)(\theta)=
\sum_{i=1}^N\sum_{j=1}^n H_{\modar \Xi} \left( (Z_j^{i,(k)}(\theta_\ell))_{\ell,k} \right)  (\theta),
 \qquad \theta \in [\theta_0,\theta_{L_n-1}].	
\end{equation}
}
\noindent
With access to public data $(Z_j^{i}(\theta))_{1\leq i\leq N,1\leq j\leq n}$, the statistician considers the estimator
\begin{equation}\label{Eq_estimator_theta}
	\widehat \theta_n^N=\argmax_{\theta\in{\modar [\theta_0,\theta_{L_n-1}]}}{\modar S^{N,\text{pub}}_n(\theta)}.
\end{equation}

{\modar To complete the construction of the private estimator it remains to choose the grid $\Xi$.  Let us start} considering the deterministic uniform grid
	$$\Xi:=\bigg\{\theta_\ell=\frac{\ell}{L_n}, \quad \ell\in\{0,\dots,L_n - 1\}\bigg\}.$$
		
 We will see it will be crucial for the analysis of the variance of the estimator to determine in certain occasions whether the true value of the parameter, denoted as $\theta^\star$, belongs to this grid. To address this, it will be useful in the sequel to introduce a random grid in which the probability of $\theta^\star$ coinciding with any point is zero. 
 	{\modar Let us remark that the extension of the contrast function is from 
 		the grid $\Xi$ to $[\theta_0,\theta_{L_n-1}]=[0,1-1/L_n]$ rather than on the whole parameter set $\Theta=[0,1]$. It is convenient to leave out some space on the the rightmost part of parameter space as we will introduce a random shift of size smaller than $1/L_n$ of the grid in this sequel, to create the randomization.} {\modch That is why we have assumed $\theta^\star \in \overset{\circ}{\Theta}=(0,1)$.} Remark that for the consistency one might take $\theta^\star \in \overset{\circ}{\Theta}=[0,1)$. However, in order to prove the central limit theorem, we need $\theta^*$ in the interior of the parameter set $\Theta$. 


\subsection{Main results: consistency and asymptotic normality}\label{Subsec_2 main results}
\noindent It is feasible to demonstrate that the estimator in \eqref{Eq_estimator_theta}, introduced as a natural extension of the Maximum Likelihood Estimator (MLE) in cases of discrete observations and under local differential privacy constraints, exhibits several desirable properties. Notably, it is consistent, as stated in the following theorem and proved in Section \ref{s: proof main}. \\
Let us recall that $\bar{\alpha}_2$ is the harmonic mean over the different privacy levels. It is such that $1/{\bar{\alpha}_2} = n^{-1} \sum_{j = 1}^n 1/\alpha_j^2$.

\begin{thm}[Consistency]{\label{th: consistency}}
Assume that A\ref{Ass_1}- A\ref{Ass_4} {\modar hold and that 
$\frac{L_n \log(n)}{\sqrt{N}} \sqrt{\frac{\log(L_n)}{\bar{\alpha}_2}} \rightarrow 0$ for $n, N, L_n \rightarrow \infty$.} Then,
the estimator $\widehat\theta_n^N$ defined in \eqref{Eq_estimator_theta} is consistent:
$$\widehat\theta_n^N \xrightarrow{\mathbb{P}} \theta^{\star}.$$
\end{thm}

\noindent {\modch Furthermore, if $\bar{\alpha}_2$ is chosen to be sufficiently large to ensure that the impact of privacy constraints is negligible, it becomes viable to regain the asymptotic normality of the estimator with a convergence rate and a variance reminiscent of the classical scenario where no privacy constraints are imposed (see for example \cite{McKean, DenDioMar, DelGenLar}). In this case the privacy-related influence becomes negligible. This assertion is formally established in the subsequent theorem whose proof is postponed to Section \ref{s: proof main}. To establish this, we introduce the following assumptions regarding the spline approximation and the discretization step.
	\begin{ass}\label{Ass_splines}[Condition spline approximation]
		\begin{enumerate}
\item \label{H: Ass_splines priv neg} Assume that $a > 3$ and that $\frac{\sqrt{N}}{L_n^{a - 2}} \rightarrow 0$.
			\item \label{H: Ass_splines priv dominate} Assume that $a > 3$ and that $\frac{\sqrt{N \bar{\alpha}_2}}{L_n^{a} \log(n)} \rightarrow 0$.
		\end{enumerate}
	\end{ass}
	\begin{ass}\label{Ass_discretization}[Condition discretization step]
		\begin{enumerate}
			\item \label{H: Ass_discret priv neg} Assume that $\sqrt{N \Delta_n} \rightarrow 0$.
			\item \label{H: Ass_discret priv dominate} Assume that $\frac{\sqrt{N \Delta_n \bar{\alpha}_2}}{L_n^{2} \log(n)} \rightarrow 0$.
		\end{enumerate}
	\end{ass}

We note that Conditions A\ref{Ass_splines}.\ref{H: Ass_splines priv neg} and A\ref{Ass_discretization}.\ref{H: Ass_discret priv neg} are essential for demonstrating asymptotic normality in scenarios where the contribution of privacy is negligible. On the other hand, A\ref{Ass_splines}.\ref{H: Ass_splines priv dominate} and A\ref{Ass_discretization}.\ref{H: Ass_discret priv dominate} are required for cases where the contribution of privacy is significant. \\
Recall that $r_{n,N} = \frac{L_n^2 \log(n)}{\sqrt{\bar{\alpha}_2}}$, as introduced in \eqref{Eq_def_rnN}. \\
We remark that A\ref{Ass_splines}.\ref{H: Ass_splines priv dominate} is equivalent to $\frac{\sqrt{N}}{L_n^{a - 2}} \frac{1}{r_{n,N}} \rightarrow 0$ and, similarly, the second point of A\ref{Ass_discretization} translates to $\sqrt{N \Delta_n} \frac{1}{r_{n,N}} \rightarrow 0$. We highlight that, in case of significant privacy, $r_{n,N}$ goes to $\infty$. It implies that Points 1 of both the Assumptions above would be enough to obtain the wanted result. We decided to require the conditions gathered in A\ref{Ass_splines}.\ref{H: Ass_splines priv dominate} and A\ref{Ass_discretization}.\ref{H: Ass_discret priv dominate} as they are weaker.}

\begin{thm}[Asymptotic normality with negligible contribution of privacy]{\label{th: as norm privacy negl}}
{\modar Assume that A\ref{Ass_1}- A\ref{Ass_5}, A\ref{Ass_splines}.\ref{H: Ass_splines priv neg} and A\ref{Ass_discretization}.\ref{H: Ass_discret priv neg} hold.} Assume moreover that {\modch $\sqrt{\log(L_n)}\,  r_{n,N} \rightarrow 0$} and $ \frac{
	{\modar L_n^{3} } \log(n)}{\sqrt{N}} \sqrt{\frac{\log(L_n)}{\bar{\alpha}_2}} \rightarrow 0$ for $N,n, L_n \rightarrow \infty$. Then
$$\sqrt{N} (\widehat\theta_n^N - \theta^{\star}) \xrightarrow{\mathcal{L}} {\mathcal N}\bigg(0,2 \Big(\int_0^T \E\bigg[\Big(\frac{\partial_\theta b(\theta^{\star}, X_s)}{{\modch \sigma}(X_s)}\Big)^2\bigg] ds \Big)^{-1}\bigg) = : {\mathcal N}\big(0,2 ( \Sigma_0 )^{-1}\big)=: \mathcal{Z}_1.$$
\end{thm}

{\modch One might question the efficiency of the proposed estimator, in the scenario of 'negligible' privacy contribution. A recent paper \cite{LAN} establishes the \textit{local asymptotic normality (LAN) property} for drift estimation in $d$-dimensional McKean-Vlasov models under continuous observations. If we narrow down their result to dimension 1, with a single-dimensional parameter $\theta$, the Fisher information derived by the authors can be expressed as follows
\begin{equation*}
\int_0^T \int_{\R} \Big(\frac{\partial_{\theta}b}{{\modch \sigma}}\Big)^2 (\theta^\star, {t}, x, \bar \mu_t) \bar \mu_t (dx) dt,
\end{equation*}
where $\bar \mu_t = \bar \mu_t^{\theta^\star}$ is the law of the McKean-Vlasov SDE under the true value $\theta^\star$ of the parameter. In the case of i.i.d. diffusion processes under consideration, the above quantity simplifies to $\int_0^T \E\left[\left(\frac{\partial_\theta b(\theta^{\star}, X_s)}{{\modch \sigma}(X_s)}\right)^2\right] ds$. It suggests that our estimator is asymptotically efficient when the privacy contribution is 'negligible'.}\\
\\
\noindent We can establish the asymptotic normality of our proposed estimator in the case where the contribution of the privacy is the dominant one as well. However, as evident from Theorem \ref{th: as norm} below, the issue of privacy significantly affects the estimation of the parameter we intend to estimate. This is demonstrated by the degradation of the convergence rate, which now becomes dependent on the privacy levels $\alpha_1, \dots , \alpha_n$. In this case, it is important to observe that we need to ensure that the various levels of privacy do not differ significantly from each other, as elaborated further below.

\begin{ass}\label{Ass_6}[Privacy ratio]
Assume that $\frac{\max_j \alpha_j}{\min_j \alpha_j} = O(1)$.
\end{ass}
\noindent
Note that this assumption is not related to privacy itself but is made to guarantee the Lindeberg condition for the CLT.
Moreover, in this context we need to introduce a random grid, as anticipated before. \\
Let us introduce a random variable, denoted as $S$, that is uniformly drawn in the interval $(0,1)$ and
the following random grid:
\begin{equation}{\label{eq: random grid}}
\Xi_S:=\bigg\{\theta_\ell=\frac{\ell + S}{L_n}, \quad \ell\in\{0,\dots,L_n - 1\}\bigg\}.   
\end{equation}
We define for $s \in [0,1]$, 
{\modar 
\begin{equation} \label{E: def hat v}
\overline{v}(s){\modch:=} (2a+1)^2 \binom{2a}{a}^2 s^{2a}(1-s)^{2a}.
\end{equation}} Let $\ell_n^\star \in \{0,\dots,L_n-1\}$ be such that $\theta^\star\in [\theta_{\ell_n^\star},\theta_{\ell_n^\star+1})$. 
 We set 
 \begin{equation}{\label{eq: def vn}}
   v_n(\theta^\star){\modch :=} \overline{v}(L_n(\theta^\star-\theta_{\ell_n^\star})).  
 \end{equation}
Remark that $v_n(\theta^\star)$ is zero if and only if $\theta^\star$ belongs to the grid $(\theta_\ell)_{0\le \ell \le L_n-1}=((\ell+S)/L_n)_{0\le \ell \le L_n-1}$. As {\modch the} shift variable $S$ has a continuous law, the probability that $\theta^\star$ lies in the grid for some $n \ge 1$ is zero. Thus, we have that, almost surely, $v_n(\theta^{\star})>0$ for all $n \ge 1$.
Moreover, the law of $v_n(\theta^{\star})$ is independent of $\theta^{\star}$ and $n$, as stated in next lemma, whose proof can be found in Section \ref{s: proof preliminary}. 
	\begin{lemma} \label{L: law random variance}
 		For all $n \ge 1$, $v_n(\theta^{\star})$ has the same law as $\overline{v}(s)$.
 	\end{lemma}

\noindent This allows us to obtain the following asymptotic result.

\begin{thm}[Asymptotic normality with significant contribution of privacy]{\label{th: as norm}}
Assume that A\ref{Ass_1}- A\ref{Ass_5}, {\modar A\ref{Ass_splines}.\ref{H: Ass_splines priv dominate}, A\ref{Ass_discretization}.\ref{H: Ass_discret priv dominate}} and A\ref{Ass_6} hold. Assume moreover that ${\modar L_n^3} \log(n) \sqrt{\frac{\log(L_n)}{ N \bar{\alpha}_2}} \rightarrow 0$ and {\modch $r_{n,N} \rightarrow \infty$} for $N, n, L_n \rightarrow \infty$. Then
$$\frac{\sqrt{N \bar{\alpha}_2}}{ 4 {\modar(a+1)} L_n^2 \log(n) \sqrt{T}} (\widehat\theta_n^N - \theta^{\star})\xrightarrow{\mathcal{L}} {\mathcal N}\big(0, \overline{v}(s)(\Sigma_0)^{-2} \big) =: \mathcal{Z}_2$$
where {\modch $\mathcal{Z}_2$ has a mixed normal distribution, with mean $0$ and conditional variance $\overline{v}(s)(\Sigma_0)^{-2}$. Moreover, this convergence holds jointly with $v_n(\theta^{\star}) \xrightarrow{\mathcal{L}} \overline{v}(s)$.}
\end{thm}

{\modch Theorems \ref{th: as norm privacy negl} and \ref{th: as norm} provide insights into how the proposed estimator behaves with varying contributions of privacy, particularly concerning the convergence of $\sqrt{\log(L_n)} r_{n,N}$ towards $0$ or those of $r_{n,N}$ towards $\infty$. As already mentioned in the introduction, the asymptotic behaviour of $r_{n,N}$ dictates the privacy regime. In the corollary below we explore the scenario where privacy levels $(\alpha_1, \dots , \textcolor{black}{\alpha_n})$ satisfy {\modar $ r_{n,N} \sim 1/c_p >0$.} Its proof is presented in Section \ref{s: proof main}.}

\begin{cor}{\label{cor: threshold}}
{\modch Assume that A\ref{Ass_1}- A\ref{Ass_5}, {\modar A\ref{Ass_splines}.\ref{H: Ass_splines priv dominate}, A\ref{Ass_discretization}.\ref{H: Ass_discret priv dominate}} and A\ref{Ass_6} hold. Assume moreover that ${\modar L_n^3} \log(n) \sqrt{\frac{\log(L_n)}{ N \bar{\alpha}_2}} \rightarrow 0$ and \textcolor{black}{there exists $c_p>0$} such that $r_{n,N} \rightarrow 1/c_p$ for $N, n, L_n \rightarrow \infty$. Then
$$\frac{\sqrt{N \bar{\alpha}_2}}{ {\modar  L_n^2 \log(n) }} (\widehat\theta_n^N - \theta^{\star})\xrightarrow{\mathcal{L}}  c_p \mathcal{Z}_1 + 4(a + 1)\sqrt{T}
\mathcal{Z}_2 = {\mathcal N}\big(0, (2c_p^2 \Sigma_0 +  {\modar 16(a+1)^2T} \overline{v}(s))(\Sigma_0)^{-2} \big) $$
and this convergence holds jointly with $v_n(\theta^{\star}) \xrightarrow{\mathcal{L}} \overline{v}(s)$. The random variables $\mathcal{Z}_1$ and $\mathcal{Z}_2$ are independent.} 
\end{cor}

{\modch Corollary \ref{cor: threshold} emphasizes the crucial roles played by both $\mathcal{Z}_1$ and $\mathcal{Z}_2$ in our analysis. The convergence rate of the proposed estimator towards the sum reveals that, in the scenario of negligible privacy, the convergence is towards $\mathcal{Z}_1$, whereas in the case of a significant contribution of privacy, the convergence is towards $\mathcal{Z}_2$.}\\
\\
{ An intriguing question emerges: How can one thoughtfully choose the parameter $L_n$ while ensuring its compliance with the aforementioned conditions? It is noteworthy that $L_n = O(n^r)$, and $r$ arbitrarily small is a feasible choice. In this case, the choice of the order of the spline approximation $a$ becomes pivotal, requiring it to be sufficiently large to ensure that $a r$ satisfies the assumptions outlined in A\ref{Ass_splines}.}\\
\\
\textcolor{black}{Extending our privacy analysis to multidimensional diffusion is theoretically possible, but it would lead to much more tedious computations, making the paper less readable. Additionally, using higher-dimensional spline approximations would create cumbersome notation (see the remark in the Appendix \ref{s: splines}), limiting our analysis to the one-dimensional case for the parameter $\theta$. To avoid unnecessary complexity and ensure clarity, we focus on the unidimensional case for both the parameter and diffusion, allowing readers to better understand the impact of the privacy constraint. The extension to the multidimensional case is left for future research.}\\
\\
\noindent \textcolor{black}{One might wonder why it is worthwhile to study statistical inference on the drift in the negligible privacy regime, especially when this regime appears unsuitable for achieving privacy goals. However, it is important to emphasize that the results in Corollary \ref{cor: threshold} provide a critical threshold between negligible and significant privacy. This helps in choosing the level of privacy. Recall that increasing privacy generally hampers statistical recovery, and the reverse is true as well. Therefore, it becomes crucial to identify a balance that offers a trade-off between statistical utility and privacy protection. By analyzing statistical inference in the negligible privacy regime, we can determine the maximum level of privacy that still allows for sufficient information recovery, ensuring that the resulting estimator is consistent and asymptotically Gaussian—essentially performing as well as if no privacy were imposed on the system.}
\\

\subsection{Particular case: drift as a polynomial function of $\theta$}
\textcolor{black}{This section focuses on the special case where the drift is a polynomial function of \(\theta\). A key aspect of our analysis involves approximating the contrast function from the grid \(\Xi\) to the entire space \(\Theta\) using spline approximations. Importantly, Hermite approximations preserve polynomial functions, which motivates our interest in studying drift functions of the form \(b(\theta, x) = b_1(\theta) b_2(x)\), where \(b_1\) is a polynomial of degree at most \(a\). \\
In this case, we can derive a sharper bound for the error introduced by moving from the contrast function to the classical one in the setting without privacy. This leads to a significant simplification: the size of the grid can now remain constant, i.e., \(L_n = L\). As a result, the condition previously stating that perfect privacy is not allowed simplifies to \(\frac{\log n}{\sqrt{N \bar{\alpha}_2}} \to 0\). Additionally, the conditions outlined in Assumption A\ref{Ass_splines}, which were previously required for the asymptotic normality of the estimator, are no longer necessary. \\
In particular, when the drift is a polynomial function of \(\theta\), the consistency and asymptotic normality results in Theorems \ref{th: consistency}, \ref{th: as norm privacy negl}, and \ref{th: as norm} are adjusted accordingly. }

\begin{thm}{\label{th: pol drift}}
{\rev Assume that \(b(\theta, x) = b_1(\theta) b_2(x)\), where \(b_1\) is a polynomial of degree at most \(a\), that A\ref{Ass_1}-A\ref{Ass_4} hold and that \(\frac{\log n}{\sqrt{N \bar{\alpha}_2}} \to 0\). Then, the following statements hold:
\begin{itemize}
    \item[(i)] The estimator \(\hat{\theta}_n^N\) is consistent in probability.
    \item[(ii)] If \(\frac{\log n}{\sqrt{\bar{\alpha}_2}} \to 0\), and A\ref{Ass_5} and A\ref{Ass_discretization}.1 are also satisfied, then
    \[
    \sqrt{N} (\widehat\theta_n^N - \theta^{\star}) \xrightarrow{\mathcal{L}} \mathcal{Z}_1,
    \]
    where \(\mathcal{Z}_1\) is as defined in Theorem \ref{th: as norm privacy negl}.
    \item[(iii)] If \(\frac{\log n}{\sqrt{\bar{\alpha}_2}} \to \infty\), and A\ref{Ass_5}, A\ref{Ass_discretization}.2, and A\ref{Ass_6} hold, then
    \[
    \frac{\sqrt{N \bar{\alpha}_2}}{\log(n)} (\widehat\theta_n^N - \theta^{\star}) \xrightarrow{\mathcal{L}} \Gamma \mathcal{Z}_2,
    \]
    where \(\mathcal{Z}_2\) is as defined in Theorem \ref{th: as norm}, and \(\Gamma := \frac{4(a + 1)}{\sqrt{T} L^2}\).
\end{itemize}}
\end{thm}
\textcolor{black}{
\noindent
This example helps simplify our results by showing how we expect the convergence rates and the conditions to behave. The extra conditions and rates involving $L_n$ mainly come from the technical challenges of using the spline approximation.
Specifically, we claim that the convergence rate mentioned in the last point above—pertaining to the case of significant privacy— could be optimal, up to an extra \(\log n\) factor.\\
The proof of Theorem \ref{th: pol drift} is provided in Section \ref{App: drift}.}

{\revar 
\subsection{Convergence rate under effective privacy constraints}{\label{s: effective}}
Our main results explicitly show the effect of the privacy levels $(\alpha_j)_j$ on the estimator behaviour. In Section \ref{Ss: effective privacy v2}, we have seen that the effective level of privacy is governed by $\alpha_{\text{eff}}=\sum_{j=1}^n \alpha_j$. The condition $\alpha_{\text{eff}}=O(1)$  is a sufficient one to ensure an effective privacy of the data. For simplicity assume that  $\alpha_j=\alpha$ for all $j \in \id{1}{n}$. 
		As one would like the effective privacy level $\alpha_{\text{eff}}$ to be fixed, we choose ${\alpha} =: \frac{\alpha_{\text{eff}}}{n}$.\\
\noindent		
Observe that, in such a case, the condition $\frac{N^{-1/2}}{L_n}r_{n,N} \rightarrow 0$ needed for consistency translates to $\frac{1}{\log(L_n)(L_n \log(n))^2} N \frac{\alpha_\text{eff}^2}{n^2} \rightarrow \infty$.  If we stick to the policy where the grid size $L_n$ tends slowly to infinity and neglect its contribution, as the one of the $\log$-terms, it implies $N >> n^2$. One might wonder if that is compatible with the condition on the discretization step that puts $N$ and $n$ in a relationship, as required for the asymptotic normality of our estimator. The case of 'negligible' privacy is not compatible with the choice $\alpha = \frac{ \alpha_\text{eff}}{n}$. Hence, we turn to the study of the regime of 'significant' privacy only. In such a case, condition A\ref{Ass_discretization}.2 is implied by $N \Delta_n \frac{\alpha_{\text{eff}}^2}{n^2} \le c$ for some $c> 0$. Hence, we are asking $N \lesssim n^3$.\\
\noindent
In summary, this section helps us understand that, to achieve an effective level of privacy as performant as possible, the number of individuals $N$ providing data and the number of data per person $n$ should be linked in a way ensuring that $n^2 << N \lesssim n^3$. We also see that the choice $\alpha\sim \alpha_\text{eff}/n$ impacts the rate of estimation as $\sqrt{N\bar{\alpha}_2}\sim c\sqrt{N/n^2}$. In the best situation, corresponding to $N\sim n^3$, this rate is of magnitude at most $\sqrt{n^3/n^2}\sim N^{1/6}$, showing a slow rate of estimation, as a function of the number of individuals.

\noindent
We do not known how far is this $N^{1/6}$ rate from an optimal one.
In the context of parametric private estimation based on a $N$-sample, the rate $\sqrt{N\alpha^2}$ is known to be optimal in some models. An example is provided by the estimation of the mean from a Gaussian sample (see \cite{KalSte24}). This gives some insight to interpret the rates of convergence in Theorems \ref{th: as norm}--\ref{th: pol drift}.\\
\noindent
 The slow rate $N^{1/6}$ of our estimator is linked to the repetition of public data due to the high frequency conditions on $\Delta_n$  appearing in Assumption \ref{Ass_discretization},
 which constrains $\alpha$ to go to zero fast enough. These conditions on the step  
		are tied to the bias issue in the high-frequency contrast function method.		
%
%
%
%
It may be possible that other methods, less subject to discretization bias, could deal with a step $\Delta_n$ going slowly to zero and eventually yield to  estimation rates faster than $N^{1/6}$.\\
		
\noindent		
Our study allows the use of different levels of privacy $\alpha_j$ at different instants $t_j$. The effective privacy constraint $\sum_{j=1}^n \alpha_j=\alpha_{\text{eff}}=O(1)$ is a sufficient one to guarantee an effective level of privacy. On the other hand, when the privacy effect is dominating, the variance of our estimator is proportional to
		$\bar{\alpha}_2^{-1}=\sum_{j=1}^n (1/\alpha_j)^2$. Minimizing $\bar{\alpha}_2^{-1}$ under our effective privacy constraint yields to a choice of constant over time privacy level $\alpha=\alpha_{\text{eff}}/n$, which appears optimal in that context. 
		However, the effective privacy constraint $\sum_{j=1}^n \alpha_j=\alpha_{\text{eff}}=O(1)$ neglects any possible decorrelation in time of the components of the hidden vector $\bm{X}^i=(X^i_{t_j})_{j=1,\dots,n}$. This constraint could be possibly sharpened if the 
		data exhibit significant independence structure, as	the context of ergodic diffusions can be an example.

	}

\section{Concluding remarks}{\label{s: discussion}}
{\modch This section provides a discussion arising from the insights gained through this paper. As highlighted earlier, we introduced the concept of local differential privacy (LDP) for diffusion processes and proposed an estimator for parameter estimation of the drift coefficient of i.i.d. diffusion under LDP constraints. Our main findings include the consistency and asymptotic normality of the proposed estimator. The asymptotic normality is achieved at two distinct convergence rates, dependent on whether the term due to the privacy constraint is the primary contribution or not, leading to a dichotomy (up to a logarithmic term). Specifically, when $r_{n,N}$ (defined as in \eqref{Eq_def_rnN}) tends to $\infty$, the contribution of privacy constraints is significant. On the other hand, if $r_{n,N} \sqrt{\log(L_n)} \rightarrow 0$, privacy is negligible, and results align with those in the case without privacy. In the threshold case, where $r_{n,N}$ converges towards a constant, the proposed estimator converges to the sum of the two random variables obtained in the two previous cases. This implies that the general limit of the estimator depends on whether 'significant' or 'negligible' privacy dominates.\\
In the following table, we compare the conditions and results in the case of negligible or significant contributions of privacy, recalling that $r_{n,N} = \frac{L_n^2 \log(n)}{\sqrt{\bar{\alpha}_2}}$.

\begin{center}
\begin{tabular}{|l|l|l|}
\hline
& \textit{Negligible Privacy} & \textit{Significant Privacy}  \\
\hline
\hline
\textbf{Consistency} &  \multicolumn{2}{c|}{ }   \\
Assumptions & \multicolumn{2}{c|}{A\ref{Ass_1}- A\ref{Ass_4} } \\ 
Perfect privacy not allowed
 &  \multicolumn{2}{c|}{$\frac{N^{-1/2}}{L_n}r_{n,N} \rightarrow 0$} \\ 
\hline
\textbf{Asymptotic normality} &  & \\
Dichotomy privacy & $\sqrt{\log(L_n)} r_{n,N} \rightarrow 0$ & $ r_{n,N} \rightarrow \infty$ \\
Spline approximation A6 &$a>3$, $\frac{\sqrt{N}}{L_n^{a - 2}} \rightarrow 0$& $a>3$, $\frac{\sqrt{N \bar{\alpha}_2}}{L_n^{a} \log(n)} \rightarrow 0$ \\
Discretization A7 & $\sqrt{N \Delta_n} \rightarrow 0$ & $\frac{\sqrt{N \Delta_n \bar{\alpha}_2}}{L_n^{2} \log(n)} \rightarrow 0$ \\
Privacy ratio A8 &Not needed & Needed\\
& &\\
Result for $n,N\rightarrow\infty$ & $\sqrt{N} (\widehat\theta_n^N - \theta^{\star}) \xrightarrow{\mathcal{L}}  {\mathcal N}(0, 2 \Sigma_0^{-1})$  & $\frac{\sqrt{N \bar{\alpha}_2}}{ 4 L_n^2{\modarn (a+1)} \log(n) \sqrt{T }} (\widehat\theta_n^N - \theta^{\star})$\\
&&$\qquad\qquad\xrightarrow{\mathcal{L}} \mathcal N(0, \overline{v}(s)\Sigma_0^{-2} )$\\
\hline
\end{tabular}
\end{center}
Here, we address potential extensions and limitations of our main results. Firstly, it is important to note that in the presence of privacy constraints, the levels of privacy $\alpha_j$ are commonly assumed to be smaller than 1. Consequently, the condition of negligible privacy (i.e. $\sqrt{\log(L_n)} r_{n,N} \rightarrow 0$) is quite stringent and rarely met in practice. Therefore, the case of 'significant' privacy emerges as potentially the most interesting. 
{\modarn In such case, we obtain a convergence rate slower than $\sqrt{N\bar{\alpha}_2}$  mainly due to the necessity of introducing a grid with size $L_n$ and the subsequent need for spline approximation. It may be possible to enhance the convergence rate by adopting a different approach, at least in certain models, and the rate optimality of the estimators is an unsolved problem.}\\ 
\noindent
Another remark stems from the dependence among data related to the same individual, impacting the effective privacy. This deviation from the privacy level introduced in Section \ref{s: formulation privacy} results from the extra information carried by side channels, as detailed in Section \ref{s: effective}. This discrepancy leads to a connection between $N$ and $n$, prompting us to ponder what would happen in the case where the time horizon $T$ approaches infinity. While studying the fixed time horizon seems natural from a practical standpoint as it aligns with the model that better fits reality, from a mathematical perspective, understanding the behavior as $T$ approaches infinity in the ergodic case becomes intriguing. In this scenario, allowing for more distant (and less dependent) observations, such as in the case where the discretization step is fixed, would reduce the amount of retrievable information from the side channel, ensuring better effective privacy. \\
\\
{\modarn In conclusion, from our current knowledge this project is a first instance of  local differential privacy for continuous time stochastic processes. 
	We believe that it can be an informative starting point, {\modhel aiding} statisticians in {\modhel comprehending} some challenges {\modhel associated with local differential privacy} for evolving time-based data. {\modhel Additionally, it} offers some insights into potential directions for future research.

}

\vspace{15pt}

\appendix

\noindent
\textbf{\Large{Appendix}}

\section{Preliminary results}\label{Sec_tools}
Before proving our main results, let us introduce some notation and provide some tools that will be useful in the sequel. \\
In particular, we will start by providing a detailed introduction about spline functions.

\subsection{Splines, some tools}{\label{s: splines}}

Here are some basic refreshments on B-splines. For a more detailed exposé, the reader can refer to the work of Lyche \textit{et al.} \cite{Lyche_2017} from which our presentation is widely inspired.
\subsubsection*{Definition and properties of B-splines}
B-splines are piecewise polynomial functions characterized by :
\begin{itemize}
\item A knot vector $\bs{t}=(t_{1},\dots,t_{M})$ that is a nondecreasing sequence of $M$ elements of some interval $I\subset \R$,
\item An integer {\modar $p$ called the degree of the spline.} 
\end{itemize}
In our framework, we exclusively consider B-splines whose associated knot sequence $\bs{t}$ satisfies 
{\modar $M= (a+1)({\modar \Lambda}+3)$ for some ${\modar \Lambda},a\in\N^*$}
 and such that the knots are repeated as follows: 
\begin{equation*}
t_1= \dots= t_{2a+2},\quad {\modch t_{(\ell+1)(a+1)+1}=\dots=t_{(\ell+2)(a+1)},\;\ell\in\id{1}{{\modar \Lambda}-1}, \quad\text{and} \quad t_{(a+1)({\modar \Lambda}+1)+1}=\dots= t_{M}.}
\end{equation*}
For the sake of convenience we introduce an auxiliary sequence, i.e. the sequence of interpolation points $\bs{\xi}=(\xi_\ell,\ell\in\id{-1}{{\modar \Lambda +1}})$ built from $\bs{t}$ as
\begin{equation}\label{Eq_splines_sequence_xi}
\xi_{\ell}=:t_{(a+1)(\ell+1)+1}= \dots= t_{(a+1)(\ell+2)},\;\ell\in\id{-1}{{\modar \Lambda}+1}.
\end{equation}
Note that $\xi_{-1}=\xi_0$ and $\xi_{{\modar \Lambda}}=\xi_{{\modar \Lambda}+1}$: the points $\xi_{-1}$ and $\xi_{{\modar \Lambda}+1}$ have only been introduced for convenience.  
To sum up,
\begin{equation}{\label{eq: def t}}
\bs{t}=(\underbrace{\xi_{-1},\dots,\xi_{-1}}_{a+1\,\text{times}},\underbrace{\xi_0,\dots,\xi_{0}}_{a+1\,\text{times}},\dots,\underbrace{\xi_{{\modar \Lambda}+1},\dots,\xi_{{\modar \Lambda}+1}}_{a+1\,\text{times}}).
\end{equation}
From now on, we assume that the knots are equally spaced, i.e. for all $\ell\in\id{{\modch 0}}{{\modar \Lambda}-1}$
\begin{equation*}
\xi_{\ell+1}-\xi_{\ell}=1/{\modar \Lambda}.    
\end{equation*}
\noindent
We recall the definition of B-splines given in Lyche \textit{et al.} \cite{Lyche_2017}:
\begin{df}\label{Def_splines}
Let $h\in\id{1}{({\modar a+1})({\modar \Lambda}+1)}$. The $h$-th B-spline $B_{h,p,\bs{t}}:I\rightarrow\R$ with degree {\modar $p\le 2a+1$} is identically zero if $t_{h+p+1}=t_h$ and is otherwise defined recursively by
\begin{equation}\label{Eq_recu_spline}
B_{h,p,\bs{t}}:=\frac{\cdot-t_{\modar h}}{t_{h+p}-t_h}B_{h,p-1,\bs{t}}+\frac{t_{h+p+1}-\cdot}{t_{h+p+1}-t_{h+1}}B_{h+1,p-1,\bs{t}},
\end{equation}
starting with
\begin{displaymath}
B_{{\modch h},0,\bs{t}}:=\left\{
\begin{array}{cl}
     1& \text{on}\;[t_{\modch h},t_{{\modch h}+1})  \\
     0& \text{otherwise}.
\end{array}\right.
\end{displaymath}
\end{df}
{\modar Note that the definition of the spline $B_{h,p,\bs{t}}$ relies only on the sequence of $p+2$ knots $t_h\le t_{h+1}\le \dots \le t_{h+p+1}$ and that the 	
	support of $B_{h,p,\bs{t}}$ is $[t_h,t_{h+p+1}]$. 		
	In the sequel, unless stated otherwise, we consider B-spline of order $p=2a+1$. For 
	 the sake of readability, we will simply denote $B_{h}$ for $B_{h,2a+1,\bs{t}}$; i.e. the B-spline with order $p=2a+1$ and underlying knot sequence $\bs{t}$ defined in Equation \eqref{eq: def t}.}
It will be helpful to use the  notation: 
\begin{equation}\label{Eq_spline_start_shape_def}
{\modar B_i^k:=B_{h,2a+1,\bs{t}}, \text{ with $h=(i+1)(a+1)+k+1$ } }	
\end{equation}
where $k\in\{0,\dots,a\}$ {\modar and $i \in \{-1,\dots, {\modar \Lambda}-1 \}$.}
{\modar  Note that the support of $B_{i}^k$ is $[\xi_i,\xi_{i+2}]$.}
We will refer to {\modar $B_{i}^k$ 	as the $p$-order B-splines}
	which \textit{starts at position} $i$ and whose \textit{shape} is $k$. {\modar For $i \in \{0,\dots,{\modar \Lambda}-2\}$ the functions $B_i^k$ are $a$-differentiable, recalling that the smoothness of a spline function $B_{h,p,\bs{t}}$ is equal to the order $p$ minus the maximal multiplicity of a knot appearing in the sequence $t_h,\dots,t_{h+p+2}$. The spline functions $B_{-1}^k$ are right continuous at $\xi_0$ and $a$-differentiable on $(\xi_0,\xi_1]$. Following the convention in \cite{Lyche_2017}, to avoid the asymmetry of splines decomposition on the closed interval $[\xi_0,\xi_{\modar \Lambda}]$, we set 
 $B_{{\modar \Lambda}-1}^k(\xi_{{\modar \Lambda}})=\lim_{x \to \xi_{ \Lambda^-}} B_{{\modch \Lambda}-1}^k(x)$. With this modification, the functions $B_{{\modar \Lambda}-1}^k$ are $a$-differentiable on $[\xi_{{\modar \Lambda}-1}, \xi_{\modar \Lambda}  )$ and left continuous at $\xi_{\modar \Lambda}$.
		The spline functions $(B_i^k)_{i,k}$ restricted to $[\xi_0,\xi_{\modar \Lambda}]$ generate a linear space of piecewise polynomial functions with degree $2a+1$, regularity of class $\cC^a$ on $(\xi_0,\xi_{\modar \Lambda})$, and admitting one-sided derivatives {\modch at the boundaries} of this interval. 
	}
{\rev 
\begin{rem}
In order to estimate a parameter in $\R^d$, we would need to extend the definition of B-splines to higher dimensions. This can be done using tensor product splines, which provide a flexible way to extend splines from one dimension to multiple dimensions by combining univariate spline basis functions in each dimension through the tensor product. More precisely, we can consider the functions
\begin{equation*}
    (x_1,\dots,x_d)\mapsto \sum_{i_1=0}^{H_1}\cdots \sum_{i_d=0}^{H_d}\gamma_{i_1,\dots,i_d} B_{i_1,p,\bs t_1}(x_1)\cdots B_{i_d,p,\bs t_d}(x_d),
\end{equation*}
where the $B_{i_j,p,\bs t_j}$ are B-splines basis functions defined by a knot vector $\bs t_j=(t_1^j,\dots,t_{M}^{j})$ in the $j$-th direction, $H_j$  number of control points (or the number of basis functions) associated to the $j$-th component, and the $\gamma_{i_1,\dots,i_d}$ are the control points arranged in a $d$-dimensional grid. 
\end{rem}}
\color{black}
\vspace{2pt}
%
%
%
\begin{rem}
 As knots are equally spaced, i.e. $\xi_{\ell+1}-\xi_{\ell}=1/{\modar \Lambda}$,
 {\modar it is easy to see from \eqref{Eq_recu_spline} that the spline functions are sharing a common scaling factor ${\modar \Lambda}$. Applying Proposition 11 in \cite {Lyche_2017}, we also get
 \begin{equation}\label{Eq_derivative_spline_bound}
\sup_{x\in[\xi_0,\xi_{\modar \Lambda}]}\Big|\frac{\partial^m }{\partial x^m }B_{i}^k (x)\Big|\le c {\modar \Lambda}^m
 \end{equation}
for $m\in\{0,\dots,a\},~ i \in \{-1,\dots,\Lambda-1\},~k \in \{0,\dots,a\}$.}
\end{rem}

\subsubsection*{Hermite interpolation}
We denote $I^{\neq,{\modar \Lambda}+1}$ the set consisting of $({\modar \Lambda}+1)$-tuples {\modar $\bs{\xi}=(\xi_0,\dots,\xi_{{\modar \Lambda}})$} of $I$ {\modar with $\xi_0<\xi_1<\cdots<\xi_{\modar {\modar \Lambda}}$.} 
%
\begin{df}[Hermite interpolation]
Let $a\in\N^*$, $\bs{\xi}=(\xi_\ell,\ell\in\id{0}{{\modch \Lambda}})\in I^{\neq,{\modar \Lambda}+1}$, {\modch $\mathfrak a_\ell^{(k)} \in \R$ for $\ell\in \id{0}{ \Lambda},\,k\in \id{0}{a}$ and $\mathfrak a=(\mathfrak a_\ell^{(k)},\,\ell\in \id{0}{{\modar \Lambda}} ,\,k\in \id{0}{a})$.}
If there exists a function $H_{\bs{\xi}}(\textcolor{black}{\mathfrak a}): [\xi_0,\xi_{\modar \Lambda}]\rightarrow\R$ of class $\cC^a$ such that
\begin{equation*}
\frac{\partial^k}{\partial\xi^k}\modch H_{\bs{\xi}}({\mathfrak a})(\xi_\ell)= \mathfrak a_\ell^{(k)} \, ; \, \forall k=0,\dots,a,\,\ell=0,\dots,{\modar \Lambda},
\end{equation*}
it is called a \text{Hermite interpolation of} $\mathfrak a=(\mathfrak a_\ell^{(k)},\,\ell\in \id{0}{{\modar \Lambda}},\,k\in \id{0}{a})$ at $\bs{\xi}=(\xi_\ell,\ell\in\id{0}{{\modar \Lambda}})$ with order $a$.
\end{df}
{\modar By Mummy \cite{Mummy89}, it is possible to relate Hermite interpolation with the B-splines functions defined in 
\eqref{Eq_spline_start_shape_def}. We set $\xi_{-1}=\xi_0$, 
	$\xi_{{\modar \Lambda+1}}=\xi_{{\modar \Lambda}}$ and define 
for $-1\le i \le { \Lambda}-1$ and $0\le k \le a$, }
%
%
%
\begin{equation}\label{Eq_function_gk_def}
	g_i^{k}(x)=\frac{1}{{\modar (2a+1)!}}(x-\xi_i)^{\beta_1^k}(x-\xi_{i+1})^{\beta_2^k}(x-\xi_{i+2})^{\beta_3^k}
\end{equation}
%
with
\begin{equation*}
\beta_1^k=(a+1)-k-1=a-k,\; \beta_2^k=a+1, \;\beta_3^k=k.\\
\end{equation*}
The following result can be found in {\modar \cite{Mummy89}}.
\begin{prop}\label{Prop_Hermite_interpolation_def}
\textcolor{black}{Let $\mathfrak a=(\mathfrak a_\ell^{(k)},\,\ell\in \id{0}{ \Lambda},\,k\in \id{0}{a})$.} The function $H_{\bs{\xi}}(\mathfrak{a}):{\modar [\xi_0,\xi_{\modar \Lambda}]}\rightarrow\R$ such that
\begin{equation}\label{Eq_approx_Ha}
H_{\bs{\xi}}(\mathfrak{a})=\sum_{i=-1}^{{\modar \Lambda}-1}\sum_{k=0}^{a}c_i^kB_{i}^k,
\end{equation}
where
\begin{equation}\label{Eq_approx_cik}
c_i^k(\mathfrak{a})=\sum_{v=0}^{{\modar a}}\frac{(-1)^v}{v!}(g_i^k)^{(p-v)}(\xi_{i+1})
\mathfrak{a}_{i+1}^v,
\end{equation}
$p=2a+1$ and the $g_i^k$ are given by \eqref{Eq_function_gk_def}, is a Hermite interpolation of $\mathfrak{a}$
 with order $a$.
\end{prop}
\noindent
{\modar Remark that although the formula \eqref{Eq_approx_Ha} defines a function on the real line, the function $H_\xi(\mathfrak{a})$ is equal to zero outside of $[\xi_0,\xi_{\modar \Lambda}]$, and it is thus generally $\cC^a$ only on the interval $[\xi_0,\xi_{\modar \Lambda}]$.}

{\modar For any function  $f \in \cC^a$, we define its projection $H_\xi f$ on the spline space by setting 
\begin{equation} \label{E: Hermite approx f}
	H_\xi f = H_\xi(\mathfrak{a}) 
	\text {\, with $\mathfrak{a}_i^{k}=
	f^{(k)}(\xi_i) $ for  $0 \le i \le {\modar \Lambda}$ and $0 \le k \le a$.}
\end{equation}
When a function $f$ is already an element of the space generated by the splines functions, then \eqref{Eq_approx_Ha}--\eqref{Eq_approx_cik} is the usual representation of splines 
given for instance in Section 1.2.4 of \cite{Lyche_2017}.}
\begin{rem}
 Hermite interpolation preserves polynomial functions with order not larger than $a$, i.e. for any polynomial function with degree $k\leq a$,
 \begin{equation*}
 {\modar    H_{\bs{\xi}}P(x)=P(x) \quad \text{ for $x \in [\xi_0,\xi_{\modar \Lambda}]$} .}
 \end{equation*}
\end{rem}
\noindent
The map $H_{\bs{\xi}}$ is linear as stated in the following lemma:
\begin{lemma}\label{Lem_linearity_spline}
Let $\phi,\psi:I\rightarrow\R$ and their B-spline interpolating functions at $\bs{\xi}$ respectively denoted by $H_{\bs{\xi}}\phi$ and $H_{\bs{\xi}}\psi$. The B-spline interpolating function of $(f+g)$ denoted by $H_{\bs{\xi}}(f+g)$ satisfies
\begin{equation*}
    H_{\bs{\xi}}(\phi+\psi)=H_{\bs{\xi}}\phi+H_{\bs{\xi}}\psi.
\end{equation*} 
\end{lemma}
\begin{proof}
By the characterization {\modar \eqref{Eq_approx_Ha}--\eqref{E: Hermite approx f}} of $H_{\bs{\xi}}$, we have
\begin{equation*}
H_{\bs{\xi}}(\phi+\psi)=\sum_{i=-1}^{{\modar \Lambda}-1}\sum_{k=0}^{a}c_i^k(\phi+\psi)B_{i}^k,
\end{equation*}
where
\begin{align*}
c_i^k(\phi+\psi)
&=\sum_{m=0}^{a}\frac{(-1)^m}{m!}(g_i^k)^{(p-m)}(\xi_{i+1})(\phi+\psi)^{(m)}(\xi_{i+1})\\
&=\sum_{m=0}^{a}\frac{(-1)^m}{m!}(g_i^k)^{(p-m)}(\xi_{i+1})\phi^{(m)}(\xi_{i+1})+\sum_{m=0}^{a}\frac{(-1)^m}{m!}(g_i^k)^{(p-m)}(\xi_{i+1})\psi^{(m)}(\xi_{i+1})\\
&=c_i^k(\phi)+c_i^k(\psi).
\end{align*}
Then,
\begin{equation*}
H_{\bs{\xi}}(\phi+\psi)=\sum_{i=-1}^{{\modar \Lambda}-1}\sum_{k=0}^{a}c_i^k(\phi+\psi)B_{i}^k
=\sum_{i=-1}^{{\modar \Lambda}-1}\sum_{k=0}^{a}c_i^k(\phi)B_{i}^k+\sum_{i=-1}^{{\modar \Lambda}-1}\sum_{k=0}^{a}c_i^k(\psi)B_{i}^k=H_{\bs{\xi}}(\phi)+H_{\bs{\xi}}(\psi).
\end{equation*}
Hence the result.
\end{proof}
%

\noindent The following results provide insights on the interpolation $H_{\bs{\xi}}$. 

\begin{lemma}\label{Lem_infinite_bound_deriv_interpolation}
For all collections $\bs{\xi}\in I^{\neq,{\modar \Lambda}+1}$, {\modar
with $\xi_{\ell+1}-\xi_{\ell}=1/\Lambda$ for $\ell =0,\dots,\Lambda-1$,
}
 and $\mathfrak a:=(\mathfrak a_\ell^{(k)},\,\ell\in \id{0}{{\modar\Lambda}},\,k\in\id{0}{a})$, the following holds true:
\begin{equation*}
{\modar \sup_{x\in[\xi_0,\xi_{\modar\Lambda}]}	\Big|\frac{\partial^u}{\partial x^u} H_{\bs{\xi}}(\mathfrak a)(x)\Big| \leq c {\modar\Lambda}^{u}\sup_{\ell,k}|\mathfrak a_\ell^{(k)}|,}
\end{equation*}
where $H_{\bs{\xi}}(\mathfrak a)$ is Hermite interpolant to the data $\mathfrak a$ at $\bs{\xi}$.
\end{lemma}
\begin{proof}
From Proposition \ref{Prop_Hermite_interpolation_def}, the function $H_{\bs{\xi}}(\mathfrak a)$ is defined by
\begin{equation*}
H_{\bs{\xi}}(\mathfrak a)=\sum_{i=-1}^{{\modar\Lambda}-1}\sum_{k=0}^{a}\sum_{v=0}^{{\modar a}}\frac{(-1)^v}{v!}(g_i^k)^{(p-v)}(\xi_{i+1})\mathfrak a_{i+1}^{(v)}B_{i}^k.
\end{equation*}
We recall that $p = 2a + 1$. 
Using \eqref{Eq_derivative_spline_bound} and noting that {\modar $|(g_i^k)^{(p-v)}(\xi_{i+1})|$ is bounded by $1$,}
we get
\begin{align*}
{\modar \sup_{x\in[\xi_0,\xi_{\modar\Lambda}]} }
\Big|\frac{\partial^u}{\partial x^u} H_{\bs{\xi}}(\mathfrak a)(x)\Big|
&\leq \sum_{i=-1}^{{\modar\Lambda}-1}\sum_{k=0}^{a}\sum_{v=0}^{{\modar a}}{\modar \frac{1}{v!} |\mathfrak a_{i+1}^{(v)}|}
{\modar \sup_{x\in[\xi_0,\xi_{\modar\Lambda}]} } \Big|
\frac{\partial^u}{\partial x^u}B_{i}^k(x)\Big|\\
&\leq c {\modar\Lambda^u}{\modar \sup_{\ell,k}|\mathfrak a_\ell^{(k)}|}
\end{align*}
for some constant $c>0$ that does not depend on ${\modar\Lambda}$.
\end{proof}
{\modar 
\begin{prop}\label{Prop_Hermite_interpolation}
Let $\bs{\xi}=(\xi_\ell,\ell\in\id{0}{{\modar\Lambda}})\in I^{\neq,{\modar\Lambda}+1}$ {\modar such that $\xi_{\ell+1}-\xi_
	\ell= 1/{\modar\Lambda}$ for $\ell=0,\dots,{\modar\Lambda}-1$,  and let $f : I \to \mathbb{R}$ a function of class $\cC^{a+1}$.
We denote its Hermite interpolation $H_{\bs{\xi}}f$ at $\bs{\xi}$ defined as \eqref{E: Hermite approx f}. For all $k \in \{0,\dots,a\}$, there exists a constant $c>0$ independent of $f$ such that 
\begin{equation}\label{eq: approx spline}
\sup_{x \in [\xi_0,\xi_{\modar\Lambda}]}	\left|\frac{\partial^k}{\partial x^k}f(x)-\frac{\partial^k}{\partial x^k}H_\xi f(x)\right|  \le \frac{c}{{\modar\Lambda}^{a+1-k}}
\sup_{x \in [\xi_0,\xi_{\modar\Lambda}]}	\left|\frac{\partial^{a+1}}{\partial x^{a+1}}f(x)\right| .
\end{equation}}
\end{prop}
}
\begin{proof}
 The proof is postponed to Subsection \ref{subsec_proof_splines}
\end{proof}
\begin{rem}
	If the function $f$ in the previous proposition is a polynomial function of order at most $a$, then we see {\modch that} the right hand side in \eqref{eq: approx spline} is zero. This is not surprising as we know that the Hermite approximation preserves polynomial functions. 
\end{rem}

{\modar In our asymptotics, {\modch we will see that} the number of points $\Lambda$ {\modch will be asymptotically equivalent to }$L_n$. {\modch Hence, it will tend to infinity and so} it will be convenient to use the scaling and translation  properties (see Section 1.1.1 in \cite {Lyche_2017}) of the spline functions to represent $B_{\ell}^k$. Assuming that $\xi_{\ell+1}-\xi_\ell= 1/{\modar\Lambda}$ for $\ell \in \{0,\dots,{\modar\Lambda}-1\}$, we can represent, for $k\in\{0,\dots,a\}$, $\ell \in \{0,\dots,{\modar\Lambda}-{\modch 1}\}$, 
	\begin{equation*}
		B_{\ell}^k(x)=  \overline{B}_k({\modar\Lambda}(x-\xi_\ell)),
	\end{equation*}
where $\overline{B}_k$ is the spline function of order $p=2a + 1$ constructed on the knots $0,
1, 2$ where the knot $0$ is repeated $a+1-k$ times, while the knot $1$ is repeated $a + 1$ times and
the knot 2 is repeated $k + 1$ times.}
%
We have the following result:

\begin{lemma}{\label{l: spline g}}
With previous notation, for all $x\in {\modar \mathbb{R}}$,
\begin{align*}
	\sum_{k=0}^{a}\overline B_{k}'(x)
	=(2a+1) \binom{2a}{a} \Big[x^{a}(1-x)^a {\modar \mathbf{1}_{[0,1]}(x)}-(x-1)^{a}(2-x)^a
	{\modar \mathbf{1}_{[1,2]}(x)}\Big].
\end{align*}    
%
\end{lemma}

\begin{proof}
The proof is postponed to Subsection \ref{Subsec_spline g}.    
\end{proof}

\subsubsection*{Public data extension by B-splines}
Starting from the notation introduced in the previous subsection on B-splines, we consider the case where $I=\Theta={\modar [0,1]}$, {\modar $\Lambda=L_n-1$} and $\xi_\ell=\theta_\ell$ for all $\ell\in\id{0}{\Lambda}$, i.e. the sequence of interpolation points \eqref{Eq_splines_sequence_xi} is here the sequence of grid points $(\theta_\ell,\,\ell\in\id{0}{L_n-1})$. 
{\modar Now, the operator $H_\Xi$ given in \eqref{E: interpol spline avant} is the operator $H_\xi$ detailed in the previous subsection.
	We recall that for any  $(i,j)\in\id{1}{N}\times\id{1}{n}$, the public data $Z_j^i=( Z_j^{i,(k)}(\theta))_{ \theta\in\Xi,0\leq k\leq a}\in\mathbb{R}^{L_n\times(a+1)}$ 
	is available. It encapsulates some proxy for the values of $f(\theta;X^i_{t_{j-1}},X^i_{t_j})$, where $\theta$ is an element of the grid, as well as for its derivatives up to order $a$.}
Applying for any $(i,j)\in\id{1}{N}\times\id{1}{n}$ the Proposition \ref{Prop_Hermite_interpolation_def} to the data $Z_j^i$
we define $H_\Xi Z_j^{i}:=H_{\xi}Z_j^{i}$ the Hermite interpolation of $Z_j^i$ on the grid $\Xi$, i.e. such that for all $k\in\left \{0, \dots , a \right \}$, $\theta\in\Xi=\{\theta_0,\dots,\theta_{L_n-1}\}$,
\begin{equation*}
(H_{\modch \Xi} Z_j^{i})^{(k)}(\theta)={\modar Z_j^{i,{(k)}}(\theta).}
\end{equation*}
\textcolor{black}{To lighten the notation we will write $H$ for $H_\Xi$.}

\subsection{Technical results}{\label{s: technical}}
In the following, it will be particularly convenient to keep track of the size of the different terms we will be working with. The definition of the function {\modar $R_{t,n}^N$} as below will help us in such sense. \textcolor{black}{We note that this notation is widely used in the literature to denote a remainder function. Key references where it has been applied in the context of parameter estimation for stochastic processes include \cite{42McK} and \cite{Shimizu}, among others.} \\
Define $\cF_{t}^N=\sigma\big(X_{s}^{i},\, s\le t,\,i\in\{1,\dots,N\}\big)\vee \sigma \big(\mathcal{E}_{u}^{i,\ell,(k)},\, t_u\le t,\,i\in\{1,\dots,N\},\,\ell \in \{0,\dots,L_n-1\}, k \in \{0, ... , a \} \big)$ and $\E_{t}[\cdot] := \E[\cdot | \cF_{t}^N]$. For a set of random variables $(R^{N}_{t, n})$ and $\tilde{k} \ge 0$, the notation $R^{N}_{t, n} = R_{t}(\Delta_n^{\tilde{k}})$ means that $R^{N}_{t, n}$ is $\cF_{t}^N$-measurable and the set $(R^{N}_{t, n}/\Delta_n^{\tilde{k}})$ is bounded in $L^q$ for all $q \ge 1$, uniformly in $t, n, N$. Hence, there exists a constant $C_q > 0$ such that 
\begin{equation}{\label{eq: def R}}
\E\bigg[\Big|\frac{R^{N}_{t, n}}{\Delta_n^{\tilde{k}}}\Big|^q\bigg] \le C_q
\end{equation}
for any $t, n, N, q \ge 1$. {\modar If this remainder term also depends on the individual $i \in \id{0}{N}$, we assume that the control is uniform in $i$.} Thanks to the definition just provided, the function $R$ has the following useful property
\begin{equation*}
R_{t}(\Delta_n^{\tilde{k}}) = \Delta_n^{\tilde{k}} R_{t}(1).  
\end{equation*}
We underline that the equation above does not entail the linearity of $R$, as in the left and the right hand side above the two functions $R$ are not necessarily the same but just two functions on which the control in \eqref{eq: def R}
is satisfied. 

Let us state a technical lemma that gathers some moment inequalities we will use several times in the sequel. For the interested reader, its proof can be found for example in Lemma 5.1 of \cite{McKean}. \\
In the sequel the notation $c$ refers to a general constant and its value may change from line to line. 

\begin{lemma}{\label{l: moment}}
Assume that Assumption \ref{Ass_1} is satisfied. Then, for all $p \ge 1$, $0 \le s < t \le T$ such that $t - s \le 1$ and $i \in \{1, \dots , N \}$, the following holds true. 
\begin{itemize}
    \item $\sup_{t \in [0, T]} \E[|X_t^{i}|^p] < c$. 
    \item $\E[|X_t^{i}- X_s^i|^p] \le c (t - s)^{\frac{p}{2}}$.
    \item $\E_s[|X_t^{i}- X_s^i|^p] \le c (t - s)^{\frac{p}{2}} {\modar R_s(1)}$.
\end{itemize}

\end{lemma}

The following lemma will be useful in studying the asymptotic behavior of the elements that will come into play. Its proof is provided in Section \ref{s: proof preliminary}.

\begin{lemma}{\label{l: asymptotic 2.5}}
Suppose that Assumption \ref{Ass_1} holds. Let $g: \mathbb{R} \rightarrow \mathbb{R}$ satisfy, for some $c > 0$, $k \in \mathbb{N}$, $x,y \in \R$
$$|g(x) - g (y)| \le c |x - y|(1 + |x| + |y|)^k.$$
Then, the following convergence in probability holds true
$$\frac{\Delta_n}{N} \sum_{i = 1}^N \sum_{j =1}^n g(X^i_{t_{j - 1}}) \rightarrow \int_0^T \E[g(X_s)] ds.$$
\end{lemma}

In the main body of our paper, we will frequently rely on the fact that
{\modar $\varphi_{\tau_n}(f^{i,\ell,(k)}_j)$} 
can be approximated as $1$.
 To establish this, the following lemma will be of utmost importance. Its proof is deferred to Section \ref{s: proof preliminary}.

\begin{lemma}{\label{l: bound proba}}
Assume that Assumptions \ref{Ass_1}- \ref{Ass_3} are in order. Recall that $f(\theta;X_{t_{j-1}}^{i},X_{t_j}^{i})$ has been given in 
{\modar \eqref{Eq_Contrast_function}}
	and $\tau_n$ has been chosen as $\sqrt{\Delta_n} \log(n)$. Then, for any $r \ge 2$ and any $k \ge 0$, there exist constants $c_1, c_2 > 0$ such that, 
$$\mathbb{P}_{t_{j-1}}(|f^{(k)}(\theta;X_{t_{j-1}}^{i},X_{t_j}^{i})| > \tau_n)
\le  \frac{\Delta_n^{r/2}}{(\log(n))^r}R_{t_{j-1}}(1)+c_1\exp\Big(-c_2(\log(n))^2\Big).$$    
\end{lemma}

\section{Proof of main results}{\label{s: proof main}}

\subsection{Consistency}

Let us start by providing the proof of the consistency as gathered in Theorem \ref{th: consistency}. This heavily relies on the fact that we can move from the contrast $S^{N,\text{pub}}_n$ as defined in \eqref{Eq_definition_wtSN} to the contrast $S_n^{N, 0}(\theta)$ that we would have in absence of privacy constraints and the latter is given by 

\begin{equation}{\label{eq: def contrast without privacy}}
S_n^{N,0}(\theta) := \sum_{i = 1}^N \sum_{j = 1}^n  f(\theta;X_{t_{j-1}}^i,X_{t_{j}}^i),
\end{equation}
{\modar where we recall 
	\begin{equation}  \label{E: def f apres}
		f(\theta;X_{t_{j-1}}^i,X_{t_{j}}^i)=\frac{2b(\theta,X_{t_{j-1}}^{i})(X_{t_j}^{i}-X_{t_{j-1}}^{i})-\Delta_n b^2(\theta,X_{t_{j-1}}^{i})}{{\modch \sigma}^{2}(X_{t_{j-1}}^{i})}.
	\end{equation}}
Then, the following proposition provides a bound, for any $ \textcolor{black}{u\in\id{0}{a}},$ on the quantity
\begin{equation*}
E_{n,N}^{(u)} := \sup_{\theta \in {\modar [\theta_0,\theta_{L_n-1}]}} \Big|\frac{\partial^u}{\partial \theta^u}({\modar S_n^{N,\text{pub}}(\theta)} - S_n^{N,0}(\theta))\Big|.
\end{equation*}

\begin{prop}{\label{prop: approx contrast}}
Assume that Assumptions \ref{Ass_1}- \ref{Ass_3} are in order. 
Then, for any $p \ge 2$ and for any $ \textcolor{black}{u\in\id{0}{a}},$ {there exists a constant $c > 0$ such that}
\begin{align*}
\left \| E_{n,N}^{(u)} \right \|_{p} &  { \le c\bigg(N\Big(\frac{1}{L_n}\Big)^{a - u - 1} + L_n^{u + 1} \sqrt{N} \log(n) \sqrt{\frac{\log(L_n)}{\bar{\alpha}_2}}+ \frac{L_n^u N}{n^{r}}\bigg),}
\end{align*}
{\modar where the constant $r>0$ can be chosen arbitrarily large, and}
where we recall that $\bar{\alpha}_2$ is such that $1/\bar{\alpha}_2=n^{-1} \sum_{j = 1}^n 1/\alpha_j^2$.
\end{prop}

\begin{proof}
	{\modar Recall that, according to Section \ref{S: Privatization mechanism}, the contrast function $S^{N,\text{pub}}$ is defined by  
	\begin{equation*}
		S^{N,\text{pub}}_n(\theta)=\sum_{i=1}^n \sum_{j=1}^N H((Z_j^{i,(k)}(\theta_\ell))_{\ell,k})(\theta),
\end{equation*}
	where $H$ is the interpolation operator \eqref{E: interpol spline avant}, fully described in Section \ref{s: splines}. The private {\modch contrast function}
	\begin{equation*}
		S^{N,0}_n(\theta) = \sum_{i=1}^n \sum_{j=1}^N f(\theta;X_{t_{j-1}}^i,X_{t_{j}}^i)
	\end{equation*}
admits the spline projection
	\begin{equation*}
	HS^{N,0}_n(\theta) = \textcolor{black} {H\Big(}\sum_{i=1}^n \sum_{j=1}^N f(\cdot\,;X_{t_{j-1}}^i,X_{t_{j}}^i)\Big)(\theta)
	= \sum_{i=1}^n \sum_{j=1}^N H((f_{j}^{i,\ell,(k)})_{\ell,k})(\theta),
\end{equation*}
where we recall $f_{j}^{i,\ell,(k)}$ is as in \eqref{E: def f_ijlk}.
}	
The proof is now divided in two steps. In the first we evaluate the error committed by approxima\textcolor{black}{ting} $S_n^{N, 0}(\theta)$ with {\modar $H S_n^{N, 0}\textcolor{black}{(\theta)}$,} while in the second we move from {\modar $H S_n^{N, 0}(\theta)$ to $S_n^{N,\text{pub}}\textcolor{black}{(\theta)}$,} the contrast we propose in presence of privacy constraints.
{\modar More formally, the proposition is a consequence of \eqref{E: contrast error part1} and \eqref{E: contrast error part2} below.} \\

\noindent \textbf{Step 1: } \\
{\modar In this part of the proof, we show that, for any $ \textcolor{black}{u\in\id{0}{a}},$
\begin{equation} \label{E: contrast error part1}
\textcolor{black}{\E}\left[ \sup_{\theta  \in {\modar [\theta_0,\theta_{L_n-1}]}} \abs*{ \frac{\partial^u }{\partial \theta^u}
(HS^{N,0}_n(\theta) - S^{N,0}_n(\theta) )
}^p \right] \le {c} \left(\frac{1}{L_n} \right)^{(a-u-1)p} N^p.
\end{equation}
}
{\modar 
\noindent 
We apply Proposition \ref{Prop_Hermite_interpolation} to the function $H{\modch S}^{N,0}_n$ {\modch with $\Lambda = L_n - 1$. Note that here and in the subsequent discussions, for the sake of simplicity in notation, we will replace $L_n - 1$ with $L_n$, as they are asymptotically equivalent.} It follows
\begin{align*}
\sup_\theta \abs*{\frac{\partial^{u}}{\partial\theta^{u}}\Big( H S_n^{N, 0}\textcolor{black}{(\theta)}- S_n^{N, 0}\textcolor{black}{(\theta)} \Big)}\le c \Big(\frac{1}{L_n}\Big)^{a-u-1}\times \sup_{\theta}\Big|\frac{\partial^{a+1}}{\partial\theta^{a+1}}S_n^{N,0}(\theta)\Big|,
\end{align*}
where the supremum on $\theta$, here and below, is on $\theta \in [\theta_0,\theta_{L_n-1}]$.}
Then, 
\begin{align}{\label{eq: 33,5}}
\E\bigg[\sup_{\theta} \Big| \frac{\partial^{u}}{\partial\theta^{u}}\Big( H S_n^{N, 0}(\theta)- S_n^{N, 0} (\theta) \Big) \Big|^p\bigg] \le c \Big(\frac{1}{L_n}\Big)^{(a - u - 1)p} \E\bigg[\sup_{\theta} \Big| \frac{\partial^{a+1}}{\partial\theta^{a+1}} S_n^{N, 0} (\theta) \Big|^p\bigg].
\end{align}
We now aim to demonstrate that, for some constant $c > 0$,
\begin{equation}{\label{eq: goal step 1}}
\E\bigg[\sup_{\theta} \Big| \frac{\partial^{a+1}}{\partial\theta^{a+1}} S_n^{N, 0} (\theta) \Big|^p\bigg] \le c N^p.
\end{equation}
Once established, it will yield the desired result for the approximation of $S_n^{N, 0}$ with its spline function, thereby concluding the proof of Step 1. To achieve this, we begin by analyzing $\partial^{a+1} S_n^{N, 0} (\theta)/\partial\theta^{a+1}$.
Using the definitions of $S_n^{N, 0}$ as given in \eqref{eq: def contrast without privacy} and the function $f$ in \eqref{E: def f apres}, one obtains
\begin{align*}
\Big|\frac{\partial^{a+1}}{\partial\theta^{a+1}} S_n^{N, 0}(\theta)\Big| & { \modar = } 
\abs*{ \sum_{i=1}^N\sum_{j=1}^n \frac{\partial^{a+1}}{\partial\theta^{a+1}} f(\theta; X^i_{t_j}, X^i_{t_j-1}) } \nonumber \\
&\leq \abs[\bigg]{ \sum_{i=1}^N\sum_{j=1}^n\frac{\partial^{a+1}}{\partial\theta^{a+1}}\frac{2b(\theta,X_{t_{j-1}}^{i})\big(X_{t_j}^{i}-X_{t_{j-1}}^{i}\big)}{{\modch \sigma}^{2}(X_{t_{j-1}}^{i})}}+\abs[\bigg]{\sum_{i=1}^N\sum_{j=1}^n \frac{\partial^{a+1}}{\partial\theta^{a+1}}\frac{\Delta_n b^2(\theta,X_{t_{j-1}}^{i})}{{\modch \sigma}^{2}(X_{t_{j-1}}^{i})}}\\
&=:\abs*{S_n^{N,(a+1,1)}(\theta)}+\abs*{S_n^{N,(a+1,2)}(\theta)}. \nonumber
\end{align*}
From Assumption \ref{Ass_2}, ${\modch \sigma}^2$ is lower bounded by ${\modch \sigma}^2_{\min}$. Moreover, Assumption \ref{Ass_3} implies that $\sup_{\theta, i, j} |\frac{\partial^{a+1}}{\partial\theta^{a+1}} b^2(\theta,X_{t_{j-1}}^{i})|^p$ is bounded by some constant $c$. Then, 
\begin{align}{\label{eq: end Snu2}}
 \sup_{{\modar \theta }} \Big|S_n^{N,(a+1,2)}(\theta)\Big|^p & \le c (\Delta_n \, n\, N)^p  \sup_{\theta, i, j} \Big|\frac{\partial^{a+1}}{\partial\theta^{a+1}} b^2(\theta,X_{t_{j-1}}^{i})\Big|^p \nonumber \\
 & \le c N^p.
\end{align}
Let us now move to the analysis of $S_n^{N,(a+1,1)}(\theta)$. We want to control $\E[\sup_{{\modar \theta }} |S_n^{N,(a+1,1)}(\theta)|^p]$ and so, in order to deal with the $\sup$ inside the expectation, we use a Kolmogorov-type argument. Observe we can write $S_n^{N,(a+1,1)}$ in an integral way by introducing the function $\psi_s (n) := \sup \{ j\textcolor{black}{\in\id{0}{n}}: t_j \le s \}$. Then, 
\begin{equation}{\label{eq: Snu1 int form 0.25}}
 S_n^{N,(a+1,1)}(\theta) = \sum_{i = 1}^N \int_0^T  \frac{\partial^{a+1}}{\partial\theta^{a+1}}\frac{ b^2(\theta,X_{\psi_n(s)}^{i})}{{\modch \sigma}^{2}(X_{\psi_n(s)}^{i})} dX_s^i. 
\end{equation}
{\modar Let us set $g(\theta,\theta',x) = \frac{\partial^{a + 1}}{\partial\theta^{a + 1}}\Big(\frac{2 b}{{\modch \sigma}^2}\Big)(\theta,x)- \frac{\partial^{a + 1}}{\partial\theta^{a + 1}}\Big(\frac{2 b}{{\modch \sigma}^2}\Big)(\theta',x)$}.
Thanks to Rosenthal inequality {\modar for centered i.i.d. variables {\modch (see Theorem 3 in \cite{Rosenthal1970})},} it follows that, for all $\theta,\theta' \in \Theta$,
{\modar 
\begin{align}
	\E\bigg[\Big|S_n^{N,(a+1,1)}(\theta)&-S_n^{N,(a+1,1)}(\theta')-\Big(\E\Big[S_n^{N,(a+1,1)}(\theta)-S_n^{N,(a+1,1)}(\theta')\Big]\Big)\Big|^p\bigg] \nonumber
	\\ \nonumber
	&= \E \bigg[ \abs[\Big]{ \sum_{i=1}^N   \int_0^T g(\theta,\theta',X^i_{\psi_n(s)}) dX^i_s - \sum_{i=1}^N
\E \Big[ \int_0^T   g(\theta,\theta',X^i_{\psi_n(s)}) dX^i_s \Big]	
  }^p\bigg] 
\\ \nonumber
& \le c \left( \sum_{i=1}^N 
 \E \bigg[ \abs[\Big]{ \int_0^T   g(\theta,\theta',X^i_{\psi_n(s)}) dX^i_s - 
	\E \Big[ \int_0^T   g(\theta,\theta',X^i_{\psi_n(s)}) dX^i_s \Big]	
}^2\bigg] 
\right)^{p/2} 
\\ \nonumber
&\quad + c \sum_{i=1}^N 
\E \bigg[ \abs[\Big]{ \int_0^T   g(\theta,\theta',X^i_{\psi_n(s)}) dX^i_s - 
	\E \Big[ \int_0^T   g(\theta,\theta',X^i_{\psi_n(s)}) dX^i_s \Big]	
}^p\bigg] 
	\\  \nonumber
& \le c N^{{\modch p}/2 -1} \sum_{i=1}^N 
\E \bigg[ \abs[\Big]{ \int_0^T   g(\theta,\theta',X^i_{\psi_n(s)}) dX^i_s - 
	\E \Big[ \int_0^T   g(\theta,\theta',X^i_{\psi_n(s)}) dX^i_s \Big]	
}^{2}\bigg]^{p/2}  
\\ \nonumber
&\quad + c \sum_{i=1}^N 
\E \bigg[ \abs[\Big]{ \int_0^T   g(\theta,\theta',X^i_{\psi_n(s)}) dX^i_s - 
	\E \Big[ \int_0^T   g(\theta,\theta',X^i_{\psi_n(s)}) dX^i_s \Big]	
}^p\bigg] ,
\intertext{having used Jensen's inequality on the first sum, as \textcolor{black}{$p \ge 2$}. {\modch Then, it is upper bounded by}} \nonumber
&  c (N^{p/2-1 }+1)\sum_{i=1}^N 
\E \bigg[ \abs[\Big]{ \int_0^T   g(\theta,\theta',X^i_{\psi_n(s)}) dX^i_s - 
	\E \big[ \int_0^T   g(\theta,\theta',X^i_{\psi_n(s)}) dX^i_s \big]	
}^{p}\bigg]  
\\ {\label{eq: kolm start 0.5}}  & \le
c (N^{p/2-1 }+1)\sum_{i=1}^N 
\E \bigg[ \abs[\Big]{ \int_0^T g(\theta,\theta',X^i_{\psi_n(s)}) dX^i_s}^p 
\bigg],  
\end{align}
where we used again Jensen inequality in the two last lines, with $p>1$.} 
{\modch Observe that, by the dynamics of $X_s^i$, we have}
\begin{align}{\label{eq: BDG on g, 1}}
\E\bigg[{\modch \Big|}\int_0^T\Big(g(\theta,\theta',X_{\psi_n(s)}^{i})  \Big)dX_s^{i}{\modch \Big|^p}\bigg] 
&\le c \E\bigg[{\modch \Big|}\int_0^T\Big(g(\theta,\theta',X_{\psi_n(s)}^{i}) \Big)b(X_s^i, \theta^{\star}) ds {\modch \Big|^p}\bigg] 
\nonumber \\
& \quad +  c \E\bigg[{\modch \Big|}\int_0^T\Big(g(\theta,\theta',X_{\psi_n(s)}^{i}) \Big){\modch \sigma}(X_s^{i}) dW_s^i{\modch \Big|^p}\bigg] \nonumber\\
 & \le c T^{p - 1} \int_0^T \E \big[{\modch |}g(\theta,\theta',X_{\psi_n(s)}^{i}) {\modch |}^p \big] ds \\
 &\quad+ c T^{\frac{p}{2} - 1} \int_0^T \E\big[{\modch |}g(\theta,\theta',X_{\psi_n(s)}^{i}) {\modch |}^p  |{\modch \sigma}(X_s^i)|^p \big] ds, \nonumber
\end{align}
where we have used Burkholder-Davis-Gundy and Jensen inequalities as well as the boundedness of the drift. Remark now that 
\begin{multline*}\big|g(\theta,\theta',X_{\psi_n(s)}^{i})\big|^p=\abs*{
\frac{\partial^{a+1}}{\partial\theta^{a+1}}\Big(\frac{2 b}{{\modch \sigma}^2}\Big)(\theta,X_{\psi_n(s)}^{i})- \frac{\partial^{a+1}}{\partial\theta^{a+1}}\Big(\frac{2 b}{{\modch \sigma}^2}\Big)(\theta',X_{\psi_n(s)}^{i})}^p
\\= \big|\theta - \theta'\big|^p \Big| \frac{\partial^{a+2}}{\partial\theta^{a+2}}\Big(\frac{2 b}{{\modch \sigma}^2}\Big)\big(\tau \theta + (1 - \tau) \theta', X_{\psi_n(s)}^{i}\big)\Big|^p
\end{multline*}
for some $\tau \in [0, 1]$. Replacing this in \eqref{eq: BDG on g, 1} and recalling that $T$ is a fixed constant, it is easy to check that \eqref{eq: BDG on g, 1} is upper bounded by 
\begin{align*}
&c |\theta - \theta'|^p \int_0^T \big( \E[R_{\psi_n(s)}(1)] + \E[R_{\psi_n(s)}(1)R_s(1)] \big)\textcolor{black}{ds} \\
\le & c |\theta - \theta'|^p,
\end{align*}
where the last follows from the definition of function $R_t$, as in \eqref{eq: def R}. Then, going back to \eqref{eq: kolm start 0.5}, we obtain 
\begin{multline*}
\E\bigg[\Big|S_n^{N,(a+1,1)}(\theta)-S_n^{N,(a+1,1)}(\theta')-\Big(\E\Big[S_n^{N,(a+1,1)}(\theta)-S_n^{N,(a+1,1)}(\theta')\Big]\Big)\Big|^p\bigg]  \\ \le c |\theta - \theta'|^p(N^\frac{p}{2} + N) 
     \le c |\theta - \theta'|^p N^\frac{p}{2},
\end{multline*}
as $p \ge 2$ and $N > 1$. 
Moreover, thanks to \eqref{eq: Snu1 int form 0.25} it is 
\begin{equation*}
	\Big|\E\big[S_n^{N,(a+1,1)}(\theta)-S_n^{N,(a+1,1)}(\theta')\big]\Big|\leq c \sum_{i = 1}^N\Big|\E\Big[ \int_0^T (g(\theta,\theta',X_{\psi_n(s)}^{i})) dX_s^i \Big]\Big| \le cN |\theta - \theta'|, 
\end{equation*}
the validity of the last inequality can be easily confirmed by closely following the reasoning presented above. It follows
\begin{equation*}
	\E\bigg[\Big|S_n^{N,(a+1,1)}(\theta)-S_n^{N,(a+1,1)}(\theta')\Big|^p\bigg]  \le c |\theta - \theta'|^p N^p.
\end{equation*}
Let us introduce the modulus of continuity $\omega_h$ i.e.
\begin{equation*}
\omega_h(f)=\underset{|\theta-\theta'|\leq h}\sup|f(\theta)-f(\theta')|.
\end{equation*}
We apply Kolmogorov's criterion as given by Theorem 2.1 in Chapter 1 in \cite{RevuzYor2005}, to get for all $|h|\le 1$
\begin{equation*}
	\E[(\omega_h(S_n^{N,(a+1,1)}))^p] \le c (h^{1-\varepsilon} N)^p,
\end{equation*}
where $\varepsilon>0$ is any fixed constant. 
%
\noindent
Recall now that the interval $[\theta_0,\theta_{L_n} {\modch -1}]\subset\Theta=[0,1]$ has a radius smaller that $1$,
\begin{align}{\label{eq: end Snu1}}
\E\Big[\sup_{\theta}\big|S_n^{N,(a+1,1)}(\theta)\big|^p\Big] & 
\le c \E\Big[\sup_{\theta}\big|S_n^{N,(a+1,1)}(\theta) - S_n^{N,(a+1,1)}(\theta^\star) \big|^p\Big] + 
c \E\Big[\big| S_n^{N,(a+1,1)}(\theta^\star) \big|^p\Big] \nonumber \\
& \leq c \E[\omega_{1}(S_n^{N,(a+1,1)})^p]+c N^p \le c N^p,
\end{align}
as we wanted. 
The bounds gathered in \eqref{eq: end Snu2} and \eqref{eq: end Snu1} yield \eqref{eq: goal step 1} and therefore conclude the proof of Step 1, as
\eqref{E: contrast error part1} is then a consequence of \eqref{eq: 33,5}.\\

\noindent \textbf{Step 2: } \\
From now on, let us compare $H S_n^{N, 0}$ {\modar with $S_n^{N,\text{pub}}$, and we aim to show that for all $u\in\id{0}{a}$,
\begin{equation} \label{E: contrast error part2}
\norm{\sup_{\theta\in [\theta_0,\theta_{L_n {\modch - 1}}]} \abs*{ \frac{\partial^u}{\partial \theta^u} (H S_n^{N, 0}(\theta) - S_n^{N,\text{pub} }(\theta))} }_{p}
	c \bigg( L_n^{u + 1} \sqrt{N} \log(n) \sqrt{\frac{\log(L_n)}{\bar{\alpha}_2}}\bigg) + c\bigg( \frac{L_n^u N}{n^{r}}\bigg).
\end{equation}
From the definitions of $HS^{N,0}_n$ and $S_n^{N,\text{pub}}$ with the linearity of the operator $H$,
}
{\modar \begin{align}
\nonumber
H S_n^{N, 0} -S_n^{N,\text{pub}}
&= H\bigg(\Big(\sum_{i=1}^N\sum_{j=1}^n f^{i,\ell,(k)}_j)_{k,\ell}\Big)\bigg)-H\bigg(\Big(\sum_{i=1}^N \sum_{j=1}^n Z^{i,(k)}_j(\theta_\ell)\Big)_{k,\ell} \bigg) \\&=H\bigg(\Big(\sum_{i=1}^N\sum_{j=1}^n\beta_j^{i,(k)}(\theta_\ell)\Big)_{k,\ell}\bigg)-H\bigg(\Big(\sum_{i=1}^N\sum_{j=1}^n\mathcal{E}_j^{i,\ell,(k)}\big)_{k,\ell}\bigg)=:H(\bm{\beta})-H(\bm{\mathcal{E}})
\label{eq: def implicite E}
\end{align}
}
where
\begin{equation*}
\beta_j^{i,(k)}(\theta_\ell)=f^{i,\ell,(k)}_{j}\big[1-\varphi_{\tau_n}(f_{j}^{i,\ell,(k)})\big],
\end{equation*}
{\modar and we used \eqref{Eq_private_Z}--\eqref{E: def f_ijlk} and recall that $\mathcal{E}_j^{i,\ell,(k)}$ are the Laplace random variables introduced in Section \ref{S: Privatization mechanism}.}
Let
\begin{equation}{\label{eq: set Omega}}
\Omega^{n}:=\Big\{\omega \in\Omega\,:\,\varphi_{\tau_n}\big(f_{j}^{i,\ell,(k)}\big)(\omega)=1,\,\forall k\leq a,\forall i\in\{1,\dots,N\},\forall j\in\{1,\dots,n\}, \forall \ell\in \{0, \dots , L_n - 1\}\Big\}.
\end{equation}
Note that on $\Omega^n$ we have $H S_n^{N, 0} -S_n^{N,\text{pub}} =-H(\bm{\mathcal{E}})$, 
whereas on $(\Omega^n)^c$,
$H S_n^{N, 0}- S_n^{N,\text{pub}} =H(\bm{\beta})-H(\bm{\mathcal{E}})$.
Then, our goal consists in finding a bound on 
\begin{align}{\label{eq: start step 2}}
 \left \| {\modar  \sup_{\theta}} \Big|\frac{\partial^u}{\partial \theta^u}(H S_n^{N, 0}(\theta)- S_n^{N,\text{pub}})(\theta)\Big| \right \|_p  & = \left \| {\modar  \sup_{\theta}} \Big|\frac{\partial^u}{\partial \theta^u}(H S_n^{N, 0}(\theta)\textcolor{black}{-}S_n^{N,\text{pub}}(\theta))\Big| 1_{\Omega^n} \right \|_p \nonumber\\
 & + \left \| {\modar  \sup_{\theta}} \Big|\frac{\partial^u}{\partial \theta^u}(H S_n^{N, 0}(\theta)-S_n^{N,\text{pub}}(\theta))\Big| 1_{(\Omega^n)^c} \right \|_p \nonumber \\
& \le \left \| {\modar  \sup_{\theta}}\Big|\frac{\partial^u}{\partial \theta^u}H (\bm{\mathcal{E}})(\theta)\Big| \right \|_p + \left \| {\modar  \sup_{\theta}} \Big|\frac{\partial^u}{\partial \theta^u}(H (\bm{\beta})(\theta)-H (\bm{\mathcal{E}})(\theta))\Big| 1_{(\Omega^n)^c} \right \|_p.
\end{align}
%
Let us start considering {\modar an upper bound on $H(\bm{\beta})$. Lemma \ref{Lem_infinite_bound_deriv_interpolation} } provides
\begin{equation}{\label{eq: bound inf norm H(beta) 4}}
\Big\|\frac{\partial^u}{\partial\theta^u} H(\bm{\beta})\Big\|_\infty  = \Big\|\frac{\partial^u}{\partial\theta^u} H\Big( \big(\sum_{i = 1}^N \sum_{j = 1}^n \beta_j^{i, (k)}(\theta_\ell)\big)_{\ell,k} \Big) \Big\|_\infty \leq c L_n^{u}\sup_{\ell,k}\Big|\sum_{i = 1}^N \sum_{j = 1}^n \beta_j^{i, (k)}(\theta_\ell)\Big|.
\end{equation}
Observe that, according to the definition of $\beta_j^{i, (k)}$ and $f^{i,\ell,(k)}_{j}$ it is 
\begin{align*}
\sup_{\ell,k}\Big|\sum_{i = 1}^N \sum_{j = 1}^n \beta_j^{i, (k)}(\theta_\ell)\Big|
&\leq \sum_{i=1}^N\sum_{j=1}^n2\sup_{\ell,k}|f^{(k)}(\theta_\ell,X_{t_{j-1}}^{i},X_{t_{j}}^{i})|\\
&\leq c \sum_{i=1}^N\sum_{j=1}^n\Big[|X_{t_{j}}^{i}-X_{t_{j-1}}^{i}|+\Delta_n\Big], 
\end{align*}
{\modar where we used the definition of $f$ recalled in \eqref{E: def f apres} and the fact that
${\modch \sigma}^2$ is lower} bounded thanks to Assumption \ref{Ass_2} and that the derivatives of $b$ with respect of $\theta$ are bounded because of Assumption \ref{Ass_3}. 
Then, second point of Lemma \ref{l: moment} ensures that
\begin{align}\nonumber
\Big\|\sup_{\ell,k}\Big|\sum_{i = 1}^N \sum_{j = 1}^n \beta_j^{i, (k)}(\theta_\ell)\Big| \Big\|_{p} &\leq c \sum_{i=1}^N\sum_{j=1}^n\big\||X_{t_{j}}^{i}-X_{t_{j-1}}^{i}|+\Delta_n\big\|_{p}\\
\label{eq: bound Lp beta 5}
&\leq c \sum_{i = 1}^N \sum_{j = 1}^n\big(\Delta_n^{1/2}+\Delta_n)\leq c N \sqrt n, 
\end{align}
having also used that $n = T/\Delta_n$ and $T$ is a fixed constant. From \eqref{eq: bound inf norm H(beta) 4} and \eqref{eq: bound Lp beta 5} follows
\begin{equation}{\label{eq: end H(beta) 5.5}}
\E\bigg[ \sup_{{\modar \theta}} \Big|\frac{\partial^u}{\partial \theta^u}|H(\bm{\beta})(\theta)|\Big|^p\bigg]\leq c L_n^{up}\E\Big[\sup_{\ell,k}\Big|\sum_{i = 1}^N \sum_{j = 1}^n \beta_j^{i, (k)}(\theta_\ell)\Big|^p\Big] \le c L_n^{up} N^p n^{p/2}, 
\end{equation}
that concludes the analysis of $H(\bm{\beta})$. \\
Let us now study $H(\bm{\mathcal{E}})$. Lemma \ref{Lem_infinite_bound_deriv_interpolation}
with the definition of $H(\bm{\mathcal{E}})$ through \eqref{eq: def implicite E} provides
\begin{equation}{\label{eq: start H E 6}}
\Big\|\frac{\partial^u}{\partial \theta^u}H(\bm{\mathcal{E}})\Big\|_\infty \leq c L_n^u \sup_{\ell,k}\Big|\sum_{i=1}^N\sum_{j=1}^n \mathcal{E}_{j}^{i,\ell,(k)}\Big|,
\end{equation}
where we recall that $\mathcal{E}_{j}^{i,\ell,(k)}\sim[ 2 \tau_n L_n(a+1) /\alpha_j]\cL(1)$ and $\tau_n=\sqrt{\Delta_n}\log(n)$. 
Let us now introduce the random variables $\overset{\circ} {\mathcal{E}}_{j}^{i,\ell,(k)}$ such that
\begin{equation}{\label{eq: norm E 7}}
\mathcal{E}_{j}^{i,\ell,(k)}={2 \tau_n L_n {\modar (a+1)}}\overset{\circ}{\mathcal{E}}_{j}^{i,\ell,(k)}=2 \sqrt{\Delta_n}\log(n)L_n
{\modar (a+1)} \overset{\circ}{\mathcal{E}}_{j}^{i,\ell,(k)},
\end{equation}
where we have also replaced the definition of $\tau_n = \sqrt{\Delta_n} \log(n) = \sqrt{T/n} \log (n)$. The random variables $\overset{\circ} {\mathcal{E}}_{j}^{i,\ell,(k)}$ are i.i.d. whose law is $\cL(1/\alpha_j)$. It is\\
Then, for any fixed $k,\ell$ we have
\begin{align}{\label{eq: normalized laplace 7}}
\sum_{i=1}^N\sum_{j=1}^n\mathcal{E}_{j}^{i,\ell,(k)}& ={{\modar 2 \sqrt{T}(a+1)}\log(n) L_n \sqrt N}\frac{1}{\sqrt N}\frac{1}{\sqrt n}{\modar \sum_{i=1}^N}\sum_{j=1}^n \overset{\circ}{\mathcal{E}}_{j}^{i,\ell,(k)}.
\end{align}
Let us now state a control on the sum of Laplace random variables. Its proof can be found in Section \ref{s: proof preliminary}. 
\begin{lemma}{\label{l: Petrov}}
	Let \textcolor{black}{$U\in\N^*$} and let us introduce a collection of independent random variables $\textcolor{black}{\bs{\mathcal U}=(\mathcal{U}_h)_{1\le h\le U}}$ such that for $h \in \id{1}{U}$, $\mathcal{U}_h\sim\mathcal{L}(1/\textcolor{black}{\gamma_h})$. \color{black} The  harmonic mean of $\gamma_1, ... , \gamma_U$ is denoted by $\bar{\gamma}^2$ and satisfies $1/\textcolor{black}{\bar{\gamma}^2} := U^{-1} \sum_{h = 1}^U 1/\textcolor{black}{\gamma_h}^2$, and we set $S_U := \sum_{h = 1}^U \textcolor{black}{\mathcal{U}_h}$. \color{black}
 Then, 
	\begin{displaymath}
		\P\bigg(\frac{|S_U|}{\sqrt{U}} \geq \lambda \bigg) \leq \left\{
		\begin{array}{rcl}
			2 e^{-\lambda^2 \textcolor{black}{\bar{\gamma}^2}/8}  & \text{if} & 0 \le  \lambda \le 2 \gamma_{\max} \sqrt{U}/\textcolor{black}{\bar{\gamma}}^2 \\
			\textcolor{black}{2}e^{-\textcolor{black}{\gamma_{\max}} \lambda \sqrt{U}/4} & \text{if} &  \lambda \ge2 \gamma_{\max} \sqrt{U}/\bar{\gamma}^2,
		\end{array}\right.   
	\end{displaymath}
	where we have also introduced $\textcolor{black}{\gamma_{\max}} := \max_{h = 1, \dots , U} \gamma_h$.
\end{lemma}
{\modar We apply this lemma to \eqref{eq: normalized laplace 7}. To this end, we define the sequence $(\gamma_h)_{h=1,\dots,nN}$ such that for all $h\in\id{1}{nN}$ that writes $h=j+kN$ with $j\in\id{1}{n}$ and $k\in\id{0}{N-1}$, we have $\gamma_h=\alpha_j$ and also set $\mathcal{U}_h=\overset{\circ}{\mathcal{E}}^{i,\ell,(k)}_j$.
Then, for any fixed $k\in\id{0}{a},\ell\in\id{0}{L_n {\modch -1}}$, we have
\begin{equation}
	\sum_{i=1}^N\sum_{j=1}^n\mathcal{E}_{j}^{i,\ell,(k)}
	= 2 \sqrt{T}(a+1)\log(n) L_n \sqrt{N} \frac{1}{\sqrt{U}} \sum_{h = 1}^U \mathcal{U}_{h}^{\ell,(k)}, \nonumber 
\end{equation}
where we have introduced $U:= N n$.	
}

\noindent Using the definition of the $\gamma_h$'s, 
we get
\begin{equation*}
\frac{1}{\bar{\alpha}_2} = \frac{1}{n} \sum_{j = 1}^n \frac{1}{\alpha_j^2} = \frac{1}{N n} \sum_{i = 1}^N \sum_{j = 1}^n \frac{1}{\alpha_j^2} = \frac{1}{U} \sum_{h = 1}^U \frac{1}{\gamma_h^2}=\frac{1}{\bar{\gamma}^2}.
\end{equation*}
Noting that we also have $\alpha_{\max}=\gamma_{\max}$, the application of Lemma \ref{l: Petrov}  directly provides 
\begin{align} \nonumber
\P\bigg(\frac{1}{\textcolor{black}{\sqrt{Nn}}}\sum_{i=1}^N\sum_{j=1}^{n}\overset{\circ}{\mathcal{E}}_{j}^{i,\ell,(k)} \geq \lambda \bigg)
&=\P\bigg(\frac{1}{\textcolor{black}{\sqrt{U}}}\sum_{h=1}^{U}\mathcal{U}_{h}^{\ell,(k)} \geq \lambda \bigg)
 \\ {\label{eq: end Petrov}}
 &\leq \left\{
\begin{array}{rcl}
   2 e^{-\lambda^2 \bar{\alpha}_2 /8}  & \text{if} & 0 \le  \lambda \le \frac{2 \alpha_{\max} \sqrt{N n}}{\bar{\alpha}_2} \\
   {\modch 2} e^{-\alpha_{\max}  \lambda \sqrt{Nn}/4} & \text{if} &  \lambda \ge \frac{{\modch 2}\alpha_{\max} \sqrt{N n}}{\bar{\alpha}_2}.
\end{array}\right.   
\end{align}
Let us now introduce a threshold $M_n$ that will be defined later by \eqref{Eq: Mn} {\modar and satisfies $\overline{\alpha}_2 M_n^2 \to \infty$}. 
Then,
\begin{align*}
&\E\bigg[\sup_{\ell,k}\Big|\frac{1}{\sqrt{Nn}}\sum_{i=1}^N\sum_{j=1}^{n}\overset{\circ}{\mathcal{E}}_{j}^{i,\ell,(k)}\Big|^p\bigg] \\
&=\E\Bigg[\sup_{\ell,k}\Big|\frac{1}{\sqrt{Nn}}\sum_{i=1}^N\sum_{j=1}^{n}\overset{\circ}{\mathcal{E}}_{j}^{i,\ell,(k)}\Big|^p\mathbf 1_{\big\{\sup_{\ell,k}|\frac{1}{\sqrt{Nn}}\sum_{i=1}^N\sum_{j=1}^{n}\overset{\circ}{\mathcal{E}}_{j}^{i,\ell,(k)}|< M_n\big\}}\Bigg]\\
& \qquad +\E\Bigg[\sup_{\ell,k}\Big|\frac{1}{\sqrt{Nn}}\sum_{i=1}^N\sum_{j=1}^{n}\overset{\circ}{\mathcal{E}}_{j}^{i,\ell,(k)}\Big|^p\mathbf 1_{\big\{\sup_{\ell,k}|\frac{1}{\sqrt{Nn}}\sum_{i=1}^N\sum_{j=1}^{n}\overset{\circ}{\mathcal{E}}_{j}^{i,\ell,(k)}|\geq M_n\big\}}\Bigg]\\
&\leq M_n^p+\E\Bigg[\sup_{\ell,k}\Big|\frac{1}{ \sup_{\ell,k}\sqrt{Nn}}\sum_{i=1}^N\sum_{j=1}^{n}\overset{\circ}{\mathcal{E}}_{j}^{i,\ell,(k)}\Big|^p\mathbf 1_{\big\{|\frac{1}{\sqrt{Nn}}\sum_{i=1}^N\sum_{j=1}^{n}\overset{\circ}{\mathcal{E}}_{j}^{i,\ell,(k)}|\geq M_n\big\}}\Bigg] \\
&\leq M_n^p+\textcolor{black}{p}\int_{M_n}^{\infty}\lambda^{p-1}\P\bigg(\sup_{\ell,k}\Big|\frac{1}{\sqrt{Nn}}\sum_{i=1}^N\sum_{j=1}^{n}\overset{\circ}{\mathcal{E}}_{j}^{i,\ell,(k)}\Big|\geq\lambda\bigg)d\lambda.
\end{align*}
We now use that 
$$\mathbb{P} \bigg( \sup_{\ell \in \{0, \dots , L_n - 1\}, \, k \in \{0, \dots , a \} } \Big|\frac{1}{\sqrt{Nn}}\sum_{i=1}^N\sum_{j=1}^{n}\overset{\circ}{\mathcal{E}}_{j}^{i,\ell,(k)}\Big| \geq\lambda \Big) \le \sum_{\ell = 0}^{L_n - 1} \sum_{k = 0}^a \mathbb{P} \Big(  \Big|\frac{1}{\sqrt{Nn}}\sum_{i=1}^N\sum_{j=1}^{n}\overset{\circ}{\mathcal{E}}_{j}^{i,\ell,(k)}\Big|\geq\lambda \bigg). $$
From \eqref{eq: end Petrov} follows
\begin{align*}
\E\Bigg[\sup_{\ell,k}\Big|\frac{1}{\sqrt{Nn}}\sum_{i=1}^N\sum_{j=1}^{n}\overset{\circ}{\mathcal{E}}_{j}^{i,\ell,(k)}\Big|^p\Bigg] 
& \le M_n^p+ \sum_{\ell = 0}^{L_n - 1} \sum_{k = 0}^a \int_{M_n}^{\infty}\textcolor{black}{p}\lambda^{p-1}\P\bigg(\frac{1}{\sqrt{Nn}}\sum_{i=1}^N\sum_{j=1}^{n}\overset{\circ}{\mathcal{E}}_{j}^{i,\ell,(k)}\geq\lambda\bigg)d\lambda\\
&\leq M_n^p+ \textcolor{black}{p} L_n(a +1) \int_{M_n}^{\frac{2 \alpha_{\max} \sqrt{N n}}{\bar{\alpha}_2}}2\lambda^{p-1}e^{-\lambda^2 \bar{\alpha}_2/8}d\lambda \\
& \qquad + \textcolor{black}{p}L_n(a +1)\int_{\frac{2 \alpha_{\max} \sqrt{Nn}}{\bar{\alpha}_2}}^\infty2\lambda^{p-1}e^{-(\lambda \alpha_{\max} \sqrt{Nn})/4}d\lambda,
\end{align*}
{\modar where the first integral is zero if $M_n > \frac{2 \alpha_{\max} \sqrt{N n}}{\bar{\alpha}_2}$.}
Now remark that {\modar this} first integral is bounded by 
\begin{equation*}
  \int_{M_n}^{\infty}2\lambda^{p-1}e^{-\lambda^2 \bar{\alpha}_2/8}d\lambda {\le c M_n^{{\modar p}} e^{- M_n^2 \bar{\alpha}_2/8}.}  
\end{equation*}

%
Regarding the second integral, we apply the change of variable $\lambda' := \lambda \alpha_{\max} \sqrt{Nn}$, that provides it is equal to 
$$(\alpha_{\max} \sqrt{Nn})^{- p} \int_{\frac{2 \alpha_{\max}^2 {Nn}}{\bar{\alpha}_2}}^\infty 2(\lambda')^{p-1}e^{-\lambda'/4}d\lambda' {\le c}\Big((\alpha_{\max} \sqrt{Nn})^{- p} \Big(\frac{ \alpha_{\max}^2 {Nn}}{\bar{\alpha}_2}\Big)^{p-1} e^{-\frac{ \alpha_{\max}^2 {Nn}}{4 \bar{\alpha}_2}}\Big).$$
It yields 
\begin{align*}
\E\Bigg[\sup_{\ell,k}\Big|\frac{1}{\sqrt{Nn}}\sum_{i=1}^N\sum_{j=1}^{n}\overset{\circ}{\mathcal{E}}_{j}^{i,\ell,(k)}\Big|^p\Bigg] 
& \le M_n^p + c L_n M_n^{{\modar p}} e^{- \frac{M_n^2 \bar{\alpha}_2}{8}} + 
c \frac{L_n}{(\alpha_{\max}^2 Nn)^{p/2}} 
\big(\frac{\alpha_{\max}Nn}{\bar{\alpha}_2}\big)^{p-1} e^{-\frac{ \alpha_{\max}^2 {Nn}}{4 \bar{\alpha}_2}}.  
\end{align*}
Comparing the first two terms in the right hand side of the equation above leads us to the choice 
\begin{equation}\label{Eq: Mn}
M_n^2 = \kappa \log(L_n)/\bar{\alpha}_2,
\end{equation}
for a constant $\kappa$ arbitrarily large. It implies
\begin{align}\label{eq: norme p sum Laplace renormalisee}
\E\Bigg[\sup_{\ell,k}\Big|\frac{1}{\sqrt{Nn}}\sum_{i=1}^N\sum_{j=1}^{n}\overset{\circ}{\mathcal{E}}_{j}^{i,\ell,(k)}\Big|^p\Bigg] & \le c\Big(\frac{\log(L_n)}{\bar{\alpha}_2} \Big)^{\frac{p}{2}}, 
\end{align}
as the the second and third terms are negligible compared to the first 
using $L_n=o(n^r)$ for some $r>0$. \textcolor{black}{Combining \eqref{eq: start H E 6} and \textcolor{black}{\eqref{eq: norme p sum Laplace renormalisee} and recalling the definition \eqref{eq: norm E 7} of the variables $\overset{\circ}{\mathcal{E}}_{j}^{i,\ell,(k)}$}, we deduce}
\begin{equation}{\label{eq: end HE 8}}
\E\Bigg[ \sup_{{\modar \theta}} \Big|\frac{\partial^u}{\partial \theta^u}|H({\bm{\mathcal{E}}})(\theta)|\Big|^p\Bigg]\leq c L_n^{up} (\sqrt{N} \log(n) L_n)^p \Big(\frac{\log(L_n)}{\bar{\alpha}_2} \Big)^{\frac{p}{2}}.   
\end{equation}
Replacing \eqref{eq: end H(beta) 5.5} and \eqref{eq: end HE 8} in \eqref{eq: start step 2} we obtain
\begin{align} \nonumber
 &\left \| \sup_{{\modar \theta}} |\frac{\partial^u}{\partial \theta^u}(H S_n^{N, 0}(\theta)-{\modar S^{N,\text{pub}}_n}(\theta))| \right \|_p \\
 &\label{E: step 2 prop almost finished}
 \le c L_n^{u + 1} \sqrt{N} \log(n) \sqrt{\frac{\log(L_n)}{\bar{\alpha}_2}} + c n^{- {\modar r_0/2}}\bigg[L_n^{u + 1} \sqrt{N} \log(n) \sqrt{\frac{\log(L_n)}{\bar{\alpha}_2}} + L_n^{u} n^{\frac{1}{2}}N\bigg],
\end{align}
having used Cauchy-Schwarz inequality and the fact that 
{\modar \begin{equation}\label{E: majo comple Omega n}
	\mathbb{P}(\Omega_n^c) \le cn^{- r_0}
	\end{equation}} 
for any $r_0 \ge 2$. {\modar 
	Indeed, thanks to  Lemma \ref{l: bound proba},}
\begin{align*}
\mathbb{P}(\Omega_n^c) & = \mathbb{P}\big(\exists \bar{k}, \bar{i}, \bar{\ell}, \bar{j}: \, |f^{\bar{i},\bar{\ell},(\bar{k})}_{\bar{j}}| > \tau_n \big) \\
& = \E\Big[ \E_{t_{\bar{j}\textcolor{black}{-1}}}\Big[\mathbf 1_{ |f^{\bar{i},\bar{\ell},(\bar{k})}_{\bar{j}}| > \tau_n}\Big]\Big] \\
& \le c \E\Big[\frac{\Delta_n^{\frac{r_0}{2}}}{(\log(n))} R_{t_{\bar{j}\textcolor{black}{-1}}}(1) + \exp(- c(\log(n))^2)\Big], 
\end{align*}
that {\modar goes to zero  at a rate given by any arbitrarily large exponent of $1/n$, recalling that $\Delta_n=T/{\modch n}$, with fixed $T$.} 
\\
It concludes the proof of {\modar Step 2, as \eqref{E: contrast error part2} is an immediate consequence of \eqref{E: step 2 prop almost finished}}. The proposition is therefore proven.
\end{proof}

We are now ready to prove the consistency of the proposed estimator based on its analogous result in the privacy-free context and on the approximation argument presented in Proposition \ref{prop: approx contrast}.

\begin{proof}[Proof of Theorem \ref{th: consistency}] 
In order to show the consistency we will prove that 
\begin{equation}{\label{eq: conv contrast}}
\frac{1}{N}({\modar S^{N,\text{pub}}_n}(\theta) - {\modar S^{N,\text{pub}}_n}(\theta^{\star})) \xrightarrow{\mathbb{P}} {\modar -} \int_0^T\E\bigg[\frac{\big(b(\theta,X_s)-b(\theta^{\star},X_s)\big)^2}{\sigma^2(X_s)}\bigg]ds =: C_\infty (\theta)
\end{equation}
uniformly in $\theta \in {\modar [\theta_0,\theta_{L_n-1}]=[0,1-1/L_n]}$. {\modar Let us stress that, as $\theta^\star \in(0,1)$, we have 
	$\theta^\star \in [0,1-1/L_n]$ for $n$ large enough.}  \\
Observe that one can write 
\begin{align}{\label{eq: pol drift}}
	\frac{1}{N}({\modar S^{N,\text{pub}}_n}(\theta) - {\modar S^{N,\text{pub}}_n}(\theta^{\star})) = 
\frac{1}{N}({\modar S^{N,\text{pub}}_n}(\theta) - S_n^{N, 0}(\theta)) - \frac{1}{N}({\modar S^{N,\text{pub}}_n}(\theta^{\star}) - S_n^{N,0}(\theta^{\star})) \\+ \frac{1}{N} (S_n^{N,0}(\theta) - S_n^{N,0}(\theta^{\star})).
\end{align}
Proposition \ref{prop: approx contrast} with $u = 0$ ensures that the $L^p$ norm of the first two terms here above is upper bounded, uniformly in $\theta$, by $c (\frac{1}{L_n})^{a - 1} + c \frac{L_n \log(n)}{\sqrt{N}} \sqrt{\frac{\log(L_n)}{\bar{\alpha}_2}} {\modar + c\frac{N}{n^{r_0}}}$, {\modar for any fixed $r_0 >0$. This term goes to $0$ under} our hypothesis. It follows they converge to $0$ in probability as well. Then, it is enough to prove 
\begin{equation}{\label{eq: conv consistency}}
 \frac{1}{N}(S_n^{N,0}(\theta) - S_n^{N,0}(\theta^{\star})) \xrightarrow{\mathbb{P}} C_\infty (\theta)   
\end{equation}
uniformly in $\theta$ to obtain \eqref{eq: conv contrast}. The derivation of \eqref{eq: conv consistency} closely mirrors the argumentation found in Steps 3 and 4 of Lemma 6.1 in \cite{McKean}. It is worth noting that our contrast function $S_n^{N,0}$ differs slightly from $S_n^N$ in \cite{McKean}. In that paper, indeed, the authors aim to jointly estimate both the drift and the diffusion coefficient in a parametric manner, whereas our focus here is solely on estimating the drift. However, similarly as in Step 3 of Lemma 6.1 in \cite{McKean}, {\modar and using \eqref{eq: def contrast without privacy}--\eqref{E: def f apres},} the following decomposition holds:
$$ \frac{1}{N}(S_n^{N,0}(\theta) - S_n^{N,0}(\theta^{\star})) = I_n^N(\theta) + 2 \rho_n^N(\theta),$$
where 
\begin{align*}
 I_n^N(\theta) &= {\modar -}\frac{\Delta_n}{N} \sum_{i = 1}^N \sum_{j = 1}^n \frac{(b(\theta, X_{t_j-1}^i)- b(\theta^{\star}, X_{t_j-1}^i))^2}{{\modch \sigma}^2(X_{t_j-1}^i)}, \\
 \rho^N_n(\theta) & = \frac{1}{N} \sum_{i = 1}^N \sum_{j = 1}^n (X_{t_j}^i - X_{t_j-1}^i - \Delta_n b(\theta^{\star}, X_{t_j-1}^i) ) \frac{b(\theta, X_{t_j-1}^i)- b(\theta^{\star}, X_{t_j-1}^i)}{{\modch \sigma}^2(X_{t_j-1}^i)}.
\end{align*}
Then, $I_n^N(\theta) \xrightarrow{\mathbb{P}} C_\infty(\theta)$ because of Lemma \ref{l: asymptotic 2.5} while $\rho^N_n(\theta) \xrightarrow{\mathbb{P}} 0$ uniformly in $\theta$ as in the proof of Lemma 6.1 in \cite{McKean}. Moreover, Step 4 of Lemma 6.1 in \cite{McKean} ensures the tightness of the two sequences $\theta \mapsto I_n^N(\theta)$ and $\theta \mapsto \rho_n^N(\theta)$, which concludes the proof of \eqref{eq: conv contrast} and implies in particular that 
\begin{equation}{\label{eq: main consistency}}
\sup_{\theta \in {\modar [\theta_0,\theta_{L_n-1}]}} \Big|\frac{1}{N} ({\modar S^{N,\text{pub}}_n}(\theta) - {\modar S^{N,\text{pub}}_n} (\theta^{\star})) - C_\infty(\theta)\Big| \xrightarrow{\mathbb{P}} 0.
\end{equation}
Such equation implies the consistency of $\hat{\theta}_n^N$. Indeed, the identifiability condition stated in Assumption \ref{Ass_4} implies that, for every $\epsilon > 0$, there exists $\eta > 0$ such that $C_\infty(\theta) {\modar <  - \eta}$ for every $\theta$ such that $|\theta - \theta^{\star}| \ge \epsilon$. Then, 
$$\left \{ |\hat{\theta}^N_n - \theta^{\star}| \ge \epsilon \right \} \subset \left \{ C_\infty(\hat{\theta}^N_n) {\modar < - \eta} \right \}.$$
Observe that we can write 
$$C_\infty(\hat{\theta}^N_n) =  \Big( C_\infty(\hat{\theta}^N_n) - \frac{1}{N}({\modar S^{N,\text{pub}}_n}(\hat{\theta}^N_n) - {\modar S^{N,\text{pub}}_n}(\theta^{\star})) \Big) + \frac{1}{N}({\modar S^{N,\text{pub}}_n}(\hat{\theta}^N_n) - {\modar S^{N,\text{pub}}_n}(\theta^{\star})).$$
Now, the first converges to $0$ in probability because of \eqref{eq: main consistency}, while the second is {\modar non-negative} because of the definition of $\hat{\theta}^N_n$. We derive that the probability of the event $\left \{ C_\infty(\hat{\theta}^N_n)  {\modar < - \eta } \right \}$ converges to $0$, which concludes the proof of the consistency.
\end{proof}

\subsection{Asymptotic normality}
In this section, we establish the asymptotic normality of the proposed estimator. The proof follows a classical path, relying on the asymptotic behavior of $\partial_\theta {\modch S^{N,\text{pub}}_n}(\theta)$ and $\partial^2_\theta {\modch S^{N,\text{pub}}_n}(\theta)$ (refer, for instance, to \cite{GenonCatalot_Jacod_1993}, Section 5a). \\
Particularly noteworthy is the observation that the behavior of the first derivative of the contrast function varies depending on whether the contribution of privacy is negligible in our estimation procedure (see Propositions \ref{prop: as norm without privacy} and \ref{prop: as norm privacy} below). Conversely, the behavior of the second derivative remains consistent, irrespective of the privacy's contribution, as elucidated in Proposition \ref{prop: second derivatives} below. The proofs of Propositions \ref{prop: as norm without privacy}, \ref{prop: as norm privacy}, and \ref{prop: second derivatives} can be found in the subsequent subsections.

\begin{prop}{[Negligible contribution of privacy] \\}{\label{prop: as norm without privacy}}
Assume that A\ref{Ass_1}- A\ref{Ass_3}, {\modar A\ref{Ass_splines}.\ref{H: Ass_splines priv neg} and A\ref{Ass_discretization}.\ref{H: Ass_discret priv neg} hold.} Assume moreover that ${\modch \sqrt{\log (L_n)} r_{n,N}} \rightarrow 0$ for $N,n, L_n \rightarrow \infty$. Then, 
$$\frac{1}{\sqrt{N}} \partial_\theta {\modch S^{N,\mathrm{pub}}_n}(\theta^{\star}) \xrightarrow{\mathcal{L}} \textcolor{black}{\mathcal N}\bigg(0,2 \int_0^T \E\bigg[\Big(\frac{\partial_\theta b(\theta^{\star}, X_s)}{{\modch \sigma}(X_s)}\Big)^2\bigg] ds\bigg) = : \textcolor{black}{\mathcal N}(0,2 \Sigma_0).$$   
\end{prop}

When the contribution of the privacy is significant, instead, the following proposition is in hold.

\begin{prop}{[Significant contribution of privacy] \\}{\label{prop: as norm privacy}}
Assume that A\ref{Ass_1}- A\ref{Ass_3},
 {\modar A\ref{Ass_splines}.\ref{H: Ass_splines priv dominate}, A\ref{Ass_discretization}.\ref{H: Ass_discret priv dominate} and A\ref{Ass_6} hold.}
Assume moreover that ${\modch r_{n,N}} \rightarrow \infty$ for $N, n, L_n \rightarrow \infty$.
Define
\begin{equation*}
N_{n,N}:=\frac{\sqrt{\bar{\alpha}_2}}{\sqrt{N}L_n^2 \log(n)}\frac{1}{4{\modar (a+1)}\sqrt{T}\sqrt{v_n(\theta^{\star})}}	\frac{\partial}{\partial \theta}{\modch S^{N,\mathrm{pub}}_n}(\theta^\star).
\end{equation*}	
Then, we have
\begin{equation*}
	(N_{n,N},v_n(\theta^{\star})) \xrightarrow{\mathcal{L}}(\textcolor{black}{\mathcal N},\overline{v}(s)) ,
\end{equation*}
where \textcolor{black}{$\mathcal N$} is a Gaussian  $\mathcal{N}(0,1)$ variable independent of $S$. We recall that $\bar{\alpha}_2$ is such that $1/\bar{\alpha}_2 = n^{-1} \sum_{j = 1}^n 1/\alpha_j^2$,  $v_n(\theta^{\star})$ {\modhel and $\overline{v}$ have respectively been introduced in \eqref{eq: def vn} and \eqref{E: def hat v}}.
\end{prop}

One can remark that the definition of the estimator guarantees that $\partial_\theta {\modch S^{N,\text{pub}}_n}(\hat{\theta}^N_n) =0$. Thanks to Taylor's formula, it implies 
\begin{equation}{\label{eq: as norm start 3}}
 (\hat{\theta}^N_n - \theta^{\star}) \int_0^1 \partial^2_\theta {\modch S^{N,\text{pub}}_n}(\theta^{\star} + s (\hat{\theta}^N_n - \theta^{\star})) ds = - \partial_\theta {\modch S^{N,\text{pub}}_n}(\theta^{\star}).  
\end{equation}
The asymptotic behaviour of the second derivative of ${\modch S^{N,\text{pub}}_n}$ is gathered in the following proposition.
\begin{prop}{\label{prop: second derivatives}}
Suppose that Assumptions A\ref{Ass_1}- A\ref{Ass_4} are in hold. Assume moreover that $a > 3$ and ${\modar L_n^3} \log(n) \sqrt{\frac{\log(L_n)}{ N \bar{\alpha}_2}} \rightarrow 0$. Then, for any $N, n \rightarrow \infty$ we have 
\begin{enumerate}
    \item $\frac{\partial^2_\theta {\modch S^{N,\mathrm{pub}}_n}(\theta^{\star})}{N} \xrightarrow{\mathbb{P}} \Sigma_0$,
    \item $\frac{1}{N} \sup_{s \le 1} |\partial^2_\theta{\modch S^{N,\mathrm{pub}}_n}(\theta^{\star} + s (\hat{\theta}^N_n - \theta^{\star})) - \partial^2_\theta {\modch S^{N,\mathrm{pub}}_n}(\theta^{\star})| \xrightarrow{\mathbb{P}} 0$.
\end{enumerate}
\end{prop}
From the propositions stated above we easily deduce the asymptotic normality as in Theorems \ref{th: as norm privacy negl} and \ref{th: as norm}. 

\begin{proof}[Proof of Theorem \ref{th: as norm privacy negl}]
In the case of negligible privacy contribution it looks convenient to write \eqref{eq: as norm start 3} as
\begin{displaymath}
(\hat{\theta}^N_n - \theta^{\star}) = - \frac{\sqrt{\frac{1}{N}} \partial_\theta {\modch S^{N,\text{pub}}_n}(\theta^{\star})}{\frac{1}{N} \int_0^1 \partial^2_\theta {\modch S^{N,\text{pub}}_n}(\theta^{\star} + s (\hat{\theta}^N_n - \theta^{\star})) ds } \quad \frac{1}{\sqrt{N}}.
\end{displaymath}
Then, thanks to Propositions \ref{prop: as norm without privacy} and \ref{prop: second derivatives}, Slutsky's theorem and the continuous mapping theorem, 
$$\sqrt{N } (\hat{\theta}^N_n - \theta^{\star}) \xrightarrow[]{\mathcal{L}} \mathcal N\big(0, 2(\Sigma_0)^{-1}\big),$$
as in the statement of Theorem \ref{th: as norm privacy negl}.   
\end{proof}

\begin{proof}[Proof of Theorem \ref{th: as norm}]
In the case of significant privacy contribution we can write \eqref{eq: as norm start 3} as 
\begin{displaymath}
(\hat{\theta}^N_n - \theta^{\star}) = - \frac{\sqrt{\frac{\bar{\alpha}_2}{N}} \frac{1}{L_n^2 \log(n) 4{\modar(a+1)}\sqrt{T}\sqrt{v_n(\theta^{\star})}} \partial_\theta {\modch S^{N,\text{pub}}_n}(\theta^{\star})}{\frac{1}{N} \int_0^1 \partial^2_\theta {\modch S^{N,\text{pub}}_n}(\theta^{\star} + s (\hat{\theta}^N_n - \theta^{\star})) ds } \quad \frac{1}{\sqrt{N \bar{\alpha}_2}} L_n^2 \log(n) 4{\modar(a+1)}\sqrt{T}\sqrt{v_n(\theta^{\star})}.
\end{displaymath}
From here we can deduce, thanks to Propositions \ref{prop: as norm privacy} and \ref{prop: second derivatives}, Slutsky's theorem and the continuous mapping theorem,
$$\frac{\sqrt{N \bar{\alpha}_2}}{L_n^2 \log(n) 4{\modar(a+1)}\sqrt{T}\sqrt{v_n(\theta^{\star})}} (\hat{\theta}^N_n - \theta^{\star}) \xrightarrow{\mathcal{L}} \mathcal N\big(0, (\Sigma_0)^{-2}\big),$$
jointly with the convergence $v_n(\theta^\star) \xrightarrow{\mathcal{L}} \overline{v}(s)$. This concludes the proof of Theorem \ref{th: as norm}.
\end{proof}

The proof of asymptotic normality is concluded by establishing the validity of the propositions stated above. We will begin with the case where privacy is negligible, as this will serve as a foundation for the more intricate proof of the significant privacy case.

\subsubsection{Proof of Proposition \ref{prop: as norm without privacy}}
\begin{proof}
The proof of Proposition \ref{prop: as norm without privacy} heavily relies on the approximation gathered in Proposition \ref{prop: approx contrast}.
We have indeed
$$\frac{1}{\sqrt{N}}\partial_\theta {\modch S^{N,\text{pub}}_n}(\theta^{\star}) = \frac{1}{\sqrt{N}}(\partial_\theta {\modch S^{N,\text{pub}}_n}(\theta^{\star}) - \partial_\theta S_n^{N,0}(\theta^{\star})) + \frac{1}{\sqrt{N}}\partial_\theta S_n^{N,0}(\theta^{\star}).  $$
Proposition \ref{prop: approx contrast} ensures that
\begin{equation}{\label{eq: pol as norm negl}}
\left \| \frac{1}{\sqrt{N}}(\partial_\theta {\modch S^{N,\text{pub}}_n}(\theta^{\star}) - \partial_\theta S_n^{N,0}(\theta^{\star})) \right \|_p \le c \frac{\sqrt{N}}{L_n^{a - 2}} + c L_n^2 \log(n) \sqrt{\frac{\log(L_n)}{\bar{\alpha}_2}} {\modar + c\frac{L_n \sqrt{N}}{n^r}},
\end{equation}
where $r>0$ can be chosen arbitrarily large. It goes to zero under
the hypothesis A\ref{Ass_splines}.\ref{H: Ass_splines priv neg},
	$r_{n,N} \sqrt{\log(L_n)} \to 0$ {\modchi and the fact that both $L_n$ and $N$ go to $\infty$ at a polynomial rate in $n$.}

Then, Proposition \ref{prop: as norm without privacy} is proven once we show that $\frac{1}{\sqrt{N}} \partial_\theta S_n^{N,0}(\theta^{\star}) \xrightarrow{\mathcal{L}} \mathcal N(0, 2 \Sigma_0)$. One can easily check this is implied by Proposition 6.2 in \cite{McKean}. Indeed, similarly as in \cite{McKean}, we can write ${\modar \frac{1}{\sqrt{N}}} \partial_\theta S_n^{N,0}(\theta) = \sum_{j = 1}^n \xi_j^{(1)}(\theta)$, with 
\begin{equation} \label{E: def xi j 1}
	\xi_j^{(1)}(\theta) := \frac{1}{\sqrt{N}} \sum_{i = 1}^N \frac{2 \partial_\theta b(\theta, X^i_{t_{j-1}})}{{\modch \sigma}^2(X^i_{t_{j-1}})}(X^i_{t_{j}}- X^i_{t_{j-1}} - \Delta_n b(\theta, X^i_{t_{j-1}})).
\end{equation}
Remark that the main difference compared to $\xi_{j,h}^{(1)}(\theta)$ in Proposition 6.2 in \cite{McKean} is that in our case the diffusion coefficient ${\modch \sigma}$ does not depend on a second parameter $\theta_2$ and the drift parameter is no longer in $\R^{p_1}$, but simply in $\R$. Moreover, the convergence gathered in (37) in Proposition 6.2 of \cite{McKean} is now replaced by 
\begin{equation*}
\sum_{j = 1}^n \E_{t_{j - 1}}[(\xi_j^{(1)}(\theta^{\star}))^2] \xrightarrow{\mathbb{P}} 4 \int_0^T \E\bigg[\bigg(\frac{\partial_\theta b(\theta^{\star}, X_s)}{{\modch \sigma}(X_s)}\bigg)^2\bigg] ds,   
\end{equation*}
{\modar that can be obtained using Lemma \ref{l: asymptotic 2.5} and $N\Delta_n \to 0$ by Assumption A\ref{Ass_discretization}.\ref{H: Ass_discret priv neg}.}\\
Then, (36) and (40) in Proposition 6.2 of \cite{McKean} provide, for some $\tilde{r} > 0$, 
\begin{equation*}
 \sum_{j = 1}^n | \E_{t_{j - 1}}[\xi_j^{(1)}(\theta^{\star})] |\xrightarrow{\mathbb{P}} 0, \qquad \sum_{j = 1}^n \E_{t_{j - 1}}[(\xi_j^{(1)}(\theta^{\star}))^{2 + \tilde{r}}] \xrightarrow{\mathbb{P}} 0.   
\end{equation*}
The proof of Proposition \ref{prop: as norm without privacy} is therefore concluded by application of Theorem 3.2 in \cite{Hall_Heyde}. 
\end{proof}

\subsubsection{Proof of Proposition \ref{prop: as norm privacy}}
\begin{proof}
We start by observing that, as a consequence of Lemma \ref{Lem_linearity_spline}, for any $i \in \{1, \dots , N \}$ and $j \in \{1, \dots , n \}$ it is
{\modar 
\begin{equation*}
	H(Z_j^i) = H \left(  (f_j^{i,\ell,(k)} \varphi_{\tau_n}(f_j^{i,\ell,(k)}))_{\ell,k}\right)
	+ H\left((\mathcal{E}_{j}^{i,\ell,(k)})_{\ell,k}\right)
\end{equation*}
where, as a reminder, $\mathcal{E}_{j}^{i,\ell,(k)}$  and $f_j^{i,\ell,(k)}$ are defined as in \eqref{E: law Epsilon} and \eqref{E: def f_ijlk}.}
%
%
Then, {\modar from \eqref{Eq_definition_wtSN}} one has 
{\modar
	\begin{align*}
S_n^{N,\text{pub}}(\theta^\star) & = \sum_{i = 1}^N \sum_{j = 1}^n  H \left(  (f_j^{i,\ell,(k)} \varphi_{\tau_n}(f_j^{i,\ell,(k)}))_{\ell,k}\right)(\theta^{\star}) +
 \sum_{i = 1}^N \sum_{j = 1}^nH\left((\mathcal{E}_{j}^{i,\ell,(k)})_{\ell,k}\right) (\theta^{\star}) \\
& =: H (\bm{\gamma}){\modchi (\theta^{\star}) } + H ( \bm{\mathcal{E}}) (\theta^{\star}),
\end{align*}
where 
	$\bm{\gamma}=(\bm{\gamma}^{\ell,k})_{\ell,k}$ and $\bm{\mathcal{E}}=(\bm{\mathcal{E}}^{\ell,(k)})_{\ell,k}$ are given by
	\begin{gather} \label{E:_def gamma k l}
		\bm{{\gamma}}^{\ell,(k)}:= \sum_{i=1}^N\sum_{j=1}^n
		 f_j^{i,\ell,(k)} \varphi_{\tau_n}(f_j^{i,\ell,(k)}),
		\\ \label{Eq_def_bm mathcal E}
		\bm{\mathcal{E}}^{\ell,(k)}:= \sum_{i=1}^N\sum_{j=1}^n \mathcal{E}^{j,\ell,(k)}_i.
	\end{gather}
We will see that the main contribution in 
$\partial_\theta S_n^{N,\text{pub}}(\theta^\star )$ comes 
from $\partial_\theta H(\bm{\mathcal{E}}) (\theta^{\star})$.}
%
{\modar By  Proposition \ref{Prop_Hermite_interpolation_def}, the Hermite interpolation of the Laplace noise term is given by}
	\begin{equation*}
		H(\bm{\mathcal{E}})(\theta)=\sum_{\ell=-1}^{{\modar L_n-2}} \sum_{m=0}^a  c_{\ell}^{m}(\bm{\mathcal{E}}) B_j^m(\theta),
	\end{equation*}
	where the coefficients in the spline decomposition are given by 
	\begin{equation}\label{eq: coeff spline H E}
	c_{\ell}^{m}( \bm{\mathcal{E}} )= \sum_{\nu=0}^a (-1)^\nu (g_\ell^m)^{(2a+1-\nu)}(\theta_{\ell+1})
	\bm{\mathcal{E}}^{\ell+1,(\nu)}, 
\end{equation} 
with the polynomial function $g_\ell^m(\theta)=\frac{1}{(2a+1)!}(\theta-\theta_\ell)^{a-m}(\theta-\theta_{\ell+1})^{a+1}(\theta-\theta_{\ell+2})^m$. The spline function $B_j^m$ is supported on $[\theta_j,\theta_{j+2}]$, where we recall that we have set $\theta_{-1}=\theta_0$ and {\modar $\theta_{L_n}=\theta_{L_n-1}$,} by repeating the two endpoints.

As the spline expansion is constructed upon a high frequency grid, the dominating term in the expression of the coefficients \eqref{eq: coeff spline H E} is the one corresponding to $\nu=0$. Hence, we isolate such contribution.
Remarking that the term corresponding to $\nu=0$ in {\modar the sum \eqref{eq: coeff spline H E} is equal to 
$\bm{\mathcal{E}}^{\ell+1,(0)}$, as $(g^{m}_\ell)^{(2a+1)} \equiv 1$, this leads} us to split $H(\bm{\mathcal{E}})$ into the sum
$H(\bm{\mathcal{E}})=H^0(\bm{\mathcal{E}})+{\modarn \overline{H}}(\bm{\mathcal{E}})$
where the two functions $H^0(\bm{\mathcal{E}})$ and ${\modarn \overline{H}}(\bm{\mathcal{E}})$ are defined by 
\begin{align*}
	H^0({\bm{\mathcal{E}}})&=\sum_{\ell=-1}^{{\modar L_n-2}} \sum_{m=0}^a \bm{\mathcal{E}}^{\ell+1,(0)} B_\ell^m
	\\
	{\modarn \overline{H}}({\bm{\mathcal{E}}})&=\sum_{\ell=-1}^{{\modar L_n-2}} \sum_{m=0}^a \overline{c}_{\ell}^{m}({\bm{\mathcal{E}}}) B_\ell^m,	
\end{align*} 
where
\begin{equation} \label{eq: def overline c}
		\overline{c}_{\ell}^{m}({\bm{\mathcal{E}}})= \sum_{\nu=1}^a (-1)^\nu (g_\ell^m)^{(2a+1-\nu)}(\theta_{\ell+1})
\bm{\mathcal{E}}^{\ell+1,(\nu)}.
\end{equation}
\noindent
The \textcolor{black}{proposition} will be proved if we show
\begin{gather} \label{eq: conv rescaled main term in partial H}
 \left( \frac{\sqrt{\bar{\alpha}_2}}{\sqrt{N}L_n^2 \log(n)}\frac{1}{4{\modar (a+1)}\sqrt{T}\sqrt{v_n(\theta^{\star})}}\frac{\partial}{\partial \theta}H^0(\bm{\mathcal{E}})(\theta^\star), v_n(\theta^\star)
\right)
\xrightarrow{\mathcal{L}}({\modch \mathcal{N}},\overline{v}(S)),
\\ \label{eq: conv rescaled neg term in partial H}
 \frac{\sqrt{\bar{\alpha}_2}}{\sqrt{N}L_n^2 \log(n)}\frac{1}{\sqrt{v_n(\theta^{\star})}}	\frac{\partial}{\partial \theta}{\modarn \overline{H}}(\bm{\mathcal{E}})(\theta^\star) 
\xrightarrow{\mathbb{P}} 0,\\ \label{eq: conv rescaled neg term f in partial H}
\frac{\sqrt{\bar{\alpha}_2}}{\sqrt{N}L_n^2 \log(n)}\frac{1}{\sqrt{v_n(\theta^{\star})}}	\frac{\partial}{\partial \theta}H({\modar \bm{\gamma}})(\theta^\star) 
\xrightarrow{\mathbb{P}} 0.
\end{gather}

We now study the asymptotic behaviour of $\frac{\partial}{\partial \theta} H^0(\bm{\mathcal{E}})(\theta^{\star})$. We recall that $\ell^{\star}_n \in \{0,\dots,{\modar L_{n}-2}\}$ denotes the value of the index such that $\theta^\star \in [\theta_{\ell^\star_n},\theta_{\ell^\star_n+1})$. Since $\theta^\star \in (0,1)$, it is possible to exclude the case ${\modar \ell^\star_n}=0$ or ${\modar \ell^\star_n=L_n-2}$ for $n$ large enough. Using that the support of $B^m_j$ is $[\theta_{j},\theta_{j+2}]$ we deduce, for $\theta \in [\theta_{\ell^{\star}_n},\theta_{\ell^{\star}_n+1})$,
\begin{equation} \label{eq: H^0 sur 2 spline}
		H^0({\bm{\mathcal{E}}})(\theta)=\sum_{\ell\in\{ \ell_n^{\star}-1,\ell_n^{\star}\}} \sum_{m=0}^a \bm{\mathcal{E}}^{\ell+1,(0)} B_\ell^m(\theta).
\end{equation}
We use the scaling properties of the spline functions, to represent $B_\ell^m(\theta)$ for $\ell \in \{1,\dots,{\modar L_n-3}\}$, in the following way
\begin{equation}\label{eq: spline rescaled on 0_2}
	B_\ell^m(\theta)=\overline B^m(L_n(\theta-\theta_\ell)),
\end{equation}
where, \textcolor{black}{for $m\in\id{0}{{\modar a}}$}, $\overline B^m$ is the spline function of order $2a+1$ constructed on the knots $0$, $1$, $2$ where the knot $0$ is repeated $a+1-m$ times, while the knot $1$ is repeated $a+1$ times and knot $2$ is repeated $m+1$ times.  The spline function $\overline B^m$ is supported on $[0,2)$.
By differentiating  \eqref{eq: H^0 sur 2 spline} and using the representation \eqref{eq: spline rescaled on 0_2}, we deduce
\begin{align*} 
	\frac{\partial}{\partial\theta}H^0({\bm{\mathcal{E}}})(\theta^{\star})&=\sum_{\ell\in\{ \ell_n^{\star}-1,\ell_n^{\star}\}} \sum_{m=0}^a \bm{\mathcal{E}}^{\ell+1,(0)} L_n \frac{\partial}{\partial \theta} {\modar \left( \overline B^m \right)}(L_n(\theta^\star-\theta_\ell))
	\\
	&= L_n \bm{\mathcal{E}}^{\ell^{\star}_n,(0)} \sum_{m=0}^a    \frac{\partial}{\partial \theta} {\modar \left( \overline B^m \right)}(L_n(\theta^\star-\theta_{\ell^\star_n-1})) +L_n \bm{\mathcal{E}}^{\ell^{\star}_n+1,(0)} \sum_{m=0}^a    \frac{\partial}{\partial \theta} 
	{\modar \left( \overline B^m \right)}(L_n(\theta^\star-\theta_{\ell^\star_n}))  
	\\
	&= L_n \bm{\mathcal{E}}^{\ell^{\star}_n,(0)} g\Big(L_n\Big(\theta^\star-\theta_{\ell^\star_n}+\frac{1}{L_n}\Big)\Big)
	+ L_n \bm{\mathcal{E}}^{\ell^{\star}_n+1,(0)} g(L_n(\theta^\star-\theta_{\ell^\star_n})),	
\end{align*}
where we have set $g(u)=\sum_{m=0}^a \frac{\partial}{\partial u} \overline B^m(u)$ for $u \in [0,2]$ {\modar and used 
$\theta_{\ell_n^\star-1}=\theta_{\ell_n^\star}-1/L_n$}. From Lemma \ref{l: spline g}, we know that $g(u)=(2a+1) \binom{2a}{a} \left[u^a(1-u)^a \mathbf{1}_{[0,1]}(u)- (u-1)^a(2-u)^a \mathbf{1}_{[1,2]}(u)\right]$. From the definition of {\modar $\overline{v}$,  given in \eqref{E: def hat v},} we deduce that
\begin{equation} \label{eq: derive contrast theta star Laplace}
		\frac{\partial}{\partial\theta}H^0({\bm{\mathcal{E}}})(\theta^\star)=
		{L_n} \overline{v}(L_n(\theta^\star-\theta_{\ell_n^\star}) )^{1/2} \times 
	{\modar 	\left[ \bm{\mathcal{E}}^{\ell^{\star}_n+1,(0)} - \bm{\mathcal{E}}^{\ell^{\star}_n,(0)} \right].}
\end{equation}
Now we prove,
\begin{equation}\label{eq: cv somme Laplace}
	\forall \ell \in \{1,\dots,L_n-3\}, \quad
	{\modar N^\ell_{n,N}:=}\frac{\sqrt{\bar{\alpha}_2}}{\sqrt{N}L_n \log(n)}\frac{1}{4{\modar (a+1)}\sqrt{T}}	\left[ {\modar \bm{\mathcal{E}}^{\ell+1,(0)} -\bm{\mathcal{E}}^{\ell,(0)}} \right] \xrightarrow{\mathcal{L}} \mathcal{N}(0,1).
\end{equation}
We write 
\begin{equation*}
\bm{\mathcal{E}}^{\ell+1,(0)} -\bm{\mathcal{E}}^{\ell,(0)} = \sum_{j=1}^N \sum_{i=1}^n 
\Big(\mathcal{E}^{i,\ell+1,(0)}_j - \mathcal{E}^{i,\ell,(0)}_j\Big) = \frac{2  \sqrt{T}(a+1)\log(n) L_n}{\sqrt{n}}\sum_{j=1}^N \sum_{j=1}^n 
\Big(\overset{\circ}{\mathcal{E}}^{i,\ell+1,(0)}_j -\overset{\circ}{\mathcal{E}}^{i,\ell,(0)}_j\Big)
\end{equation*}
where we used \eqref{eq: norm E 7} and $\tau_n=\sqrt{\Delta_n} \log(n)=\frac{\sqrt{T}}{\sqrt{n}} \log(n)$. It yields,
\begin{equation}\label{eq: sum Laplace pour cv}
\frac{\sqrt{\bar{\alpha}_2}}{\sqrt{N}L_n \log(n)4{\modar (a+1)}\sqrt{T}}	
	\left[ \bm{\mathcal{E}}^{\ell+1,(0)} -\bm{\mathcal{E}}^{\ell,(0)} \right]
	=\sum_{j=1}^N \sum_{i=1}^n \frac{\sqrt{\bar{\alpha}_2}}{2\sqrt{Nn}} 
	\Big(\overset{\circ}{\mathcal{E}}^{i,\ell+1,(0)}_j -\overset{\circ}{\mathcal{E}}^{i,\ell,(0)}_j\Big). 
\end{equation}

For each fixed value of $\ell$, the sequence $\Big(\overset{\circ}{\mathcal{E}}^{i,\ell+1,(0)}_j -\overset{\circ}{\mathcal{E}}^{i,\ell,(0)}_j\Big)_{1\le i\le N,1 \le j \le n}$ is  {\modar constituted} of independent centered variables with variance depending on the index $j$ and equal to $4/\alpha_j^2$.  We deduce
from $\sum_{j=1}^n 1/\alpha_j^2 = n/\bar{\alpha}_2$ that the R.H.S of \eqref{eq: sum Laplace pour cv} has a unit variance. We apply a Central Limit Theorem for triangular array, as Theorem 18.1 of \cite{Billingsley_Book}, to deduce that the R.H.S. of \eqref{eq: sum Laplace pour cv} converges in law to a $\mathcal{N}(0,1)$. To apply this  Central Limit Theorem, it is necessary that the sequence satisfies a Lindeberg condition as given by equation (18.2) in  \cite{Billingsley_Book}. Here, the Lindeberg condition follows from the fact that the variances of all terms in the sum  \eqref{eq: sum Laplace pour cv} are comparable {\modar up to a constant,} using $1\le \frac{\sup_j \alpha_j}{\inf_j \alpha_j} =O(1)$. Consequently, the convergence
\eqref{eq: cv somme Laplace} is proved.

Recalling the definition of {\modar $N^\ell_{n,N}$} in \eqref{eq: cv somme Laplace}, the 
family of random variables {\modar {$(N^\ell_{n,N})_{n\ge 1, 1\le \ell \le L_{n}-3}$}} is such that for all fixed $\ell$, the convergence  ${\modar N^\ell_{n,N}} \xrightarrow[n \to \infty]{\mathcal{L}} \mathcal{N}(0,1)$ holds true. Remark also that from the expression of  {\modar $N^\ell_{n,N}$} as sum of independent Laplace variables, we see that the law of {\modar $N^\ell_{n,N}$} does not depend upon the index $\ell \in \{1,\dots,L_n-3\}$. The sequence of Laplace variables $(\mathcal{E}^{i,\ell,(0)}_j)_{i,\ell,j}$ is independent of the shift variable $S$ and thus is also independent of the random index $\ell^\star_n=\lfloor L_n \theta^\star - S \rfloor$. In entails the independence of the random index $\ell^\star_n$ with the family $({\modar N}^\ell_{n,N})_{\ell,n,N}$. From these descriptions, we deduce the equality in law
\begin{equation*}
	( {\modar  N^{\ell^\star_n}_{n,N}},S) \overset{\mathcal{L}}{=}	({\modar N^{1}_{n,N}},S).
\end{equation*}
Consequently, using the convergence in law of {\modar $N^{1}_n$,} we deduce 
\begin{equation} \label{eq: stable convergence of Z ln}
	({\modar N^{\ell^\star_n}_{n,N}},S) \xrightarrow[]{\mathcal{L}}	(\mathcal{N},S),
\end{equation}
where $\mathcal{N}\sim\mathcal{N}(0,1)$ is a random variable independent of $S$.
Comparing the expression \eqref{eq: derive contrast theta star Laplace} with the left hand side of \eqref{eq: cv somme Laplace}, we have 
\begin{equation*}
\frac{\sqrt{\bar{\alpha}_2}}{\sqrt{N}L_n^2 \log(n)}\frac{1}{4{\modar (a+1)}\sqrt{T}\sqrt{v_n(\theta^{\star})}}\frac{\partial}{\partial \theta}H^0(\bm{\mathcal{E}})(\theta^\star) = N_{n,N}^{\ell_n^\star}.
\end{equation*}
 Since $v_n(\theta^\star)={\modar \overline{v}(L_n(\theta^\star-\theta_{\ell_n^*}))}$ is measurable with respect to the random variable $S$, and with law $\overline{v}(S)$ by Lemma \ref{L: law random variance}, we deduce 
from \eqref{eq: stable convergence of Z ln} that the convergence \eqref{eq: conv rescaled main term in partial H} holds true.\\
\\
Now, we prove \eqref{eq: conv rescaled neg term in partial H}.  We write
\begin{align*}
		\frac{\partial}{\partial \theta}
	{\modarn \overline{H}}({\bm{\mathcal{E}}})(\theta^\star)&=\sum_{\ell=-1}^{L_n-1} \sum_{m=0}^a \overline{c}_{\ell}^{m}({\modar {\bm{\mathcal{E}}}}) 
	\frac{\partial}{\partial \theta}B_\ell^m (\theta^\star)
	\\
	&=\sum_{\ell\in\{\ell_n^\star-1,\ell_n^\star\}} \sum_{m=0}^a \overline{c}_{\ell}^{m}({\modar {\bm{\mathcal{E}}}}) 
	\frac{\partial}{\partial \theta}B_\ell^m (\theta^\star),
\end{align*}
where we used that the support of $B^m_j$ is $[\theta_{j},\theta_{j+2})$ and $\theta^\star \in [\theta_{\ell^{\star}_n},\theta_{\ell^{\star}_n+1})$. For $n$ large enough we have $\ell^{\star}_n \in \{1,\dots,{\modar L_n-3}\}$, 
and thus using \eqref{eq: spline rescaled on 0_2}, we deduce
 $\norm{ \frac{\partial}{\partial \theta} B_\ell^m }_\infty \le cL_n$ for some constant $c$ and $\ell\in\{\ell_n^\star-1,\ell_n^\star\}$ . It yields,
\begin{equation*}
\ab{	\frac{\partial}{\partial \theta}
	{\modarn \overline{H}}({\bm{\mathcal{E}}})(\theta^\star) } \le cL_n (a+1)
\sup_{\ell \in\{\ell_n^\star-1,\ell_n^\star\}} \sup_{m \in \{0,\dots,a\}}  \ab{ \overline{c}_{\ell}^{m} ({\modar {\bm{\mathcal{E}}}})}. 
\end{equation*}
For $\nu \ge 1$, we have $|(g_\ell^m)^{(2a+1-\nu)}(\theta_{\ell+1})| \le c/L_n$, and 
recalling \eqref{eq: def overline c}, it implies $|\overline{c}_{\ell}^{m}
({\modar \bm{\mathcal{E} }}) | 
 \le c(L_n)^{-1} \sup_{\nu \in\{1,\dots,a\}} |\bm{\mathcal{E}}^{\ell+1,(\nu)}|$ for all $\ell\in\{-1,\dots,L_n-2\}$ and $m\in\{0,\dots,a\}$. We deduce,
\begin{align*}
\ab{	\frac{\partial}{\partial \theta}	{\modarn \overline{H}}({\bm{\mathcal{E}}})(\theta^\star) } 
&\le c\sup_{\ell \in\{\ell_n^\star-1,\ell_n^\star\}} \sup_{\nu \in \{1,\dots,a\}} \ab{\bm{\mathcal{E}}^{\ell+1,(\nu)}}
\\
&\le c
\sup_{\ell \in\{0,\dots,L_n-1\}} \sup_{\nu \in \{1,\dots,a\}} 
	 \ab{ \sum_{i=1}^N \sum_{j=1}^n \mathcal{E}^{i,\ell,(\nu)}_j},
\end{align*}
where in the last line we used the definition \textcolor{black}{\eqref{Eq_def_bm mathcal E} of $\bm{\mathcal{E}}^{\ell+1,(\nu)}$.} From this, we derive an upper bound for the $L^2$ norm of $|	\frac{\partial}{\partial \theta}
{\modarn \overline{H}}({\bm{\mathcal{E}}})(\theta^\star)|$:
\begin{align*}
\E\left[ \ab{	\frac{\partial}{\partial \theta}
{\modarn \overline{H}}({\bm{\mathcal{E}}})(\theta^\star) }^2\right]
& \le  
\E\left[ \sup_{\ell \in\{0,\dots,L_n - 1\}} \sup_{\nu \in \{1,\dots,a\}} 
\ab{ \sum_{i=1}^N \sum_{j=1}^n \mathcal{E}^{i,\ell,(\nu)}_j }^2\right]		\\
&	\le c \log(n)^2 N L_n^2
\E\left[ \sup_{\ell \in \{0,\dots,L_n - 1\}} \sup_{\nu \in \{1,\dots,a\}} 
\ab{ \frac{1}{\sqrt{nN}} \sum_{i=1}^N \sum_{j=1}^n \textcolor{black}{\overset{\circ}{\mathcal{E}}^{i,\ell,(\nu)}_j }}^2\right],
\end{align*}
where we used 
\eqref{eq: normalized laplace 7}. Then, \eqref{eq: norme p sum Laplace renormalisee} with $p=2$, gives,
\begin{equation*}
\E\left[ \ab{	\frac{\partial}{\partial \theta}
{\modarn \overline{H}}({\bm{\mathcal{E}}})(\theta^\star) }^2\right] \le c \log(n)^2 N L_n^2  \frac{\log(L_n)}{\bar{\alpha}_2}.
\end{equation*}
We deduce from $1/L_n \to 0$, that $\frac{\sqrt{\bar{\alpha}_2}}{\sqrt{N}L_n^2 \log(n)}\frac{\partial}{\partial \theta}{\modarn \overline{H}}(\bm{\mathcal{E}})(\theta^\star) $ converges to zero in $L^2$-norm and thus in probability. Since, using Lemma \ref{L: law random variance}, the law of $v_n(\theta^\star)$ does not depend on $n$ and $\P(v_n(\theta^\star)>0)=1$, we obtain \eqref{eq: conv rescaled neg term in partial H}. \\
\\
Last step is devoted to the proof of \eqref{eq: conv rescaled neg term f in partial H}. {\modar Applying Lemma \ref{Lem_infinite_bound_deriv_interpolation}, and recalling \eqref{E:_def gamma k l} we have
\begin{align*}
	\abs*{\frac{\partial}{\partial \theta} H(\bm{\gamma})(\theta^\star)} &\le c L_n \sup_{\ell,k} \abs*{ \sum_{i=1}^N \sum_{j=1}^n f_j^{i,\ell,(k)} \varphi_{\tau_n}(f_j^{i,\ell,(k)}) }
	\\ & \le c L_n n N \sup_{i,j,\ell,k} \abs{f_j^{i,\ell,(k)} \varphi_{\tau_n}(f_j^{i,\ell,(k)}) }
	\\ &
	\le c L_n n N \tau_n \le c L_n \sqrt{n} \log(n) N ,
\end{align*}
where we have used $\tau_n=\sqrt{\Delta_n} \log(n) \le c n^{-1/2} \log(n)$. {\modchi Recall we have introduced the set $\Omega^n$ in \eqref{eq: set Omega}, in the proof of Proposition \ref{prop: approx contrast}. The bound above implies}  
\begin{equation*}
	\E \left[ 	\abs*{\frac{\partial}{\partial \theta} H(\bm{\gamma})(\theta^\star)}   \mathbf{1}_{(\Omega^n)^c} \right] \le
	 c L_n \sqrt{n} \log(n) N \P\left((\Omega^n)^c\right) 
	 \le c L_n \sqrt{n} \log(n) N n^{-r_0},	 
\end{equation*}
for any $r_0$ arbitrarily large, recalling \eqref{E: majo comple Omega n}. It entails,
\begin{equation} \label{E: cv partial H gamma Omega comp}
	\frac{\sqrt{\bar{\alpha}_2}}{\sqrt{N}L_n^2\log(n)} 	\abs*{\frac{\partial}{\partial \theta} H(\bm{\gamma})(\theta^\star)}  \mathbf{1}_{(\Omega^n)^c}
	\xrightarrow[]{L^1} 0,
\end{equation}
using that both $L_n$ and $N$ are polynomial in $n$, together with the fact 
that $r_0$ can be as large as needed.

On the set $\Omega^n$ we have for $k \in \id{0}{a}$ and $\ell \in \id{0}{L_n-1}$,
\begin{equation*}
\bm{\gamma}^{\ell,(k)}=\sum_{i=1}^N \sum_{j=1}^n f^{i,\ell,(k)}_j \varphi_{\tau_n}(f^{i,\ell,(k)}_j)=\sum_{i=1}^N \sum_{j=1}^n f^{i,\ell,(k)}_j=\frac{\partial^k }{\partial \theta^k}S_n^{N,0}(\theta_\ell)
\end{equation*} 
 by \eqref{E: def f_ijlk} and \eqref{eq: def contrast without privacy}.
  Thus, on $\Omega^n$, we have $H(\bm{\gamma})=HS_n^{N,0}$ and we can write
\begin{align*}
	\frac{\partial}{\partial \theta} H(\bm{\gamma})(\theta^\star)=\frac{\partial}{\partial \theta} H S_n^{N,0}(\theta^\star)
	&=
	\frac{\partial}{\partial \theta} \left( H S^{N,0}_n- S^{N,0}_n \right) (\theta^\star)  
	+
	\frac{\partial}{\partial \theta} S^{N,0}_n (\theta^\star)  
	\\&= I_1(\theta^\star)+I_2(\theta^\star).
\end{align*}
From \eqref{E: contrast error part1} with $u=p=1$, we get $\E \left[|I_1(\theta^\star)|\bm{1}_{\Omega^n}\right] \le 
\E \left[|I_1(\theta^\star)|\right]
 \le L_n^{2-a} N$. We deduce from A\ref{Ass_splines}.\ref{H: Ass_splines priv dominate}
\begin{equation}\label{E: cv I1 star CLT priv}
	\frac{\sqrt{\bar{\alpha}_2}}{\sqrt{N}L_n^2 \log(n)} I_1(\theta^\star) \to 0
\end{equation}
in $L^1$ and thus in probability.}

\noindent To conclude, let us analyze {\modar $I_2(\theta^{\star})$. Using} the notation {\modar \eqref{E: def xi j 1}} introduced in the proof of Proposition \ref{prop: as norm without privacy}, we have 
\begin{equation*}
\sqrt{{\bar{\alpha}_2}} \frac{1}{L_n^2 \log(n)}\frac{1}{\sqrt{N}} {\modar \frac{\partial}{\partial \theta}}S_n^{N, 0} (\theta^{\star}) 
 = \sqrt{{\bar{\alpha}_2}} \frac{1}{L_n^2 \log(n)} \sum_{j = 1}^n \xi_j^{(1)} (\theta^{\star}) =: \sum_{j = 1}^n \tilde{\xi}_j(\theta^{\star}). 
\end{equation*}
Then, 
$$\sum_{j = 1}^n {\modar \abs{ \E_{t_{j - 1}}[\tilde{\xi}_j(\theta^{\star})]}} = \sqrt{{\bar{\alpha}_2}} \frac{1}{L_n^2 \log(n)} \sum_{j = 1}^n 
{\modar \abs{\E_{t_{j - 1}}[{\xi}_j^{\modhel (1)}(\theta^{\star})]}} \xrightarrow{\mathbb{P}} 0,$$
where {\modar we used that, from the proof of (36) in Proposition 6.2 of \cite{McKean},} we have 
$\E_{t_{j - 1}}[ {\xi}_j^{\modhel (1)}(\theta^{\star})]= R_{t_{j-1}} (\sqrt{N}\Delta_n^{3/2} )$ 
and $\sqrt{{\bar{\alpha}_2}} \frac{1}{L_n^2 \log(n)} \sqrt{N \Delta_n} \rightarrow 0$ {\modar from A\ref{Ass_discretization}.\ref{H: Ass_discret priv dominate}.} 
 \\
Moreover, let us consider 
$$\sum_{j = 1}^n \E_{t_{j - 1}}[(\tilde{\xi}_j(\theta^{\star}))^2] = {\bar{\alpha}_2} \frac{1}{L_n^4 (\log(n))^2} \sum_{j = 1}^n \E_{t_{j - 1}}[({\xi}_j^{\modhel (1)}(\theta^{\star}))^2]. $$ 
{\modar Following the proof of (37) in Proposition 6.2 of \cite{McKean}, we have
{\modchi 
\begin{align*}
 \E_{t_{j - 1}}[({\xi}_j^{\modhel (1)}(\theta^{\star}))^2] & \le \frac{1}{N} \sum_{i_1, i_2 = 1}^N (R_{t_j-1}(\Delta_n)\mathbf 1_{i_1= i_2} + R_{t_j-1}(\Delta_n^2)\mathbf 1_{i_1 \neq i_2}) \\
& \le R_{t_j-1}(\Delta_n + N \Delta_n^2).    
\end{align*}
Hence, $\sum_{j = 1}^n \E_{t_{j - 1}}[(\tilde{\xi}_j(\theta^{\star}))^2]$ converges to $0$ in $L^1$ as ${\bar{\alpha}_2} \frac{1}{L_n^4 \log(n)^2}(1 + N \Delta_n) \rightarrow 0$ because of $r_{n,N} \to \infty$ and 
the hypothesis A\ref{Ass_discretization}.\ref{H: Ass_discret priv dominate}.} Then, Lemma 9 in \cite{GenonCatalot_Jacod_1993} guarantees that
\begin{equation}{\label{eq: conv I2 4}}
 \sqrt{\frac{\bar{\alpha}_2}{N}} \frac{1}{L_n^2 \log(n)} I_2(\theta^{\star}) =  \sum_{j = 1}^n \tilde{\xi}_j(\theta^{\star}) \xrightarrow{\mathbb{P}} 0. 
\end{equation}
From \eqref{E: cv partial H gamma Omega comp}, \eqref{E: cv I1 star CLT priv} and \eqref{eq: conv I2 4}, we obtain the convergence to zero of
$\frac{\sqrt{\bar{\alpha}_2}}{\sqrt{N}L_n^2 \log(n)} \frac{\partial}{\partial \theta}H({\modar \bm{\gamma}})(\theta^\star) $.
It yields \eqref{eq: conv rescaled neg term f in partial H}}  remarking that, because of Lemma \ref{L: law random variance}, the law of $v_n(\theta^{\star})$ does not depend on $n$ and $\mathbb{P}(v_n(\theta^{\star}) > 0) = 1$. 
Then, the proof of the proposition is concluded.
\end{proof}

\subsubsection{Proof of Corollary \ref{cor: threshold}}
\begin{proof}
  {\modch Observe that the proof of Corollary \ref{cor: threshold} heavily relies on the results obtained until this point. Indeed, its proof follows the same route of Theorems \ref{th: as norm privacy negl} and \ref{th: as norm} above. Nevertheless, in this instance, the terms associated with significant and negligible privacy each play a contributory role. 
  	{ {\modarn We start by  writing  \eqref{eq: as norm start 3} as
  		\begin{equation*}
  			(\hat{\theta}^N_n - \theta^{\star}) = - \frac{\sqrt{\frac{\bar{\alpha}_2}{N}}\frac{1}{L_n^2 \log(n)} \partial_\theta {\modch S^{N,\text{pub}}_n}(\theta^{\star})}{\frac{1}{N} \int_0^1 \partial^2_\theta {\modch S^{N,\text{pub}}_n}(\theta^{\star} + s (\hat{\theta}^N_n - \theta^{\star})) ds } \quad  \frac{L_n^2 \log(n)}{\sqrt{\bar{\alpha}_2N}}                    .
  		\end{equation*}
  }	
  	 {\modarn Following	 the proof of Proposition \ref{prop: as norm privacy} above,  one has
\begin{equation*}
	\sqrt{\frac{\bar{\alpha}_2}{N}} \frac{1}{L_n^2 \log(n)} \frac{\partial}{\partial \theta} S_n^{N,\text{pub}}(\theta^\star)
	=\sqrt{\frac{\bar{\alpha}_2}{N}} \frac{1}{L_n^2 \log(n)} \left[\frac{\partial}{\partial \theta} H^0(\bm{\mathcal{E}})(\theta^\star) + \frac{\partial}{\partial \theta} {\modarn \overline{H}}(\bm{\mathcal{E}})(\theta^\star) + \frac{\partial}{\partial \theta} H(\bm{\gamma})(\theta^\star)  \right].
\end{equation*}  	}
The main modification with the proof Proposition \ref{prop: as norm privacy} is that, under the hypothesis of Corollary \ref{cor: threshold}, the convergence result \eqref{eq: conv rescaled neg term f in partial H} does not hold.} Indeed,  $\sqrt{\frac{\bar{\alpha}_2}{N}} \frac{1}{L_n^2 \log(n)} I_2(\theta^{\star})$ no longer converges to $0$.
Instead,  as $\sqrt{{\bar{\alpha}_2}} \frac{1}{L_n^2 \log(n)} \rightarrow c_p$ for $n, N \rightarrow \infty$, Proposition 6.2 of \cite{McKean} directly provides 
$$ \sqrt{\frac{\bar{\alpha}_2}{N}} \frac{1}{L_n^2 \log(n)} I_2(\theta^{\star}) ~\mathop{\sim}_{n,N \to \infty}~ \frac{c_p}{\sqrt{N}} I_2(\theta^{\star}) \xrightarrow{\mathcal{L}} \mathcal{N}(0, 2 c_p^2 \Sigma_0 ) = c_p \mathcal{Z}_1.$$ {\modarn Let us remark that the condition $N \Delta_n \to 0$ necessary
	to apply Proposition 6.2 of \cite{McKean} holds from Assumption A\ref{Ass_discretization}.\ref{H: Ass_discret priv dominate} with $\sqrt{{\bar{\alpha}_2}} \frac{1}{L_n^2 \log(n)} \rightarrow c_p>0$. Then, we deduce, relying on \eqref{E: cv partial H gamma Omega comp}--\eqref{E: cv I1 star CLT priv},
\begin{equation*}
	\sqrt{\frac{\bar{\alpha}_2}{N}} \frac{\partial}{\partial \theta} H(\bm{\gamma})(\theta^\star)  \xrightarrow[]{\mathcal{L}}    c_p \mathcal{Z}_1.
\end{equation*}
Now, \eqref{eq: conv rescaled main term in partial H}--\eqref{eq: conv rescaled neg term in partial H} {\modchi together} with the independence of the variables $(\mathcal{E}^{i,\ell,(k)}_j)_{i,j,\ell,k}$ {\modchi from} the processes $(X^i)_i$ yields
\begin{equation*}
			\sqrt{\frac{\bar{\alpha}_2}{N}} \frac{1}{L_n^2 \log(n)} \frac{\partial}{\partial \theta} S_n^{N,\text{pub}}(\theta^\star)
		\xrightarrow[]{\mathcal{L}}	\mathcal{N}\left(0, 2 c_p^2 \Sigma_0 + 16(a+1)^2 T \overline{v}(S)  \right).
\end{equation*}
} 
%
\noindent The analysis of the second derivatives of the contrast function is due once again to Proposition \ref{prop: second derivatives} that, in the same way as in the proof of Theorems \ref{th: as norm privacy negl} and \ref{th: as norm} above, implies the wanted result {\modar on $\frac{\sqrt{N \bar{\alpha}_2}}{ {\modar  L_n^2 \log(n)}} (\widehat\theta_n^N - \theta^{\star})$.}} 
\end{proof}

\subsubsection{Proof of Proposition \ref{prop: second derivatives}}
\begin{proof}
The main tool consists in the approximation argument gathered in Proposition \ref{prop: approx contrast}, this time for $u = 2$. Indeed, 
$$\frac{1}{{N}}\partial_\theta^2 {\modch S^{N,\text{pub}}_n}({\modar \theta}) = \frac{1}{{N}}\big(\partial^2_\theta {\modch S^{N,\text{pub}}_n}({\modar \theta}) - \partial^2_\theta S_n^{N,0}({\modar \theta})\big) + \frac{1}{{N}}\partial_\theta^2 S_n^{N,0}({\modar \theta}).  $$
Proposition \ref{prop: approx contrast} implies
$$\left \| 
{\modar \sup_{\theta \in [\theta_0,\theta_{L_n-1}]}}
\frac{1}{{N}}(\partial^2_\theta {\modch S^{N,\text{pub}}_n}({\modar \theta}) - \partial^2_\theta S_n^{N,0}({\modar \theta})) \right \|_p \le c \frac{1}{L_n^{a - 3}} + c L_n^3 \log(n) \sqrt{\frac{\log(L_n)}{ N \bar{\alpha}_2}} {\modar + c \frac{L_n \sqrt{N}}{n^r}},$$
{\modar where $r>0$ can be chosen arbitrarily large.  Hence, such $L^p$-norm goes} to $0$ as we have chosen $a > 3$ and $L_n^3 \log(n) \sqrt{\frac{\log(L_n)}{ N \bar{\alpha}_2}} \rightarrow 0$ for $N, n \rightarrow \infty$. Then, following the proof of Proposition 6.3 in \cite{McKean} (see in particular Equation (54)), it is straightforward to check that 
$$\frac{1}{N} \partial^2_\theta S_n^{N,0} (\theta) \xrightarrow{\mathbb{P}} 2 \int_0^T \E\bigg[\Big(\frac{\partial_\theta b(\theta, X_s)}{{\modch \sigma}(X_s)}\Big)^2 - \frac{\partial_\theta^2 b(\theta, X_s)}{{\modch \sigma}^2(X_s)}(b(\theta^{\star}, X_s) - b(\theta, X_s) )\bigg] ds = : 2 \Sigma(\theta)$$
uniformly in $\theta \in \Theta$. 
It 
concludes the proof of the first point. \\
The second point follows from the continuity of $\Sigma(\theta)$ at $\theta = \theta^{\star}$ and the consistency of $\hat{\theta}_n^N$ as proved in Theorem \ref{th: consistency}.   
\end{proof}

\subsection{Proof of the case where the drift is polynomial in $\theta$}\label{App: drift}
{\rev This section focuses on proving the results from Theorem \ref{th: pol drift}. We assume that the drift function takes the form \(b(\theta, x) = b_1(\theta) b_2(x)\), where \(b_1\) is a polynomial of degree at most \(a\). The proofs heavily rely on the fact that, when the drift is a polynomial function of \(\theta\), Proposition \ref{prop: approx contrast} can be replaced by the following result.
}

\begin{prop}{\label{prop: approx drift pol}}
\textcolor{black}{Assume that Assumptions \ref{Ass_1}-\ref{Ass_3} hold. Then, for any \(p \geq 2\) and for any \(u \in \{0, \dots, a\}\),
\begin{align*}
\left\| E_{n,N}^{(u)} \right\|_{p} & \le c \left( L_n^{u + 1} \sqrt{N} \log(n) \sqrt{\frac{\log(L_n)}{\bar{\alpha}_2}} + \frac{L_n^u N}{n^{r}} \right),
\end{align*}
where the constant \(r > 0\) can be chosen arbitrarily large.}
\end{prop}

\begin{proof}
\textcolor{black}{The proof of Proposition \ref{prop: approx drift pol} closely follows that of Proposition \ref{prop: approx contrast}. The key difference is that, since Hermite approximations preserve polynomial functions, the right-hand side in Proposition \ref{Prop_Hermite_interpolation} becomes exactly zero. As a result, in the proof of Proposition \ref{prop: approx contrast}, the first step can be skipped entirely. Step 2 then yields the desired result.}
\end{proof}

\subsubsection{Proof of Theorem \ref{th: pol drift}}
\begin{proof}
\textcolor{black}{To prove (i), it suffices to replace Proposition \ref{prop: approx contrast} with Proposition \ref{prop: approx drift pol} in the proof of Theorem \ref{th: consistency}, specifically in Equation \eqref{eq: pol drift}. This substitution directly yields the desired result.}\\
\noindent
\textcolor{black}{Now, let us proceed to the proof (ii), which addresses the asymptotic normality of the estimator in the scenario where the privacy constraints are negligible. This result relies on Proposition \ref{prop: as norm without privacy}, which uses Proposition \ref{prop: approx contrast}, particularly in Equation \eqref{eq: pol as norm negl}. By substituting it with Proposition \ref{prop: approx drift pol}, we obtain the following expression:
$$\left \| \frac{1}{\sqrt{N}}(\partial_\theta {S^{N,\text{pub}}_n}(\theta^{\star}) - \partial_\theta S_n^{N,0}(\theta^{\star})) \right \|_p \le c L_n^2 \log(n) \sqrt{\frac{\log(L_n)}{\bar{\alpha}_2}} + c\frac{L_n \sqrt{N}}{n^r},$$
which explains why the hypothesis A\ref{Ass_splines}.\ref{H: Ass_splines priv neg} is no longer necessary in our analysis. This concludes the proof of the second point of the theorem.}
\\
\noindent
\textcolor{black}{Next, we need to demonstrate (iii), i.e. the asymptotic normality in cases where the privacy constraints are no longer negligible. This is based in Proposition \ref{prop: as norm privacy}, which uses the spline approximation to establish the convergence in \(L^1\) of \(I_1(\theta^\star)\) as shown in Equation \eqref{E: cv I1 star CLT priv}. For this, hypothesis A\ref{Ass_splines}.\ref{H: Ass_splines priv dominate} is required. However, because the drift is polynomial in \(\theta\) and the error from the spline approximation is zero, the term \(I_1(\theta^\star)\) also becomes zero. This lets us remove the requirement of A\ref{Ass_splines}.\ref{H: Ass_splines priv dominate}. Therefore, the conclusion of the proof of Theorem \ref{th: pol drift} follows easily from the previous arguments.
}
\end{proof}

\section{Proof of preliminary results}{\label{s: proof preliminary}}

\subsection{Proof of Lemma \ref{L: law random variance}}
\begin{proof}
 From the definition of $\ell_n^\star$, we have $\theta_{\ell_n^\star}=\frac{\ell_n^\star +S}{L_n} \le \theta^\star < \frac{\ell_n^\star+1 +S}{L_n}$. We deduce that $\textcolor{black}{\ell_n^\star}=\lfloor L_n \theta^\star -S\rfloor$ and
 $\theta_{\ell_n^\star}=\frac{\lfloor L_n \theta^\star -S\rfloor +S}{L_n}$.  Consequently,
 $v_n(\theta^\star)=\overline{v}(L_n(\theta^\star-\theta_{\ell_n\star}))=
 \overline{v}(L_n\theta^\star-\lfloor L_n \theta^\star -S\rfloor -S)$. Let us denote by $\vartheta^\star_n=L_n\theta^\star-\lfloor L_n\theta^\star \rfloor \in [0,1)$ the fractional part of $L_n\theta^\star$. We have $L_n\theta^\star-\lfloor L_n \theta^\star -S\rfloor = \vartheta^\star_n - \lfloor \vartheta^\star_n -S \rfloor$ and we deduce that $v_n(\theta^\star)=\overline{v}(\vartheta^\star_n -S - \lfloor \vartheta^\star_n -S \rfloor)$. 
 Now, if we check that 
 \begin{equation} \label{eq: egalite loi tronque unif}
 	\vartheta^\star_n -S - \lfloor \vartheta^\star_n -S \rfloor \overset{\mathcal{L}}{=} S, 
 \end{equation}
the lemma will be proved. Let us recall \textcolor{black}{that the random variable $S$ is uniformly distributed on $(0,1)$}, and therefore we need to show that the L.H.S. of \eqref{eq: egalite loi tronque unif} shares the same law.  For $g$ a non negative measurable real function, we write $\E[g(	\vartheta^\star_n -S - \lfloor \vartheta^\star_n -S\rfloor )]=\E[g(	\vartheta^\star_n -S) \mathbf{1}_{S \le \vartheta_n^\star}] + \E[g(	\vartheta^\star_n -S +1) \mathbf{1}_{S > \vartheta_n^\star}]$
where we used that $\vartheta^\star_n \in[0,1)$. Since $S$ is uniform, this gives $\E[g(	\vartheta^\star_n -S - \lfloor \vartheta^\star_n -S\rfloor )]=\int_0^{\vartheta_n^\star} g(\vartheta_n^\star - s) ds + 
\int_{\vartheta_n^\star}^1 g(\vartheta_n^\star - s +1) ds = \int_0^{\vartheta_n^\star} g( s) ds + 
\int_{\vartheta_n^\star}^1 g(s ) ds=\int_0^1 g(s)ds$, by change of variables.  We deduce \eqref{eq: egalite loi tronque unif} and the lemma follows.
 \end{proof}

\subsection{Proof of Proposition \ref{Prop_Hermite_interpolation}}\label{subsec_proof_splines}
\begin{proof}
{\modar Let $x_0\in[\xi_0,\xi_{\modch \Lambda}]$ and let $P_{x_0}$ be the Taylor approximation of $f$ at the point $x_0$ with order $a$, defined by
	$P_{x_0}(x)=\sum_{k=0}^a f^{(k)}(x_0)(x-x_0)^k/k!$.}	
{\modar 
	We know that the spline approximation of the polynomial function $P_{x_0}$ of degree at most $a$ is exact, yielding to $HP_{x_0}=P_{x_0}$. 
Hence, using that for $k\le a$ the $k$-th derivative of $f$ and $P_{x_0}$ are the same at the point $x_0$, we can write for $k \in \{0,\dots,a\}$,
\begin{align} \label{eq: depuis egalite f Px en x0}
\frac{\partial^k }{\partial x^k}Hf(x_0)	- \frac{\partial^k }{\partial x^k}f(x_0)&=
\frac{\partial^k }{\partial x^k}Hf(x_0)	- \frac{\partial^k }{\partial x^k}P_{x_0}(x_0)\\
&\nonumber =
\frac{\partial^k }{\partial x^k}Hf(x_0)	- \frac{\partial^k }{\partial x^k}HP_{x_0}(x_0),
\end{align}
where in the second line we used $HP_{x_0}=P_{x_0}$.
}
%
%
{\modar We write}
\begin{align} \nonumber
\bigg|\frac{\partial^k}{\partial x^k}\big({\modar H 
}f-HP_{x_0}\big)(x_0)\bigg|
&=\bigg|\frac{\partial^k}{\partial x^k}
\bigg[
{\modar \sum_{i=-1}^{{\modch \Lambda}-1} \sum_{r=0}^a [c_i^r(f)-c_i^r(P_{x_0})]B_i^r(x)}\bigg]\bigg|_{x=x_0}\bigg|\\
&  \label{Eq_bound_partial_Hf_2}
\leq c({\modar a}+1){\modch \Lambda}^k\sup_{i\in\{i_{0}-1,i_0\},r\in\{0,\dots,a\}}|c_i^{r}(f)-c_i^{r}(P_{x_0})|{\modar ,}
\end{align}
{\modar where $i_0 \in \{0,\dots,{\modch \Lambda}-1\}$ is such that $x_{0} \in [\xi_{i_0},\xi_{i_0+1})$ and we used 
$\norm{\frac{\partial^k}{\partial x^k} B^r_i}_\infty \le c \Lambda^k$
.}
%
%
%
{\modar From \eqref{Eq_approx_cik}, }
\begin{align}
{\modar \big|c_i^r(f)-c_i^r(P_{x_0})\big| }
&\nonumber =\bigg|
{\modar \sum_{m=0}^a \frac{(-1)^m}{m!}(g_i^r)^{(2a+1-m)}(\xi_{i+1})\big[f^{(m)}(\xi_{i+1})-P_{x_0}^{(m)}(\xi_{i+1})\big]}\bigg|\\
& \label{eq: borne pour coeff c f cPx}
\leq {\modar c \sum_{m=0}^{a}\frac{1}{{\modch \Lambda}^{m}}\big|f^{(m)}(\xi_{i})-P_{x_0}^{(m)}(\xi_{i})\big|},
\end{align}
since $(\modar{g_i^r})^{(2a+1-m)}$ is the {\modar $(2a+1-m)$-derivative of a $(2a+1)$-degree} polynomial function, and so it is an $m$-degree polynomial function computed by differentiation of \eqref{Eq_function_gk_def}.\\
{\modar Collecting \eqref{eq: depuis egalite f Px en x0}, \eqref{Eq_bound_partial_Hf_2} and \eqref{eq: borne pour coeff c f cPx} we deduce
	\begin{equation*}
		\left|\frac{\partial^k }{\partial x^k}Hf(x_0)	- \frac{\partial^k }{\partial x^k}f(x_0)\right|		\le c \sup_{i\in\{i_0-1,i_0\}} \sum_{m=0}^a {\modch \Lambda}^{k-m}\big|f^{(m)}(\xi_{i})-P_{x_0}^{(m)}(\xi_{i})\big|  .
	\end{equation*}
Then, using a Taylor expansion of order $a-m$ around $x_0$ for $f^{(m)}(\xi_i)-P_{x_0}^{(m)}(\xi_i)$, and recalling that $f^{(k)}(x_{0})-P_{x_0}^{(k)}(x_{0})=0$ for $k\in\{0,\dots,a\}$, we deduce that 
$\left|f^{(m)}(\xi_i)-P_{x_0}^{(m)}(\xi_i)\right| \le |\xi_i-x_0|^{(a+1-m)} \norm{f^{(a+1)}-P_{x_0}^{(a+1)}}_\infty =  |\xi_i-x_0|^{(a+1-m)} \norm{f^{(a+1)}}_\infty \le c {\modch \Lambda}^{m-a-1}\norm{f^{(a+1)}}_\infty$. It yields
	\begin{equation*}
	\left|\frac{\partial^k }{\partial x^k}Hf(x_0)	- \frac{\partial^k }{\partial x^k}f(x_0)\right|		\le c 
	 \sum_{m=0}^a L_n^{k-m} {\modch \Lambda}^{m-a-1}\norm{f^{(a+1)}}_\infty \le c {\modch \Lambda}^{k-a-1}\norm{f^{(a+1)}}_\infty .
\end{equation*}
As $x_0\in [\xi_0,\xi_{\modch \Lambda}]$ is arbitrary, we get the result.
}
   
\end{proof}

\subsection{Proof of Lemma \ref{l: spline g}}\label{Subsec_spline g}
\begin{proof}
{\modar We start with the introduction of some useful notation for the proof. First, we define
	$\bs{\overline t}=\{\overline t_0,\cdots,\overline t_{3a+2}\}$ a sequence of knots where $\overline{t}_l=0$ for $0\le l \le a$, $\overline{t}_l=1$ for $a+1\le l \le 2a+1$ and  $\overline{t}_l=2$ for $2a+2 \le l \le 3a+2$. Recalling Definition \ref{Def_splines}, we have that, for $0\le k \le a$, $\overline{B}_{k}=B_{h,2a+1,\overline{t}}$ is the B-spline of degree $2a+1$ relying on the $2a+3$ knots $\{\overline{t}_k,\dots,\overline{t}_{k+2a+2}\}=
	\underbrace{\overline t_k,\dots,\overline t_a}_{=0},\underbrace{\overline t_{a+1},\dots,\overline t_{2a+1}}_{=1},\underbrace{ \overline t_{2a+2},\dots, \overline t_{2a+2+k}}_{=2}  $. 
	This leads us to denote by $	\overline B_{(\beta_1,\beta_2,\beta_3)}$
	the B-spline of order $\beta_1+\beta_2+\beta_3-2$ where  $(\beta_1,\beta_2,\beta_3)$ stands for the number of repetitions of each knots : $0$ is repeated $\beta_1\modhel{=a+1-k}$ times, $1$ is repeated $\beta_2\modhel{=a+1}$ times, and $2$ {\modch is} repeated $\beta_3\modhel{=k+1}$ times. Using this notation, we have $\overline B_k:=\overline B_{k,2a+1,\bs{\overline t}}=\overline B_{(a+1-k,a+1,k+1)}$, for 
	$0\le k \le a$.}
{\modar For example,} $\overline B_0$ is a B-spline based on points $\{\overline t_0,\cdots,\overline t_{2a+{\modar 2}}\}=\underbrace{\overline t_0,\cdots,\overline t_a}_{=0},\underbrace{\overline t_{a+1},\cdots,\overline t_{2a+{\modar 1}}}_{=1},\underbrace{\overline t_{2a+{\modar2}}}_{=2}$, that we sum up through the notation
\begin{equation*}
\overline B_{0}=\overline B_{(a+1,a+1,1)}.
\end{equation*}
%
%
Recall that {\modar by Theorem 9 in \cite{Lyche_2017},} $\overline B_{j,p,\bs{\overline t}}$ satisfies
\begin{equation*}
\overline B_{j,p,\bs{\overline t}}'=p\Bigg[\frac{\overline B_{j,p-1,\bs{\overline t}}}{(\overline t_{j+p}-\overline t_{j})}-\frac{\overline B_{j+1,{\modar p-1},\bs{\overline t}}}{(\overline t_{j+p+1}-\overline t_{j+1})}\Bigg].
\end{equation*}
Noting that $\overline B_k=\overline B_{k,2a+1,\bs{\overline t}}=\overline B_{(a+1-k,a+1,k+1)}$, we have
{\modar  for $0<k< a$,}
\begin{equation*}
\overline B_{k}'
=\frac{2a+1}{2}\Big[\overline B_{k,2a,\bs{\overline t}}-\overline B_{k+1,2a,\bs{\overline t}}\Big]=\frac{2a+1}{2}\Big[\overline B_{(a+1-k,a+1,k)}-\overline B_{(a-k,a+1,k+1)}\Big],
\end{equation*}
as well as
\begin{equation*}
\overline B_{0}'
{\modar =(2a+1) \overline B_{(a+1,a+1,0)}-\frac{2a+1}{2}\overline B_{(a,a+1,1)}} \;\;\text{and}\;\;\overline B_{a}'
={\modar \frac{(2a+1)}{2} \overline B_{(1,a+1,a)}- (2a+1)\overline B_{(0,a+1,a+1)}.}
\end{equation*}
Then,
\begin{align*}
\sum_{k=0}^{a}\overline B_{k}'
&=\sum_{k=1}^{a-1}\frac{2a+1}{2}{\modchi \Big[\overline B_{(a+1-k,a+1,k)}-\overline B_{(a-k,a+1,k+1)}\Big]}+(2a+1)\Big[\overline B_{(a+1,a+1,0)}-{\modar \frac{1}{2}}\overline B_{(a,a+1,1)}\Big]\\
&+(2a+1)\Big[{\modar \frac{1}{2}}\overline B_{(1,a+1,a)}-\overline B_{(0,a+1,a+1)}\Big]\\
&=\frac{2a+1}{2}\Big[\overline B_{{\modch (}a,a+1,1 {\modch )}}-\overline B_{ {\modch (} 1,a+1,a {\modch )}}\Big]+(2a+1)\Big[\overline B_{(a+1,a+1,0)}-{\modar \frac{1}{2}}\overline B_{(a,a+1,1)}+{\modar \frac{1}{2}}\overline B_{(1,a+1,a)}-\overline B_{(0,a+1,a+1)}\Big]
\\
&{\modar =(2a+1)[\overline B_{(a+1,a+1,0)}-\overline B_{(0,a+1,a+1)} ]}.
%
\end{align*}
We can show (\textcolor{black}{see for instance \cite{deBoor_Book}, Section 3}) that
\begin{equation*}
\overline B_{(a+1,a+1,0)}(x)=\binom{2a}{a}x^{a}(1-x)^a {\modar \mathbf{1}_{[0,1]}(x)}
\;\;\text{and}\;\; \overline B_{(0,a+1,a+1)}(x)=\binom{2a}{a}(x-1)^{a}(2-x)^a
{\modar \mathbf{1}_{[1,2]}(x).}
\end{equation*}
%
{\modar We deduce that, for $x \in \mathbb{R}$,}
\begin{align*}
\sum_{k=0}^{a}\overline B_{k}'(x)
=(2a+1) \binom{2a}{a} \Big[x^{a}(1-x)^a {\modar \mathbf{1}_{[0,1]}(x)}-(x-1)^{a}(2-x)^a
{\modar \mathbf{1}_{[1,2]}(x)}\Big].
\end{align*}    
Hence, the result. 
\end{proof}

\subsection{Proof of Lemma \ref{l: asymptotic 2.5}}
\begin{proof}
We start by proving that 
$$\frac{\Delta_n}{N} \sum_{i = 1}^N \sum_{j = 1}^n f (X^i_{t_{j - 1}}) - \frac{1}{N} \sum_{i = 1}^N \int_0^T f(X_s^i) ds \xrightarrow{L^1} 0.$$
Indeed, one can write $\int_0^T f(X_s^i) ds$ as $\sum_{j = 1}^n \int_{t_{j - 1}}^{t_j} f(X_s^i) ds$. Recall moreover that $\Delta_n = t_j - t_{j - 1} = \int_{t_{j - 1}}^{t_j} ds.$ Hence, the norm $1$ of the difference above is bounded by 
\begin{align*}
&\frac{1}{N} \sum_{i = 1}^N \sum_{j = 1}^n \int_{t_{j - 1}}^{t_j} \E[|f (X^i_{t_{j - 1}}) - f (X^i_{s})|] ds \\
& \le \frac{c}{N} \sum_{i = 1}^N \sum_{j = 1}^n \int_{t_{j - 1}}^{t_j} \E\big[\textcolor{black}{|X^i_{t_{j - 1}} -X^i_{s}|(1 + |X^i_{t_{j - 1}}| + |X^i_s|)^k}\big] ds \\
 & \le \frac{c}{N} \sum_{i = 1}^N \sum_{j = 1}^n \int_{t_{j - 1}}^{t_j} |t_{ j-1} - s|^\frac{1}{2} ds \le c \Delta_n^\frac{1}{2},
\end{align*}
having used Cauchy-Schwarz inequality and Points 1 and 2 of Lemma \ref{l: moment}. Then, it clearly goes to $0$ as we wanted. To conclude the proof remark that the law of large number provides 
$$\frac{1}{N} \sum_{i = 1}^N \int_0^T f(X_s^i) ds \xrightarrow{\mathbb{P}} \int_0^T \E[f(X_s)] ds,$$
as $\{(X^i_s)_{s\in[0,T]}, i \in \id{1}{N}\}$ are i.i.d. processes.
\end{proof}

\subsection{Proof of Lemma \ref{l: bound proba}}
\begin{proof}
Recall that, for all $i\in\{1,\dots,N\}$, $j \in \{1, \dots , n \}$ and $k \in \N$ it is 
$$f^{(k)}(\theta; X_{t_{j - 1}}^i, X_{t_j}^i) = 2 \frac{\partial_\theta^k b(\theta,X_{t_{j-1}}^{i})(X_{t_j}^{i}-X_{t_{j-1}}^{i})}{{\modch \sigma}^{2}(X_{t_{j-1}}^{i})} + \Delta_n \frac{\partial_\theta^k (b^2(\theta,X_{t_{j-1}}^{i}))}{{\modch \sigma}^{2}(X_{t_{j-1}}^{i})}.$$
Then, we have
\begin{align*}
\mathbb{P}_{t_{j-1}}(|f^{(k)}(\theta;X_{t_{j-1}}^{i},X_{t_j}^{i})| > \tau_n)
&\le \mathbb{P}_{t_{j-1}}\Big(\frac{2 \partial_\theta^k b(\theta,X_{t_{j-1}}^{i})(X_{t_j}^{i}-X_{t_{j-1}}^{i})}{{\modch \sigma}^{2}(X_{t_{j-1}}^{i})}>\frac{\sqrt{\Delta_n}\log(n)}{2}\Big)\\
&+\mathbb{P}_{t_{j-1}}\Big(\frac{\Delta_n \partial_\theta^k( b^2(\theta,X_{t_{j-1}}^{i}))}{{\modch \sigma}^{2}(X_{t_{j-1}}^{i})}>\frac{\sqrt{\Delta_n}\log(n)}{2}\Big)\\
&=:P_1+P_2.
\end{align*}
Since ${\modch \sigma}$ and the derivatives of $b$ are lower and upper bounded by Assumptions \ref{Ass_2} and \ref{Ass_3} respectively, and $X_{t_{j-1}}^{i}$ is $\cF_{t_{j-1}}^N$-measurable, the second term can be controlled as 
\begin{equation*}
P_2\le \frac{\textcolor{black}{c}\Delta_n^{r/2}}{(\log(n))^r}\E_{t_{j-1}}\Bigg[\textcolor{black}{\bigg|}\frac{\partial_\theta^k (b^2(\theta,X_{t_{j-1}}^{i}))}{{\modch \sigma}^2(X_{t_{j-1}}^{i})}\textcolor{black}{\bigg|^{r}}\Bigg]\le c \left(\frac{\sqrt{\Delta_n}}{\log(n)}\right)^r,
\end{equation*}
for all $r\ge 1$. In the first term, replacing the increments of $X$ by its dynamics, we get
\begin{align*}
P_1
&\le \mathbb{P}_{t_{j-1}}\Big(\frac{2\big| \partial_\theta^k b(\theta,X_{t_{j-1}}^{i})\big|}{{\modch \sigma}^{2}(X_{t_{j-1}}^{i})}\Big|\int_{t_{j-1}}^{t_j}b(\theta,X_{t_{j-1}}^{i})ds\Big|>\frac{\sqrt{\Delta_n}\log(n)}{6}\Big)\\
&+\mathbb{P}_{t_{j-1}}\bigg(\frac{2\big| \partial_\theta^k b(\theta,X_{t_{j-1}}^{i})\big|}{{\modch \sigma}^{2}(X_{t_{j-1}}^{i})}\Big|\int_{t_{j-1}}^{t_j}({\modch \sigma}(X_{s}^{i})-{\modch \sigma}(X_{t_{j-1}}^{i}))dW_s^i\Big|>\frac{\sqrt{\Delta_n}\log(n)}{6}\bigg)\\
&+\mathbb{P}_{t_{j-1}}\bigg(\frac{2\big| \partial_\theta^k b(\theta,X_{t_{j-1}}^{i})\big|}{{\modch \sigma}(X_{t_{j-1}}^{i})}\Big|W_{t_j}^{i}-W_{t_{j-1}}^{i}\Big|>\frac{\sqrt{\Delta_n}\log(n)}{6}\bigg)\\
&=:P_{11}+P_{12}+P_{13}.
\end{align*}
From Markov inequality, we get for any $r\ge 1$,
\begin{align*}
P_{11}\le \frac{\Delta_n^{r/2}}{(\log(n))^r}\E_{t_{j-1}}\Bigg[\bigg|\frac{2 \partial_\theta^k b(\theta,X_{t_{j-1}}^{i})}{{\modch \sigma}^{2}(X_{t_{j-1}}^{i})}\int_{t_{j-1}}^{t_j}b(\theta,X_{t_{j-1}}^{i})ds\bigg|^r\Bigg]\le c \left(\frac{\sqrt{\Delta_n}}{\log(n)}\right)^r,
\end{align*}
where we have used that $b$ and $\partial_\theta^k b$ are bounded and ${\modch \sigma}$ is lower bounded as given in Assumptions \ref{Ass_2} and \ref{Ass_3}.\\
%
%
%
\textcolor{black}{Successively applying Markov and Burkholder-Davis-Gundy inequalities, before using the Lipschitzness of ${\modch \sigma}$, we get {\modch for any $r \ge 2$} 
\begin{align*}
P_{12}
&\le \frac{\textcolor{black}{c}}{\big(\sqrt{\Delta_n}\log(n)\big)^r}\E_{t_{j-1}}\Bigg[\bigg|\frac{2 \partial_\theta^k b(\theta,X_{t_{j-1}}^{i})}{{\modch \sigma}^{2}(X_{t_{j-1}}^{i})}\bigg|^r\bigg|\int_{t_{j-1}}^{t_j}({\modch \sigma}(X_{s}^{i})-{\modch \sigma}(X_{t_{j-1}}^{i}))dW_s^{i}\bigg|^{r}\Bigg]\\
&\le \frac{c}{\big(\sqrt{\Delta_n}\log(n)\big)^r}\E_{t_{j-1}}\Bigg[\bigg|\int_{t_{j-1}}^{t_j}|{\modch \sigma}(X_{s}^{i})-{\modch \sigma}(X_{t_{j-1}}^{i})|^2ds\bigg|^{r/2}\Bigg]\\
&\le \frac{c}{\big(\sqrt{\Delta_n}\log(n)\big)^r}\E_{t_{j-1}}\Bigg[\bigg|\int_{t_{j-1}}^{t_j}|X_{s}^{i}-X_{t_{j-1}}^{i}|^2ds\bigg|^{r/2}\Bigg]\\
&\le \frac{c\Delta_n^{r/2-1}}{\big(\sqrt{\Delta_n}\log(n)\big)^r}\E_{t_{j-1}}\bigg[\int_{t_{j-1}}^{t_j}|X_{s}^{i}-X_{t_{j-1}}^{i}|^rds\bigg]\\
&\le \frac{c\Delta_n^{r/2-1}\Delta_n^{r/2+1}}{\big(\sqrt{\Delta_n}\log(n)\big)^r}R_{t_{j-1}}(1)=\frac{c\Delta_n^{r}}{\big(\sqrt{\Delta_n}\log(n)\big)^r}R_{t_{j-1}}(1),
\end{align*}}
where we have also used Jensen's inequality in the fourth line and Point 3 of Lemma \ref{l: moment} to get the last one.
Last,
\begin{align*}
P_{13}
&=\mathbb{P}_{t_{j-1}}\Bigg(\Big|W_{t_j}^{i}-W_{t_{j-1}}^{i}\Big|>\frac{\sqrt{\Delta_n}\log(n)}{12}\frac{|{\modch \sigma}(X_{t_{j-1}}^{i})|}{\sqrt{|\partial_\theta^k b^2(\theta,X_{t_{j-1}}^{i})|}}\Bigg)\\
&=\mathbb{P}\Bigg(|Y|>\frac{\log(n)}{12}\frac{|{\modch \sigma}(x)|}{\sqrt{|\partial_\theta^k b^2(\theta,x)|}}\Bigg)\bigg|_{x=X_{t_{j-1}}^{i}}=2\mathbb{P}\Bigg(Y>\frac{\log(n)}{12}\frac{|{\modch \sigma}(x)|}{\sqrt{|\partial_\theta^k b^2(\theta,x)|}}\Bigg)\bigg|_{x=X_{t_{j-1}}^{i}}
\end{align*}
where $Y$ is a standard Gaussian random variable. Recalling that for all $z>0$,
\begin{equation*}
e^{z^2/2}\P(Y>z)=\frac{1}{\sqrt{2\pi}}\int_z^{+\infty}e^{-(y^2-z^2)/2}dy\le \frac{1}{\sqrt{2\pi}}\int_0^{+\infty}e^{-(y-z)^2/2}dy=\frac{1}{2},
\end{equation*}
we get
\begin{equation*}
P_{13}\le \exp\Bigg(-\frac{1}{2}\bigg(\frac{\log(n)}{12}\frac{|{\modch \sigma}(X_{t_{j-1}}^{i})|}{\sqrt{|\partial_\theta^k b^2(\theta,{\modch X_{t_{j-1}^i}})|}}\bigg)^2\Bigg)
\end{equation*}
Combining all previous bounds, there exist constants $c_1$ and $c_2$ such that
\begin{equation*}
\mathbb{P}_{t_{j-1}}(|f(\theta;X_{t_{j-1}}^{i},X_{t_j}^{i})| > \tau_n)
\le  \frac{\Delta_n^{r/2}}{(\log(n))^r}R_{t_{j-1}}(1)+c_1\exp\bigg(-c_2(\log(n))^2\frac{{\modch \sigma}^{2}(X_{t_{j-1}}^{i})}{ \partial_\theta^k b^2(\theta,X_{t_{j-1}}^{i})}\bigg).
\end{equation*}
Hence, the result.
\end{proof}

\subsection{Proof of Lemma \ref{l: Petrov}}
\begin{proof}
Our lemma is based on Theorem 15, p.52 of \cite{Petrov}. In order to apply it we need a control on the exponential moments of the Laplace. \textcolor{black}{Recalling that the moment generating function of a Laplace random variable with mean $0$ and scale parameter $b$ is $t\mapsto 1/(1-b^2t^2)$ we have
$\E\big[e^{t\cL(1/\gamma_h)}\big] 
=1/\big(1-(t/\gamma_h)^2\big)$.}
Remark that, for $u \le 1/4$ it is $1/(1 - u) < e^{2u}$. Then, for $t \le \gamma_h/2$, we have $1/\big(1-(t/\gamma_h)^2\big) \le \exp\big(2(t/\gamma_h)^2\big)$. From here we conclude
$$\E\Big[e^{t\cL(1/\gamma_h)}\Big] \le e^{2t^2/\gamma_h^2} \, \mbox{ for } t\le \frac{1}{2} \gamma_{\max}.$$
According to Petrov's notation we therefore have $T:= \gamma_{\max}/2$, $g_h := 4 /{\modch \gamma_h}^2$, $G := 4 \sum_{h = 1}^U1/\gamma_h^2$. Then, from Theorem 15, p.52 of \cite{Petrov} it follows
\begin{displaymath}
\P\big(S_U \geq x \big) \leq \left\{
\begin{array}{rcl}
   e^{-\frac{x^2}{2 G}}  & \text{if} & 0 \le x \le \frac{\gamma_{\max} G}{2} \\
    e^{-\frac{\gamma_{\max} x}{4}} & \text{if} & x \ge \frac{\gamma_{\max} G}{2}
\end{array}\right.
\end{displaymath}
Hence, by symmetry, 
\begin{align}{\label{eq: Petrov A}}
\P\bigg(\frac{|S_U|}{\sqrt{U}} \geq \lambda \bigg) = \P\big({|S_U|} \geq \lambda \sqrt{U} \big)  \leq \left\{
\begin{array}{rcl}
   2 e^{-\frac{\lambda^2 U}{2 G}}  & \text{if} & 0 \le  \lambda \sqrt{U} \le \frac{\gamma_{\max} G}{2} \\
    \textcolor{black}{2}e^{-\frac{\gamma_{\max}  \lambda \sqrt{U}}{4}} & \text{if} &  \lambda \sqrt{U} \ge \frac{\gamma_{\max} G}{2}.
\end{array}\right.
\end{align}
Observe now that, from the definition of $G$ it is
$$G = 4 U \frac{1}{U} \sum_{h = 1}^U \frac{1}{\gamma_h^2} = 4U \frac{1}{\bar{\gamma}^2}.$$
Then, $ \lambda \sqrt{U} \le \gamma_{\max} G/2$ if and only if $ \lambda \sqrt{U} \le 2U\gamma_{\max}/\bar{\gamma}^2$, i.e. $ \lambda\le 2 \sqrt{U}\gamma_{\max}/\bar{\gamma}^2$. Replacing such observation in \eqref{eq: Petrov A} we obtain 
\begin{displaymath}
 \P\bigg(\frac{|S_U|}{\sqrt{U}} \geq \lambda \bigg) \leq \left\{
\begin{array}{rcl}
   2 e^{-\frac{\lambda^2 \bar{\gamma}^2 }{8}}  & \text{if} & 0 \le  \lambda  \le 2 \sqrt{U} \gamma_{\max}/\bar{\gamma}^2 \\
    \textcolor{black}{2}e^{-\frac{\gamma_{\max}  \lambda \sqrt{U}}{4}} & \text{if} &  \lambda \ge 2 \sqrt{U}  \gamma_{\max}/\bar{\gamma}^2,
\end{array}\right.   
\end{displaymath}
as we wanted.

\end{proof}

\subsection{Proof of equation \eqref{Eq : LDP overline bm Q}}\label{Ss: proof LDP overline bm Q}

{\revar
It is sufficient to prove the formula for $\bm{A}$ given as $\bm{A} = \prod_{j=1}^n A_j$ with $A_j\in \Xi_{\mathcal{Z}_j}$ for all $j\in\{1,\dots,n\}$. Let us denote by $q_j(d\tilde{x}_0,
\dots,d\tilde{x}_n \mid X_{t_j}^i=x_j,X_{t_{j-1}}^i=x_{j-1})$ the distribution of $\bm{X}^i=(X^i_l)_{l=0,\dots,n}$ conditional to $X_{t_j}^i=x_j, X_{t_{j-1}}^i=x_{j-1}$. Define, as a reminder, $\bm{\overline{Q}}(A_1\times\dots\times A_n \mid X^i_{t_j}=x_j, X^i_{t_{j-1}}=x_{j-1}):=\mathbb{P}(Z^i_1\in A_1,\dots,Z^i_n\in A_n  \mid X^i_{t_j}=x_j, X^i_{t_{j-1}}=x_{j-1})$ for $(A_1,\dots,A_n) \in \prod_{l=1}^n \Xi_{\mathcal{Z}_l}$, which gives the law of the whole vector of public data containing information about  $(X^i_{t_j}, X^i_{t_{j-1}})$.
With this notation
\begin{multline*}
\overline{\bm{Q}} (\bm{A}\mid X^i_{t_j}=x_j, X^i_{t_{j-1}}=x_{j-1}) \\ = \int_{\mathbb{R}^{n+1}} \prod_{l=1}^n Q_l(A_l \mid  X_{t_l}^i=\tilde{x}_l,X_{t_{l-1}}^i=\tilde{x}_{l-1} )  q_j(d\tilde{x}_0,
\dots,d\tilde{x}_n \mid X_{t_j}^i=x_j,X_{t_{j-1}}^i=x_{j-1}).
\end{multline*}
From \eqref{eq: def local privacy}, $ Q_l^\star(A_l):=\inf_{(\tilde{x}_l',\tilde{x}_{l-1}')} Q_l(A_l \mid  X_{t_l}^i=\tilde{x}_l',X_{t_{l-1}}^i=\tilde{x}_{l-1}' ) \ge e^{-\alpha_l} Q_l(A_l \mid  X_{t_l}^i=\tilde{x}_l,X_{t_{l-1}}^i=\tilde{x}_{l-1} )$ for any $\tilde{x}_l,\tilde{x}_{l-1}$.
Hence,
\begin{multline}\label{Eq : bm Q x}
\overline{\bm{Q}}(\bm{A}\mid X^i_{t_j}=x_j, X^i_{t_{j-1}}=x_{j-1}) \le \exp\Big(\sum_{l=1}^n \alpha_l\Big) \times \\ \int_{\mathbb{R}^{n+1}} \prod_{l=1}^n Q_l^\star(A_l)  q_j(d\tilde{x}_0,
\dots,d\tilde{x}_n \mid X_{t_j}^i=x_j,X_{t_{j-1}}^i=x_{j-1})  \\=  \exp\Big(\sum_{l=1}^n \alpha_l\Big)  \times \prod_{l=1}^n Q_l^\star(A_l).
\end{multline}
Now, we use that for any $x_j',x_{j-1}'$, we have 
$\int_{\mathbb{R}^{n+1}}q(d\tilde{x}_0,\dots,d\tilde{x}_n \mid 
X_{t_j}^i=x'_j,X_{t_{j-1}}^i=x'_{j-1})=1$, and deduce
\begin{multline}\label{Eq : bm Q x prime}
\prod_{l=1}^n Q_l^\star(A_l)= \int_{\mathbb{R}^{n+1}} \prod_{l=1}^n Q_l^\star(A_l)  q_j(d\tilde{x}_0,
\dots,d\tilde{x}_n \mid X_{t_j}^i=x'_j,X_{t_{j-1}}^i=x'_{j-1})
\\ \le  \int_{\mathbb{R}^{n+1}} \prod_{l=1}^n Q_l(A_l \mid  X_{t_l}^i=\tilde{x}_l,X_{t_{l-1}}^i=\tilde{x}_{l-1} )  q_j(d\tilde{x}_0,
\dots,d\tilde{x}_n \mid X_{t_j}^i=x'_j,X_{t_{j-1}}^i=x'_{j-1})
\\
= \overline{\bm{Q}} (\bm{A}\mid X^i_{t_j}=x'_j, X^i_{t_{j-1}}=x'_{j-1}) 
\end{multline}
where we used the definition of $Q_l^\star(A_l)$ in the second line. 
Joining \eqref{Eq : bm Q x} and \eqref{Eq : bm Q x prime} gives the equation \eqref{Eq : LDP overline bm Q}. \qed
}

{ \rev
\section*{Acknowledgements}
We are grateful to two anonymous referees for truly helpful comments and suggestions.
}
\color{black}

\appendix

\color{black}


\end{document}